\documentclass[letter,11pt]{article}

\usepackage[T1]{fontenc}

\usepackage[margin=1.25 in, top=.87 in, bottom=0.9 in]{geometry}

\makeatletter
\g@addto@macro\normalsize{%
  \setlength\abovedisplayskip{7pt}
  \setlength\belowdisplayskip{7pt}
  \setlength\abovedisplayshortskip{7pt}
  \setlength\belowdisplayshortskip{7pt}
}
\makeatother
\usepackage{tocloft}

\usepackage[english]{babel}
\usepackage[utf8]{inputenc}
\usepackage{amsfonts}
\usepackage{mathrsfs}
\usepackage{bbm}
\usepackage{latexsym}
\usepackage{math dots}
\usepackage{amssymb}
\usepackage{mathtools}

\usepackage{enumitem}
\setlist{nolistsep}

\usepackage{titling}
\posttitle{\par\end{center}\vspace{-.7em}}

\usepackage{titlesec}
\titleformat{\section}
  {\Large\center\bfseries}
  {\thesection.}{.7em}{}
\titlespacing*{\section}{0pt}{3.5ex plus 0ex minus 0ex}{1.5ex plus 0ex}
\titleformat{\subsection}
  {\large\center\bfseries}
  {\thesubsection.}{.7em}{}
\titlespacing*{\subsection}{0pt}{3.5ex plus 0ex minus 0ex}{1.5ex plus 0ex}
\titleformat{\subsubsection}
  {\center\bfseries}
  {\thesubsubsection.}{.7em}{}
\titlespacing*{\subsubsection}{0pt}{3.5ex plus 0ex minus 0ex}{1.5ex plus 0ex}

\addto\captionsenglish{}

\usepackage{chngcntr}

\usepackage[linktocpage=true]{hyperref}
\usepackage{doi}

\usepackage{xcolor}
\definecolor{Color1}{rgb}{0.9, 0.9, 0.99}
\definecolor{Color2}{rgb}{0.78, 0.11, 0.0}
\definecolor{Color3}{rgb}{0.39, 0.71 ,0.0}
\hypersetup{citecolor = black,colorlinks,
			linkcolor = black,
			urlcolor = Color2}
			
\usepackage{amsthm}
\usepackage[capitalize]{cleveref}
\newtheoremstyle{plain}{3mm}{3mm}{\slshape}{}{\bfseries}{.}{.5em}{}
\newtheoremstyle{definition}{2mm}{2mm}{}{}{\bfseries}{.}{.5em}{}
\newtheoremstyle{inproofclaim}{2mm}{2mm}{}{}{\bfseries}{.}{.5em}{}
\theoremstyle{plain}

\newtheorem{Theorem}{Theorem}

\newtheorem{Lemma}[Theorem]{Lemma}
\newtheorem{Proposition}[Theorem]{Proposition}
\newtheorem{Corollary}[Theorem]{Corollary}
\newtheorem{Conjecture}[Theorem]{Conjecture}	
\theoremstyle{definition}
\newtheorem{Definition}[Theorem]{Definition}
\newtheorem{Remark}[Theorem]{Remark}

\theoremstyle{inproofclaim}

\theoremstyle{plain} 
\newcounter{MainTheoremCounter}

\newtheorem{Maintheorem}[MainTheoremCounter]{Theorem}

\theoremstyle{plain}
\newtheorem*{namedthm}{\namedthmname}
\newcounter{namedthm}
\makeatletter
	\newenvironment{named}[2]
	{\def\namedthmname{#1}
	\refstepcounter{namedthm}
	\namedthm[#2]\def\@currentlabel{#1}}
	{\endnamedthm}
\makeatother

\numberwithin{equation}{section}
\counterwithin{Theorem}{section}

\allowdisplaybreaks

\usepackage[alphabetic, nobysame]{amsrefs}
\renewcommand{\eprint}[1]{\href{https://arxiv.org/abs/#1}{arXiv:#1}}
\renewcommand{\MR}[1]{}

\BibSpec{article}{%
    +{}  {\PrintAuthors}                {author}
    +{,} { \textit}                     {title}
    +{.} { }                            {part}
    +{:} { \textit}                     {subtitle}
    +{,} { }                            {journal}
    +{}  { \textbf}                     {volume}
    +{}  { \PrintDatePV}                {date}
    +{,} { \issuetext}                  {number}
    +{,} { pp.~\eprintpages}            {pages}
    +{,} { }                            {status}
    +{,} { \url}                        {url}
    +{,} { \PrintDOI}                   {doi}
    +{,} { available at \eprint}        {eprint}
    +{}  { \parenthesize}               {language}
    +{.} { }                            {note}
    +{.} {}                             {transition}
}

\newcommand{\Cesaro}{Ces\`{a}ro}

\newcommand{\Szemeredi}{Szemer\'{e}di}

\newcommand{\Oh}{{\rm O}}
\newcommand{\oh}{{\rm o}}
\newcommand{\N}{\mathbb{N}}
\newcommand{\Z}{\mathbb{Z}}
\newcommand{\R}{\mathbb{R}}
\newcommand{\C}{\mathbb{C}}

\newcommand{\T}{\mathbb{T}}

\newcommand{\Cont}{\mathsf{C}}

\newcommand{\define}[1]{{\itshape #1}}

\renewcommand{\epsilon}{\varepsilon}
\renewcommand{\leq}{\leqslant}
\renewcommand{\geq}{\geqslant}
\renewcommand{\setminus}{\backslash}

\renewcommand{\phi}{\varphi}

\renewcommand{\d}{~\mathsf{d}}

\newcommand{\plh}{{\mkern-1mu\times\mkern-1.5mu}}
\newcommand{\PLH}{{\mkern-1mu\times\mkern-.5mu}}
\newcommand{\oPLH}{{\mkern1.8mu\otimes\mkern1.6mu}}
\newcommand{\Hardy}{\mathcal{H}}
\newcommand{\LE}{\mathcal{LE}}
\renewcommand{\S}{\mathcal{S}}

\newcommand{\sq}{\square}
\renewcommand{\mark}[1]{\widehat{#1}}

\newcommand{\markfour}[1]{{#1}^*}
\newcommand{\dom}{\mathsf{dom}}

\author{Florian~K.~Richter}
\date{\small \today}
\title{\bfseries Uniform distribution in nilmanifolds along functions from a Hardy field}

\begin{document}

\maketitle
\begin{abstract}
\noindent We study equidistribution properties of translations on nilmanifolds along functions of polynomial growth from a Hardy field. More precisely, if $X=G/\Gamma$ is a nilmanifold, $a_1,\ldots,a_k\in G$ are commuting nilrotations, and $f_1,\ldots,f_k$ are functions of polynomial growth from a Hardy field then we show that
\begin{itemize}
\item the distribution of the sequence $a_1^{f_1(n)}\cdot\ldots\cdot a_k^{f_k(n)}\Gamma$ is governed by its projection onto the maximal factor torus, which extends Leibman's Equidistribution Criterion from polynomials to a much wider range of functions; and
\item the orbit closure of $a_1^{f_1(n)}\cdot\ldots\cdot a_k^{f_k(n)}\Gamma$ is always a finite union of sub-nilmanifolds, which extends previous work of Leibman and Frantzikinakis on this topic.
\end{itemize} 
\end{abstract}

\small
\tableofcontents
\thispagestyle{empty}
\normalsize


\section{Introduction}
\label{sec_intro}

The study of the distribution of orbital sequences in nilsystems 
is an important part of contemporary ergodic theory.
Not only do nilsystems play a key role in the structure theory of measure preserving systems (cf.\ \cite{HK18}), but they are also tightly connected to the theory of higher order Fourier analysis, which finds applications to combinatorics and number theory (cf.\ \cite{Tao12}). 
The purpose of this article is to study the distribution of orbits in nilsystems along functions of polynomial growth from a Hardy field. This has applications to additive combinatorics and leads to new refinements of \Szemeredi{}'s theorem on arithmetic progressions (see \cref{sec_applications}). Our main results in this direction are Theorems \ref{thm_B}, \ref{thm_C}, \ref{thm_D}, and \ref{thm_E} below, which expand on the work of Leibman on the uniform distribution of polynomial sequences in nilmanifolds \cite{Leibman05a}, and the work of Frantzikinakis on the uniform distribution of nil-orbits along functions of different polynomial growth from a Hardy field \cite{Frantzikinakis09}. Our results also connect to various conjectures and problems posed by Frantzikinakis over the years, including \cite[Conjecture on p.~357]{Frantzikinakis09}, \cite[Problems 1 and 4]{Frantzikinakis10}, \cite[Problem 1]{Frantzikinakis15b}, and \cite[Problems 23 and 25]{Frantzikinakis16arXiv}.

A closed subgroup $\Gamma$ of a Lie group $G$ is called \define{uniform} if the quotient $G/\Gamma$ is compact, and it is called \define{discrete} if there exists a cover of $\Gamma$ by open subsets of the ambient group $G$ in which every open set contains exactly one element of $\Gamma$.
Given a ($s$-step) nilpotent Lie group $G$ and a uniform and discrete subgroup $\Gamma$ of $G$, the quotient space $X\coloneqq G/\Gamma$ is called a \define{($s$-step) nilmanifold}.
For any group element $a\in G$ and point $x\in X$, we define the translation of $x$ by $a$ as $ax\coloneqq (ab)\Gamma$, where $b$ is any element in $G$ such that $x=b\Gamma$. 
This way, the Lie group $G$ acts continuously and transitively on the nilmanifold $X$. There exists a unique Borel probability measure on $X$ invariant under this action by $G$, called the \define{Haar measure} on $X$ (see \cite{Raghunathan72}), which we denote by $\mu_X$. A sequence $(x_n)_{n\in\N}$ of points in $X$ is then said to be \define{uniformly distributed} in $X$ if 
\begin{equation*}
\label{eqn_def_ud}
\lim_{N\to\infty}\frac{1}{N}\sum_{n=1}^N F(x_n) = \int F\d\mu_X
\end{equation*}
holds for every continuous function $F\in\Cont(X)$.


When it comes to the study of uniform distribution in nilmanifolds, an important role is played 
by the largest toral factor of the nilmanifold, called the \define{maximal factor torus}.
Let $G^\circ$ denote the identity component of a nilpotent Lie group $G$, 
and let $[G^\circ,G^\circ]$ be the commutator subgroup generated by $G^\circ$. 

\begin{Definition}[Maximal factor torus]
Given a connected nilmanifold $X=G/\Gamma$, the \define{maximal factor torus} of $X$ is the quotient $[G^\circ,G^\circ]\backslash X$. We will use $\vartheta\colon X\to [G^\circ,G^\circ]\backslash X$ to denote the natural factor map from $X$ onto $[G^\circ,G^\circ]\backslash X$.
\end{Definition}

The maximal factor torus is diffeomorphic to a torus $\T^d\coloneqq \R^d/\Z^d$ whose dimension $d$ equals the dimension of the quotient group $G^\circ/[G^\circ,G^\circ]$.
Moreover, as the name suggests, it is the torus of highest dimension that is a factor of $X$. 



If $X=G/\Gamma$ is connected then it is well-known that the distribution of orbits along many sequences is governed by their projection onto $[G^\circ,G^\circ]\backslash X$. A classical result in this direction is Green's Theorem \cite{Green61} (see also \cite{AGH63,Parry69,Parry70,Leibman05a}) which states that a niltranslation acts ergodically on the nilmanifold if and only if it acts ergodically on the maximal factor torus. 
An important generalization of Green's Theorem is Leibman's Equidistribution Criterion for polynomial sequences. Given an element $a$ in a simply connected nilpotent Lie group $G$, we write $\dom(a)$ for the set of all $t\in\R$ for which $a^t$ is a well-defined element of the group.\footnote{For instance, a rational number ${r}/{q}$ with $\gcd(r,q)=1$ belongs to $\dom(a)$ if and only if there exists $b\in G$ such that $b^q=a^r$. Since $G$ is assumed to be simply connected, if such an element $b$ exists then it must be unique. 
Note that $\dom(a)=\R$ if and only if $a\in G^\circ$.}

\begin{Theorem}[Leibman's Equidistribution Criterion, {\cite[Theorem C]{Leibman05a}}]
\label{thm_leib_crit}
\label{thm_leib_crit_non-connected}
Let $G$ be a simply connected nilpotent Lie group and $\Gamma$ a uniform and discrete subgroup of $G$. Suppose
$$
u(n)\, =\,  a_1^{p_1(n)}\cdot\ldots\cdot a_k^{p_k(n)},\qquad\forall n\in\N,
$$
where $a_1,\ldots,a_k\in G$ are commuting and $p_1,\ldots,p_k\in\R[t]$ with $p_i(\N)\subset \dom(a_i)$. Then the
 following are equivalent:
\begin{enumerate}
[label=(\roman{enumi}),ref=(\roman{enumi}),leftmargin=*]
\item
\label{condition_i_leib}
$(u(n)\Gamma)_{n\in\N}$ is uniformly distributed in the nilmanifold $X=G/\Gamma$.
\item\label{condition_ii_leib}
$(\vartheta(u(n)\Gamma))_{n\in\N}$ is uniformly distributed in the maximal factor torus $[G^\circ,G^\circ]\backslash X$.
\end{enumerate}
\end{Theorem}

\begin{Remark}
Leibman actually proves \cref{thm_leib_crit} without the assumption that the elements $a_1,\ldots,a_k\in G$ are commuting.
But for the scope of this paper, we stick to the commuting case because it suffices for the applications to combinatorics that we have in mind.
\end{Remark}

Closely related to \cref{thm_leib_crit} is Leibman's Equidistribution Theorem, which describes the orbit closure of polynomial sequences in a nilmanifold.

\begin{Theorem}[Leibman's Equidistribution Theorem, {\cite[Theorem B]{Leibman05a}}]
\label{thm_leib_orb_cls}
Let $G$ be a simply connected nilpotent Lie group, $\Gamma$ a uniform and discrete subgroup of $G$, and
$$
u(n)\, \coloneqq\,  a_1^{p_1(n)}\cdot\ldots\cdot a_k^{p_k(n)},\qquad\forall n\in\N,
$$
where $a_1,\ldots,a_k\in G$ are commuting and $p_1,\ldots,p_k\in\R[t]$ with $p_i(\N)\subset \dom(a_i)$. Then there exists a closed and connected subgroup $H$ of $G$ and points $x_0,\ldots,x_{q-1}\in X$ such that, for all $r=0,1,\ldots,q-1$, the set $Y_r\coloneqq H x_r$ is a closed sub-nilmanifold of $X$ and the sequence $(u(qn+r)\Gamma)_{n\in\N}$ is uniformly distributed in $Y_r$.
\end{Theorem}

It is natural to ask whether Theorems \ref{thm_leib_crit} and \ref{thm_leib_orb_cls} remain true if one replaces the polynomials $p_1,\ldots,p_k$ with other sufficiently smooth and eventually monotone functions of polynomial growth. 
For instance, we can consider 
the class of \define{logarithmico-exponential functions} introduced by Hardy in \cite{Hardy12,Hardy10}. This class, which we denote by $\LE$, consists of all real-valued functions defined on some half-line $[t_0,\infty)$ that can be build from real polynomials, the logarithmic function $\log(t)$, and the exponential function $\exp(t)$ using the standard arithmetical operations $+$, $-$, $\cdot$, $\div$ and the operation of composition.
Examples of logartihmico-exponential functions are ${p(t)}/{q(t)}$ for $p(t),q(t)\in \R[t]$, $t^c$ for $c\in \R$, ${t}/{\log (t)}$, and $e^{\sqrt{t}}$, as well as any products or linear combinations thereof.

$\LE$ is an example of a so-called \define{Hardy field}. 
Although our main results apply to arbitrary Hardy fields, and are stated as such in this introduction, we delay giving the definition of a Hardy field until \cref{sec_prelims} (see \cref{def_hardy_fields}). Instead, we ask the reader to keep $\LE$ as a representative example in mind, since our results are already new and interesting for this class.

Given two functions $f,g\colon[1,\infty)\to\R$ we will write $f(t)\prec g(t)$ when ${g(t)}/{f(t)}\to\infty$ as $t\to\infty$, and $f(t)\ll g(t)$ when there exist $C>0 $ and $t_0\geq 1$ such that $f(t)\leq C g(t)$ for all $t\geq t_0$.  
We say $f(t)$ has \define{polynomial growth} if it satisfies $|f(t)|\ll t^d$ for some $d\in\N$; in this case the smallest such $d$ is called the \define{degree} of $f$ and denoted by $\deg(f)$. 

In the case of tori, the uniform distribution of functions from a Hardy field has been studied extensively by Boshernitzan \cite{Boshernitzan94}. 
In the case of nilmanifolds, Frantzikinakis obtained the following result.

\begin{Theorem}[{\cite[Theorem 1.3, part (i)]{Frantzikinakis09}}]
\label{thm_frantz_i}
Let $\Hardy$ be a Hardy field, $G$ a connected and simply connected nilpotent Lie group, $\Gamma$ a uniform and discrete subgroup of $G$, and consider the nilmanifold $X^k\coloneqq G^k/\Gamma^k$. Suppose
$$
v(n) \,=\,  \big(a_1^{f_1(n)}\hspace{-.2em},\ldots, \hspace{.1em}a_k^{f_k(n)}\big),\qquad\forall n\in\N,
$$
where $a_1,\ldots,a_k\in G$, $f_1,\ldots,f_k\in\Hardy$ have different growth\footnote{A finite set of functions $\{f_1,\ldots,f_k\}$ is said to have \define{different growth} if, after potentially reordering, one has $f_1(t)\prec \ldots\prec f_k(t)$.}, and for any $f\in\{f_1,\ldots,f_k\}$ there exists $\ell\in\N$ such that $t^{\ell-1}\log(t)\prec f(t) \prec t^\ell$.
Then the sequence $(v(n)\Gamma)_{n\in\N}$ is uniformly distributed in the sub-nilmanifold $\overline{a_1^\R\Gamma}\times\ldots\times\overline{a_k^\R\Gamma}$.
\end{Theorem}

\cref{thm_frantz_i} has significant implications to additive combinatorics, leading to analogues and refinements of \Szemeredi{}'s theorem on arithmetic progressions (see \cite{FW09,Frantzikinakis10,Frantzikinakis15b}).
It was conjectured by Frantzikinakis that the assumption in \cref{thm_frantz_i} that the functions $f_1,\ldots,f_k$ have different growth can be relaxed considerably (cf.\ \cite[Conjecture on p.~357]{Frantzikinakis09}).
This also relates to \cite[Problems 1]{Frantzikinakis10}, \cite[Problem 1]{Frantzikinakis15b}, and \cite[Problem 23]{Frantzikinakis16}.
Our first main result addresses a special case of this conjecture.
It gives a generalization of \cref{thm_leib_crit_non-connected} to all finite collections $f_1,\ldots,f_k$ of functions from a Hardy field $\Hardy$ satisfying \ref{property_P} below. Throughout this work, we use $f'(t)$, $f''(t)$, and $f^{(n)}(t)$ to denote the $1^{\text{st}}$, $2^{\text{nd}}$, and $n^{\text{th}}$ derivative of a function $f(t)$, respectively. 
\vspace{.5em}

\parbox{\dimexpr\linewidth-6.5em}{
\begin{enumerate}
[label=\textbf{Property\,(P):},ref=Property\,(P),leftmargin=*]
\item\label{property_P}
For all $c_1,\ldots ,c_k\in\R$ and $n_1,\ldots,n_k\in\N\cup\{0\}$ the function
$
f(t)\, =\, c_1 f_1^{(n_1)}(t)+\ldots+c_k f_k^{(n_k)}(t)
$
has the property that for every $p\in\R[t]$ either $|f(t)-p(t)|\ll 1$ or $|f(t)-p(t)| \succ \log(t)$.
\end{enumerate}
}
\vspace{.5em}
%

\begin{Maintheorem}
\label{thm_A}
\label{thm_B}
Let $\Hardy$ be a Hardy field, $G$ a simply connected nilpotent Lie group, $\Gamma$ a uniform and discrete subgroup of $G$, and assume $X=G/\Gamma$ is connected.
Suppose
$$
v(n)\,=\, a_1^{f_1(n)}\cdot\ldots\cdot a_k^{f_k(n)},\qquad\forall n\in\N,
$$
where $a_1,\ldots,a_k\in G$ are commuting,  $f_1,\ldots,f_k\in \Hardy$ satisfy \ref{property_P}, and $f_i(\N)\subset \dom(a_i)$.
Then the following are equivalent:
\begin{enumerate}
[label=(\roman{enumi}),ref=(\roman{enumi}),leftmargin=*]
\item
\label{condition_i_thmA}
$(v(n)\Gamma)_{n\in\N}$ is uniformly distributed in the nilmanifold $X=G/\Gamma$.
\item\label{condition_ii_thmA}
$(\vartheta(v(n)\Gamma))_{n\in\N}$ is uniformly distributed in the maximal factor torus $[G^\circ,G^\circ]\backslash X$.
\end{enumerate}
\end{Maintheorem}

Whilst \cref{thm_A} provides us with a convenient criteria for checking whether a sequence of the form $n\mapsto a_1^{f_1(n)}\cdot\ldots\cdot a_k^{f_k(n)}\Gamma$ is uniformly distributed in the entire nilmanifold, it would also be desirable to have an analogue of \cref{thm_leib_orb_cls} describing the orbit closure of such sequences in general. The next result provides exactly that.

\begin{Maintheorem}
\label{thm_C}
Let $G$ be a simply connected nilpotent Lie group, $\Gamma$ a uniform and discrete subgroup of $G$, and $\Hardy$ a Hardy field.
Let
$$
v(n)\,=\, a_1^{f_1(n)}\cdot\ldots\cdot a_k^{f_k(n)},\qquad\forall n\in\N,
$$
where $a_1,\ldots,a_k\in G$ are commuting, $f_1,\ldots,f_k\in \Hardy$ satisfy \ref{property_P}, and $f_i(\N)\subset \dom(a_i)$ for all $i=1,\ldots,k$. Then there exists a closed and connected subgroup $H$ of $G$ and points $x_0,x_1,\ldots,x_{q-1}\in X$ such that, for all $r=0,1,\ldots,q-1$, the set $Y_r\coloneqq H x_r$ is a closed sub-nilmanifold of $X$ and $(v(qn+r)\Gamma)_{n\in\N}$ is uniformly distributed in $Y_r$.
\end{Maintheorem}

Note that if $f_1,\ldots,f_k$ are real polynomials then \ref{property_P} is automatically satisfied, which is why Theorems \ref{thm_A} and \ref{thm_C} imply Theorems \ref{thm_leib_crit_non-connected} and \ref{thm_leib_orb_cls}.
More generally, any finite subset of $\{c_1 t^{r_1}+\ldots+c_m t^{r_m}: m\in\N, c_1,\ldots,c_m,r_1,\ldots,r_m\in\R\}$ satisfies \ref{property_P}. This shows that, in contrast to \cref{thm_frantz_i}, Theorems \ref{thm_A} and \ref{thm_C} apply to collections of functions that don't necessarily have different growth. 
But Theorems \ref{thm_A} and \ref{thm_C} do not imply all cases of \cref{thm_frantz_i}, because not every collection of functions $f_1,\ldots,f_k$ satisfying the hypothesis of \cref{thm_frantz_i} also satisfy  \ref{property_P}. For instance, $f_i(t)=t^i\log(t)$ for $i\in\{1,\ldots,k\}$ is such an example.


We also prove analogues of Theorems \ref{thm_A} and \ref{thm_C} that apply to an arbitrary collection of functions  $f_1,\ldots,f_k\in\Hardy$ of polynomial growth, even when \ref{property_P} is not satisfied.
However, in the absence of \ref{property_P} we need to replace \Cesaro{} averages with other, weaker, methods of summation.

\begin{Definition}
\label{def_W-ud}
Let $W\colon\N\to(0,\infty)$ be a non-decreasing sequence, let $w(n)\coloneqq \Delta W(n)= W(n+1)-W(n)$ be its discrete derivative, and assume $W(n)\to\infty$ as $n\to\infty$ and $w(n)\ll 1 $.
A sequence $(x_n)_{n\in\N}$ of points in a nilmanifold $X=G/\Gamma$ is said to be \define{uniformly distributed  with respect to $W$-averages} in $X$ if 
\begin{equation*}
\lim_{N\to\infty}\frac{1}{W(N)}\sum_{n=1}^N w(n) F(x_n) = \int F\d\mu_X
\end{equation*}
holds for every continuous function $F\in\Cont(X)$.
\end{Definition}

When dealing with functions of ``slow growth'' from a Hardy field, it turns out to be necessary to switch form \Cesaro{} averages to $W$-averages to adequately capture the way in which orbits along such functions distribute. For example, if $f(t)$ satisfies $\log\log(t)\prec f(t)\ll \log(t)$ then the sequence $(f(n)\bmod 1)_{n\in\N}$ is not uniformly distributed with respect to \Cesaro{} averages in the unit interval $[0,1)$, but it is uniformly distributed with respect to \define{logarithmic averages}, i.e., $W$-averages where $W(n)=\log(n)$ for all $n\in\N$. Likewise, if $\log\log\log(t)\prec f(t)\ll \log\log(t)$ then the sequence $(f(n)\bmod 1)_{n\in\N}$ is neither uniformly distributed with respect to \Cesaro{} averages nor uniformly distributed with respect to logarithmic averages, but it is uniformly distributed with respect to \define{double-logarithmic averages}, i.e., $W$-averages where $W(n)=\log\log(n)$ for all $n\in\N$ (cf.\ \cref{thm_bosh_W-averages} below).

Theorems \ref{thm_D} and \ref{thm_E} below are generalizations of Theorems \ref{thm_B} and \ref{thm_C} that apply to all finite collections of functions $f_1,\ldots,f_k$ of polynomial growth from a Hardy field. However, the increased generality comes at the cost of having to change the method of summation from \Cesaro{} averages to $W$-averages.

\begin{Maintheorem}
\label{thm_D}
Let $\Hardy$ be a Hardy field, $G$ a simply connected nilpotent Lie group, $\Gamma$ a uniform and discrete subgroup of $G$, and assume $X=G/\Gamma$ is connected.
Suppose
$$
v(n)\,=\, a_1^{f_1(n)}\cdot\ldots\cdot a_k^{f_k(n)},\qquad\forall n\in\N,
$$
where $a_1,\ldots,a_k\in G$ are commuting and $f_1,\ldots,f_k\in \Hardy$ have polynomial growth and satisfy $f_i(\N)\subset \dom(a_i)$.
Then there exists $W\in \Hardy$ with $1\prec W(t) \ll t$ such that the following are equivalent:
\begin{enumerate}
[label=(\roman{enumi}),ref=(\roman{enumi}),leftmargin=*]
\item
\label{condition_i_thmD}
The sequence $(v(n)\Gamma)_{n\in\N}$ is uniformly distributed with respect to $W$-averages in the nilmanifold $X=G/\Gamma$.
\item\label{condition_ii_thmD}
The sequence $(\vartheta(v(n)\Gamma))_{n\in\N}$ is uniformly distributed with respect to $W$-averages in the maximal factor torus $[G^\circ,G^\circ]\backslash X$.
\end{enumerate}
\end{Maintheorem}

Aside from this change in the averaging scheme, \cref{thm_D} implies \cref{thm_frantz_i} as well as Frantzinakis' aforementioned conjecture \cite[Conjecture on p.~357]{Frantzikinakis09}. 
Moreover, \cref{thm_D} provides the following aesthetic corollary. 

\begin{Corollary}
\label{cor_of_thm_D}
Let $\Hardy$ be a Hardy field, $G$ a simply connected nilpotent Lie group, $\Gamma$ a uniform and discrete subgroup of $G$, and assume $X=G/\Gamma$ is connected.
Suppose
$$
v(n)\,=\, a_1^{f_1(n)}\cdot\ldots\cdot a_k^{f_k(n)},\qquad\forall n\in\N,
$$
where $a_1,\ldots,a_k\in G$ are commuting and $f_1,\ldots,f_k\in \Hardy$ have polynomial growth and satisfy $f_i(\N)\subset \dom(a_i)$.
Then the following are equivalent:
\begin{enumerate}
[label=(\roman{enumi}),ref=(\roman{enumi}),leftmargin=*]
\item
\label{condition_i_thmD_cor}
The sequence $(v(n)\Gamma)_{n\in\N}$ is dense in the nilmanifold $X=G/\Gamma$.
\item\label{condition_ii_thmD_cor}
The sequence $(\vartheta(v(n)\Gamma))_{n\in\N}$ is dense in the maximal factor torus $[G^\circ,G^\circ]\backslash X$.
\end{enumerate}
\end{Corollary}

Finally, let us state our last main result.

\begin{Maintheorem}
\label{thm_E}
Let $G$ be a simply connected nilpotent Lie group, $\Gamma$ a uniform and discrete subgroup of $G$, and $\Hardy$ a Hardy field.
Suppose
$$
v(n)\,=\, a_1^{f_1(n)}\cdot\ldots\cdot a_k^{f_k(n)},\qquad\forall n\in\N,
$$
where $a_1,\ldots,a_k\in G$ are commuting and $f_1,\ldots,f_k\in \Hardy$ have polynomial growth and satisfy $f_i(\N)\subset \dom(a_i)$.
Then there exists $W\in \Hardy$ with $1\prec W(t) \ll t$, a closed and connected subgroup $H$ of $G$, and points $x_0,x_1,\ldots,x_{q-1}\in X$ such that $Y_r\coloneqq H x_r$ is a closed sub-nilmanifold of $X$ and $(v(qn+r)\Gamma)_{n\in\N}$ is uniformly distributed with respect to $W$-averages in $Y_r$ for all $r=0,1,\ldots,q-1$.
\end{Maintheorem}

The connection between $W$ and $f_1,\ldots,f_k$ in Theorems \ref{thm_D} and \ref{thm_E} is given by a variant of \ref{property_P}:
\vspace{.5em}

\parbox{\dimexpr\linewidth-4.7em}{
\begin{enumerate}
[label=\textbf{Property\,(P$_W$):},ref=Property\,(P$_W$),leftmargin=*]
\item\label{property_P_W}
For all $c_1,\ldots ,c_k\in\R$ and $n_1,\ldots,n_k\in\N\cup\{0\}$ the function
$
f(t)\, =\, c_1 f_1^{(n_1)}(t)+\ldots+c_k f_k^{(n_k)}(t)
$
has the property that for every $p\in\R[t]$ either $|f(t)-p(t)|\ll 1$ or $|f(t)-p(t)| \succ \log(W(t))$.
\end{enumerate}
}
\vspace{.5em}

\noindent As we will show below (see \cref{cor_finding_W}), for an arbitrary finite collection of functions from a Hardy field $f_1,\ldots,f_k$ of polynomial growth there exists some $W$ with $1\prec W(t)\ll t$ such that \ref{property_P_W} holds. 
Moreover, it will be clear from the proofs of Theorems \ref{thm_D} and \ref{thm_E} that if $f_1,\ldots,f_k$ satisfy \ref{property_P_W} for some $W$ then we can take this $W$ to be the same as the one appearing in the statements of Theorems \ref{thm_D} and \ref{thm_E}.

\subsection{Applications to Combinatorics}
\label{sec_applications}

The motivation for obtaining Theorems \ref{thm_A}, \ref{thm_C}, \ref{thm_D}, and \ref{thm_E} is their connection to additive combinatorics. Indeed, these results play a crucial role in a forthcoming paper \cite{BMR20draft}, where a far-reaching generalization of \Szemeredi{}'s theorem on arithmetic progressions is explored.
To motivate our combinatorial results in this direction, let us first recall the statement of \Szemeredi{}'s Theorem.
The \define{upper density} of a set $E\subset\N$ is defined as $\overline{d}(E)=\limsup_{N\to\infty}|E\cap\{1,\ldots,N\}|/N$.


\begin{Theorem}[\Szemeredi{}'s Theorem]{}
\label{thm_szmeredi}
For any set $E\subset \N$ of positive upper density and any $k\in\N$ there exist $a,n\in\N$ such that $\{a,a+n,\ldots,a+(k-1)n\}\subset E$.
\end{Theorem}

\Szemeredi{}'s Theorem has been generalized numerous times and in many different directions. One of the most noteworhty extensions is due to Bergelson and Leibman in \cite{BL96}, where a polynomial version was obtained. The following theorem pertains to the one-dimensional case of their result.

\begin{Theorem}[Polynomial \Szemeredi{} Theorem]{}
\label{thm_poly_szmeredi}
For any set $E\subset \N$ of positive upper density and any polynomials $p_1,\ldots,p_k\in\Z[t]$ satisfying $p_1(0)=\ldots=p_k(0)=0$ there exist $a,n\in\N$ such that $\{a,\, a+p_1(n),\ldots,a+p_k(n)\}\subset E$.
\end{Theorem}

\cref{thm_poly_szmeredi} was later improved in \cite{BLL08} to include an ``if and only if'' condition.

\begin{Theorem}{}
\label{thm_poly_szmeredi_iff}
Given $p_1,\ldots,p_k\in\Z[t]$ the following are equivalent:
\begin{enumerate}
[label=(\roman{enumi}),ref=(\roman{enumi}),leftmargin=*]
\item
The polynomials $p_1,\ldots,p_k$ are \define{jointly intersective}, i.e., for any $m\in\N$ there exists $n\in\N$ such that $p_i(n)\equiv {0}\bmod{m}$ for all $i\in\{1,\ldots,k\}$.
\item
For any set $E\subset \N$ of positive upper density there exist $a,n\in\N$ such that $\{a,\, a+p_1(n),\ldots,a+p_k(n)\}\subset E$.
\end{enumerate}
\end{Theorem}

Via the Host-Kra structure theory (see  \cite{HK05a, HK18}), \Szemeredi{}'s Theorem and its generalizations are intimately connected to questions about uniform distribution in nilmanifolds.
This connection played an important role in the proof of \cref{thm_poly_szmeredi_iff} in \cite{BLL08}, but was also used by Frantzikinakis in \cite{Frantzikinakis15b} (see also \cite{FW09,Frantzikinakis10}) to derive from \cref{thm_frantz_i} the following combinatorial theorem. Let $\lfloor .\rfloor\colon\R\to\Z$ denote the \define{floor function}.

\begin{Theorem}
\label{thm_frantzi}
Let $f_1,\ldots,f_k$ be functions from a {Hardy field} of different growth and with the property that for every $f\in\{f_1,\ldots,f_k\}$ there is $\ell\in\N$ such that $t^{\ell-1}\log t\prec f(t)\prec t^\ell$. Then for any set $E\subset \N$ of positive upper density there exist $a,n\in\N$ such that $\{a,\, a+\lfloor f_1(n)\rfloor ,\ldots,a+\lfloor f_k(n)\rfloor \}\subset E$.
\end{Theorem}

Given a finite collection of functions $f_1,\ldots,f_k$ we denote by $\mathrm{span}_\R(f_1,\ldots,f_k)$ the set of all functions of the form $f(t)=c_1 f_1(t)+\ldots+c_k f_k(t)$ for $(c_1,\ldots,c_k)\in\R^k$, and by $\mathrm{span}_\R^*(f_1,\ldots,f_k)$ the set of all functions of the form $f(t)=c_1 f_1(t)+\ldots+c_k f_k(t)$ for $(c_1,\ldots,c_k)\in\R^k\setminus\{0\}$.
The following open conjecture is an extension of \cref{thm_frantzi} and was posed by Frantzikinakis on multiple occasions.

\begin{Conjecture}[see {\cite[Problems 4 and 4']{Frantzikinakis10}} and {\cite[Problem 25]{Frantzikinakis16}}]
\label{conj_frantzi}
Let $f_1,\ldots,f_k$ be functions of polynomial growth from a Hardy field such that
$$
|f(t)-p(t)|\to\infty
$$
for every $p\in\Z[t]$, and every $f\in\mathrm{span}_\R^*(f_1,\ldots,f_k)$. Then for any set $E\subset \N$ of positive upper density there exist $a,n\in\N$ such that $\{a,a+\lfloor f_1(n)\rfloor,\ldots,a+\lfloor f_k(n)\rfloor \}\subset E$.
\end{Conjecture}

Finally, another variant of \Szemeredi{}'s Theorem, which also involves functions from a Hardy field, 
was obtained in \cite{BMR17arXiv}.

\begin{Theorem}
\label{thm_thick_szemeredi}
Let $f$ be a function from a \define{Hardy field} and assume there is $\ell\in\N$ such that $t^{\ell-1}\prec f(t)\prec t^\ell$. Then for  any set $E\subset \N$ of positive upper density there exist $a,n\in\N$ such that $\{a,\, a+\lfloor f(n)\rfloor,a+\lfloor f(n+1)\rfloor,\ldots,a+\lfloor f(n+k)\rfloor \}\subset E$.
\end{Theorem}

Similar to the proofs of \cref{thm_poly_szmeredi_iff} and \cref{thm_frantzi}, the proof of \cref{thm_thick_szemeredi} also hinges on uniform distribution results in nilmanifolds. 

With the help of Theorems \ref{thm_D} and \ref{thm_E}, we prove in \cite{BMR20draft} a theorem which not only unifies Theorems~\ref{thm_szmeredi},~\ref{thm_poly_szmeredi},~\ref{thm_poly_szmeredi_iff},~\ref{thm_frantzi},~\ref{thm_thick_szemeredi}, but also confirms \cref{conj_frantzi}.

\begin{Maintheorem}[\cite{BMR20draft}]
\label{thm_main_combinatorial}
Let $[.]\colon\R\to\Z$ be the rounding to the closest integer function, let $f_1,\ldots,f_k$ be functions from a Hardy field with polynomial growth, and 
assume at least one of the following two conditions holds:
\begin{enumerate}
[label=(\arabic{enumi}),ref=(\arabic{enumi}),leftmargin=*]
\item
\label{itm_cond_2}
For all $q\in \Z[t]$ and $f\in\mathrm{span}_\R^*(f_1,\ldots,f_k)$ we have $\lim_{t\to\infty}|f(t)-q(t)|= \infty$.
\item
\label{itm_cond_1}
There is jointly intersective collection of polynomials $p_1,\ldots,p_\ell\in\Z[t]$ such that any real polynomial ``appearing'' in $\mathrm{span}_{\R}(f_1,\ldots,f_k)$ also appears in $\mathrm{span}_{\R}(p_1,\ldots,p_\ell)$, where we say a polynomial $p\in\R[t]$ ``appears'' in $\mathrm{span}_{\R}(f_1,\ldots,f_k)$ if there is $f\in\mathrm{span}_{\R}(f_1,\ldots,f_k)$ such that $\lim_{t\to\infty}|f(t)-p(t)|= 0$.
\end{enumerate}
\vspace{-.2em}
Then for any set $E\subset \N$ of positive upper density there exist $a,n\in\N$ such that $\{a,\, a+[f_1(n)],\ldots,a+[f_k(n)]\}\subset E$.
\end{Maintheorem}

%
%
%

\paragraph{\textbf{Acknowledgements}.} The author thanks Vitaly Bergelson, Nikos Frantzikinakis, and Joel Moreira, and the anonymous referee for providing useful comments. 
The author is supported by the National Science Foundation under grant number DMS~1901453.

\section{Preliminaries}
\label{sec_prelims}


In the proofs of our main theorems we utilize numerous well-known facts and results regarding Hardy fields, nilpotent Lie groups, and nilmanifolds. For convenience, we collect them here in this preparatory section. 

\subsection{Preliminaries on Hardy fields}
\label{sec_prelims_hardy}

A \define{germ
at $\infty$} is any equivalence class of real-valued functions in one real variable under the equivalence relationship
$(f\sim g) \Leftrightarrow \big(\exists t_0>0
~\text{such that}~f(t)=g(t)~\text{for all}~t\in [t_0,\infty)\big)$.
Let $\mathsf{B}$ denote the set of all \define{germs} at $\infty$ of real valued functions defined
on some half-line $[s,\infty)$ for some $s\in\R$.
Note that $\mathsf{B}$ forms a ring under
pointwise addition and multiplication, which we denote by
$(\mathsf{B},+,\cdot)$.
\begin{Definition}[see {\cite[Definition 1.2]{Boshernitzan94}}]
\label{def_hardy_fields}
Any subfield of the ring $(\mathsf{B},+,\cdot)$ that is closed under
differentiation is called a
\define{Hardy field}.
\end{Definition} 
By abuse of language, we say that a function $f\colon[s,\infty)\to\R$ belongs to some Hardy field $\Hardy$ (and write $f\in\Hardy$)
if its germ at $\infty$ belongs to $\Hardy$.

Functions from a Hardy field have a number of convenient properties. For instance, it was shown in \cite[Proposition 2.1]{Boshernitzan81} that for any function $f$ belonging to a Hardy field $\Hardy$ and any $c\in \R$ the function $f(t)-c$ is either eventually non-negative or eventually non-positive. Combined with the fact that $\Hardy$ is a field and closed under differentiation, this implies that: 
\begin{itemize}
\item
the limit $\lim_{t\to\infty}f(t)$ always exists as an element in $\R\cup\{-\infty,\infty\}$;
\item
$f$ is either eventually increasing, eventually decreasing, or eventually constant;
\item
for any $f,g\in\Hardy$ either $f(t)\prec g(t)$, or $g(t)\prec f(t)$, or $\lim_{t\to\infty}f(t)/g(t)$ is a non-zero real number.
\end{itemize}
The following lemma was proved in \cite[Subsection 2.1]{Frantzikinakis09} using L'H{\^o}pital's rule.
\begin{Lemma}
\label{lem_useful_hardy}
Let $\Hardy$ be a Hardy field.
\begin{enumerate}
\item
If $f\in\Hardy$ satisfies $t^{-k}\prec f(t)\prec t^k$ for some $k\in\N$ and $f(t)$ is not asymptotically equal to a constant then
$$
\frac{f(t)}{t \log^2(t)}\prec f'(t)\ll \frac{f(t)}{t}.
$$
\item
If $f\in\Hardy$ satisfies $t^{1/k}\prec f(t)\prec t^k$ for some $k\in\N$ then $\lim_{t\to\infty}tf(t)/f'(t)$ is a non-zero constant.
\end{enumerate}
\end{Lemma}
For more information on Hardy fields we refer the reader to \cite{Boshernitzan81, Boshernitzan82, Boshernitzan84a, Boshernitzan84b, Boshernitzan94, Frantzikinakis09}.

\subsection{Preliminaries on nilpotent Lie groups}
\label{sec_prelims_nilpotent_Lie_groups}

Let $G$ be a $s$-step nilpotent Lie group with identity element $1_G$.
The \define{lower central series} of $G$, which we denote by $C_\bullet \coloneqq  \{C_1,C_2,\ldots, C_{s},C_{s+1}\}$, is a decreasing nested sequence of normal subgroups,
\[
G=C_1 \trianglerighteq C_2 \trianglerighteq \ldots \trianglerighteq C_s \trianglerighteq C_{s+1}=\{1_G\},
\]
where $C_{i+1}:=[C_i,G]$ is the subgroup of $G$ generated by all the commutators $aba^{-1}b^{-1}$ with $a\in C_i$ and $b\in G$. 
Note that $C_{s+1}=\{1_G\}$ because $G$ is $s$-step nilpotent.
Also, each $C_i$ is a closed subgroup of $G$ (cf.\ \cite[Section 2.11]{Leibman05a}).

The \define{upper central series} of $G$, denoted by $Z_\bullet \coloneqq  \{Z_0,Z_1,Z_2,\ldots, Z_{s}\}$, is an increasing nested sequence of normal subgroups,
$$
\{1_G\}= Z_0 \trianglelefteq Z_1\trianglelefteq\ldots \trianglelefteq Z_{s-1}\trianglelefteq Z_s=G,
$$
where the $Z_0,Z_1,\ldots, Z_s$ are defined inductively by $Z_0=\{1_G\}$ and $Z_{i+1}=\{a\in G:  [a,b]\in Z_i \text{ for all }b\in G\}$. Note that $Z_1$ is equal to the \define{center} $Z(G)$ of $G$ and $Z_s=G$ because $G$ is $s$-step nilpotent.

Given a uniform and discrete subgroup $\Gamma$ of a nilpotent Lie group $G$, an element $g\in G$ with the property that $g^n\in \Gamma$
for some $n\in \N$ is called \define{rational} (or \define{rational with respect to $\Gamma$)}.
A closed subgroup $H$ of $G$ is then called \define{rational}
(or \define{rational with respect to $\Gamma$}) if rational elements are dense in $H$.
For example, the subgroups $C_1,\ldots, C_{s}, C_{s+1}$ in the lower central series of $G$, as well as $Z_0,Z_1,\ldots,Z_s$ in the upper central series of $G$, are rational with respect to any uniform and discrete subgroup
$\Gamma$ of $G$ (cf.\ \cite[Corollary 1 of Theorem 2.1]{Raghunathan72} for a proof of this fact for connected $G$
and \cite[Section 2.11]{Leibman05a} for the general case).

Rational subgroups play a key role in the description of sub-nilmanifolds.
If $X=G/\Gamma$ is a nilmanifold, then a \define{sub-nilmanifold} $Y$ of $X$ is any
closed set of the form $Y=Hx$, where $x\in X$ and $H$ is a closed
subgroup of $G$.
It is not true that for every closed subgroup $H$ of $G$ and
every element $x=g\Gamma$ in $X=G/\Gamma$ the set $Hx$ is a sub-nilmanifold of $X$, because $Hx$ need not be closed. In fact, it is shown in \cite{Leibman06} that $Hx$ is closed in $X$ (and hence a sub-nilmanifold) if and only if the subgroup $g^{-1}Hg$ is rational with respect to $\Gamma$.

For more information on rational elements and rational subgroups see \cite{Leibman06}.

\subsection{Preliminaries on the center and central characters}
\label{sec_prelims-center}

Throughout the paper we use $Z(G)$ to denote the center of a group $G$. 

\begin{Lemma}
\label{lem_normal-center-intersection}
Let $G$ be a nilpotent group and $L$ a non-trivial normal subgroup of $G$. Then $L\cap Z(G)\neq \{1_G\}$.
\end{Lemma}

\begin{proof}

We will make use of the upper central series
$
\{1_G\}= Z_0 \trianglelefteq Z_1\trianglelefteq\ldots \trianglelefteq Z_s=G,
$
which was defined in the previous subsection.
Consider the intersection $L_j\coloneqq L\cap Z_{j}$ for $j=0,1,\ldots,s$. Note that $L_0=\{1_G\}$ and $L_s=L$.
Let $J=\{1\leq j\leq s: L_j\neq\{1_G\} \}$ and note that $J$ is non-empty because $L_s\neq\{1_G\}$.
Let $j_{\min{}}$ denote the minimum of $J$. By definition, we have $[Z_{j_{\min{}}},G]\subset Z_{j_{\min{}}-1}$, which implies $[L_{j_{\min{}}},G]\subset Z_{j_{\min{}}-1}$. Moreover, $L_{j_{\min{}}}$ is a normal subgroup of $G$, because it is the intersection of two normal subgroups of $G$, and hence
$[L_{j_{\min{}}},G]\subset L_{j_{\min{}}}$.
We conclude that
\[
[L_{j_{\min{}}},G]\subset L_{j_{\min{}}}\cap Z_{j_{\min{}}-1}=L_{j_{\min{}}-1}.
\]
In view of the minimality assumption on $j_{\min{}}$ we have $L_{j_{\min{}}-1}=\{1_G\}$, from which it follows that $[L_{j_{\min{}}},G]=\{1_G\}$. This proves that $L_{j_{\min{}}}$ is a subset of $Z(G)$. Thus $L\cap Z(G)= L_1=L_{j_{\min{}}}\neq \{1_G\}$ as desired.
%
%
%
\end{proof}

\begin{Corollary}
\label{cor_normal-center-intersection}
Let $G$ be a simply-connected nilpotent Lie group and $L$ a non-trivial, connected, and normal subgroup of $G$. Then $L\cap Z(G)^\circ\neq \{1_G\}$.
\end{Corollary}

\begin{proof}
According to \cref{lem_normal-center-intersection} there exists an element $a\neq 1_G$ in the intersection $L\cap Z(G)$. Since $L$ is connected and $a\in L$, the element  $a$ belongs to the identity component $G^\circ$ of $G$. Moreover, since $G$ is simply-connected, the $1$-parameter subgroup $a^\R=\{a^t: t\in\R\}$ is well defined. It follows from \cite[Lemma 3]{Malcev49} (and, alternatively, also from the Baker-Campbell-Hausdorff formula) that if $a$ commutes with an element $b\in G$ then the entire $1$-parameter subgroup $a^\R$ commutes with $b$. This implies that if $a$ belongs to the center of $G$, then so does $a^\R$. In particular, $a^\R\subset Z(G)$, which proves $a\in Z(G)^\circ$ and hence $L\cap Z(G)^\circ\neq \{1_G\}$.
\end{proof}

\begin{Definition}[Central characters]
\label{def_central_character}
Let $G$ be a nilpotent Lie group and $\Gamma$ a uniform and discrete subgroup of $G$. 
A \define{central character of $(G,\Gamma)$} is any continuous map $\phi\colon X\to\C$ with the property that 
there exists a continuous group homomorphism $\chi\colon Z(G)\to\{z\in\C: |z|=1\}$ such that
\begin{equation}
\label{eqn_central_charcter_functional_equation}
\phi(s x)=\chi(s)\phi(x), \qquad\forall s\in Z(G),~\forall x\in X.
\end{equation}
\end{Definition}

\begin{Remark}
\label{rem_central_characters_are_uniformly_dense}
Since the set of all central characters is closed under conjugation and separates points\footnote{\label{ftn_separate_points}We claim that for any two distinct points $x,y\in X=G/\Gamma$ there exists a central character $\phi$ such that $\phi(x)\neq \phi(y)$. To verify this claim, we distinguish two cases, the case $y\notin \{sx: s\in Z(G)\}$ and the case $y\in \{sx: s\in Z(G)\}$.
If we are in the first case then $Z(G)x\neq Z(G)y$. This implies there exists a continuous function $\phi'$ on the quotient space $Z(G)\backslash X$ with $\phi'(Z(G)x)\neq \phi'(Z(G)y)$, which we can lift to a continuous and $Z(G)$-invariant function $\phi$ on $X$ satisfying $\phi(x)\neq \phi(y)$.
If we are in the second case then there is $s_0\in Z(G)$ such that $y=s_0x$. Note that $s_0$ cannot be an element of $\Gamma$ because $x\neq y$. Let $\chi_0$ be any group character of $Z(G)$ with the property that $Z(G)\cap \Gamma\subset\ker\chi$ and $\chi_0(s_0)\neq 1$. Define $T\coloneqq Z(G)/(\Gamma\cap Z(G))$ and observe that $T$ is a compact abelian group. Let $U$ be a small neighborhood of $x$, and let $\rho\colon X\to [0,1]$ be a continuous function with the property that $\rho(x)=1$ and $\rho(z)=0$ for all $z\notin U$. Now define $\phi(z)\coloneqq \int \rho(sz)\chi_0(s) \d\mu_{T}(s)$, where $\mu_T$ is the normalized Haar measure on $T$. It is straightforward to check that $\phi$ is a non-zero and continuous function on $X$ satisfying $\phi(s z)=\chi_0(s)\phi(z)$ for all $s\in Z(G)$ and $z\in X$. In particular $\phi(y)=\phi(s_0x)=\chi_0(s_0)\phi(x)$ and so $\phi(y)\neq\phi(x)$.} in $X$, it follows from the Stone-Weierstrass Theorem that their linear span is uniformly dense in $\Cont(X)$. We will make use of this fact multiple times in the upcoming sections. 
\end{Remark}

\begin{Remark}
\label{rem_zero_mean_central_char}
Let $\phi$ be a central character of $(G,\Gamma)$ and let  $\chi\colon Z(G)\to\{z\in\C: |z|=1\}$ be the corresponding continuous group homomorphism such that \eqref{eqn_central_charcter_functional_equation} is satisfied.
We claim that if $\chi$ is non-trivial (meaning that there exists $s\in Z(G)$ such that $\chi(s)\neq 1$) then the integral $\int \phi \d\mu_X$ equals $0$. To verify this claim, note that the measure $\mu_X$ is invariant under left-multiplication by $s$. Therefore,
$$
\int \phi(x)\d\mu_X(x) \,  =\,
\int \phi(sx)\d\mu_X(x)
\,=\, 
\chi(s) \int \phi(x)\d\mu_X(x),
$$
which can only hold if $\int \phi \d\mu_X=0$.
\end{Remark}

\subsection{Preliminaries on relatively independent self-products}
\label{sec_prelims-relative-independent-joinings}

One of the key ideas featured in the proofs of Theorems \ref{thm_A}, \ref{thm_C}, \ref{thm_D}, and  \ref{thm_E} is the utilization of a special type of ``relative product group''. Similar product groups played an important role in the inductive procedure employed by Green and Tao in \cite{GT12a}.

\begin{Definition}[Relatively independent product]
Let $G$ be a group and $L$ a normal subgroup of $G$.
We define the \define{relatively independent self-product of $G$ over $L$} as
$$
G\plh_L G\coloneqq \{(a_1,a_2)\in G\PLH G: a_1a_2^{-1}\in L\}. 
$$ 
\end{Definition}

If $G$ is a nilpotent Lie group, $\Gamma$ a uniform and discrete subgroup of $G$, and $L$ a normal and rational subgroup of $G$ then the group
$$
\Gamma\plh_L\Gamma\coloneqq \{(\gamma_1,\gamma_2)\in \Gamma\PLH \Gamma : \gamma_1\gamma_2^{-1}\in L\}=(\Gamma\PLH\Gamma)\cap (G\plh_L G)
$$
is a uniform and discrete subgroup of $G\plh_L G$.
This gives rise to the nilmanifold
$$
X\plh_L X\coloneqq (G\plh_L G)/(\Gamma\plh_L\Gamma),
$$
which we call the \define{relatively independent self-product of $X$ over $L$}.

\begin{Remark}
In ergodic theory, the notion of a relatively independent self-joining of a system $X$ over one of its factors $Y$ is an important notion and finds many applications (see \cite[Definition 6.15]{EW11} for the definition). If $X=G/\Gamma$ is a nilmanifold and $L$ is a normal and rational subgroup of $G$ then the quotient space $Y=L\backslash X$ is a factor of $X$. It turns out that that the relatively independent self-product $X\plh_L X$ is exactly the same as the relatively independent self-joining of $X$ over the factor $Y$.
\end{Remark}

Next, let us state and prove a few results regarding relatively independent self-products that will be useful in the later sections.

\begin{Lemma}
\label{lem_commutator-of-Gx_LG}
We have $\big[ G\plh_L G, G\plh_L G\big]=[G,G]\plh_{[G,L]}[G,G]$.
\end{Lemma}
\begin{proof}
Note that $[G,L]$ is a normal subgroup of $G$ {and a subset of $L$, because $L$ is a normal subgroup of $G$}. It follows that $[G,L]\PLH[G,L]$ is a normal subgroup of $G\PLH G$. {Moreover, $[G,L]\PLH[G,L]$ is a subset of $\big[ G\plh_L G, G\plh_L G\big]$ because $[G,L]\subset L$, and } $[G,L]\PLH[G,L]$ is a subset of $[G,G]\plh_{[G,L]}[G,G]$ {because $[G,L]\subset [G,G]$}. Thus, to show
$$
\big[ G\plh_L G, G\plh_L G\big]=[G,G]\plh_{[G,L]}[G,G],
$$
it suffices to show
\begin{equation}
\label{eqn_mod_identity}
\big[ G\plh_L G, G\plh_L G\big]\bmod [G,L]\PLH[G,L] =[G,G]\plh_{[G,L]}[G,G] \bmod [G,L]\PLH[G,L].
\end{equation}
Note that $\big[ G\plh_L G, G\plh_L G\big]$ is generated by elements of the form $([g_1,g_2], [g_1 l_1, g_2,l_2])$ for $g_1,g_2\in G$ and $l_1,l_2\in L$. Since elements in $G$ commute with elements in $L$ modulo $[G,L]$, we have $([g_1,g_2], [g_1 l_1, g_2,l_2])\equiv ([g_1,g_2], [g_1 , g_2])\bmod [G,L]\PLH[G,L]$. In other words,
\[
\big[ G\plh_L G, G\plh_L G\big]\bmod [G,L]\PLH[G,L] = \{(g,g): g\in [G,G]\} \bmod [G,L]\PLH[G,L].
\]
Similarly, since $[G,G]\plh_{[G,L]}[G,G]=\{(g,gh): g\in [G,G],\, h\in [G,L]\}$, we have
\[
[G,G]\plh_{[G,L]}[G,G] = \{(g,g): g\in [G,G]\} \bmod [G,L]\PLH[G,L].
\]
This finishes the proof of  \eqref{eqn_mod_identity}.
\end{proof}

From \cref{lem_commutator-of-Gx_LG} we can derive the following corollary.

\begin{Corollary}
\label{cor_horizontal_characters_of_Gx_LG}
$G/[G,G]\PLH L/[G,L]$ and $G\plh_L G/[G\plh_L G,G\plh_L G]$ are isomorphic as nilpotent Lie groups.
\end{Corollary}
\begin{proof}
Consider the map $\Phi\colon G\PLH L\to G\plh_L G/[G\plh_L G,G\plh_L G]$ defined as
\begin{equation*}
\Phi(a,b)= (a,ba)[G\plh_L G,G\plh_L G],\qquad \forall a\in G,~ \forall b\in L.
\end{equation*}
Clearly, $\Phi$ is well defined, smooth, surjective, and a homomorphism. Moreover, it follows from \cref{lem_commutator-of-Gx_LG} that the kernel of $\Phi$ equals $[G,G]\PLH [G,L]$. Indeed, $(a,b)$ belongs to the kernel of $\Phi$ if and only if $(a,ab)$ belongs to $[G\plh_L G,G\plh_L G]$. According to \cref{lem_commutator-of-Gx_LG}, this happens exactly when $(a,ab)$ belongs to $[G,G]\plh_{[G,L]}[G,G]$. Using the definition of relative independent product groups, we see that $(a,ab)\in [G,G]\plh_{[G,L]}[G,G]$ if and only if $a\in [G,G]$ and $b\in [G,L]$ as claimed.
\end{proof}

\cref{cor_horizontal_characters_of_Gx_LG} helps us better understand the maximal factor torus of the relatively independent self-product $X\plh_L X$, which will turn out to be an important aspect in the proofs of our main results. First, let us introduce the notion of a \define{horizontal character}.

\begin{Definition}[Horizontal characters, {cf.\ \cite{GT12a}}]
\label{def_horizontal_character}
Let $G$ be a nilpotent Lie group and $\Gamma$ a uniform and discrete subgroup of $G$. 
A \define{horizontal character of $(G,\Gamma)$} is any continuous map $\eta\colon X\to\C\setminus\{0\}$ satisfying
\begin{equation}
\label{eqn_horizontal_charcter_functional_equation}
\eta(ab\Gamma)\,=\, \eta(a\Gamma)\eta(b\Gamma), \qquad\forall a,b\in G.
\end{equation}
We say $\eta$ is non-trivial if $\eta$ is not constant equal to $1$.
\end{Definition}

\begin{Remark}
\label{rem_horizontal_characters_of_Gx_LG}
With the help of \cref{cor_horizontal_characters_of_Gx_LG} it is easy to describe the horizontal characters of the relatively independent product $(G\plh_L G,\Gamma\plh_L \Gamma)$ in terms of the horizontal characters of $(G,\Gamma)$ and $(L,\Gamma_L)$, where $\Gamma_L\coloneqq \Gamma\cap L$. Indeed, for any horizontal character $\eta$ of $(G\plh_L G,\Gamma\plh_L\Gamma)$ there exists a horizontal character $\eta_1$ of $(G,\Gamma)$ and a horizontal character $\eta_2$ of $(L,\Gamma_L)$ with $[G,L]\subset \ker \eta_2$ such that
$$
\eta\big((a_1,a_2)\Gamma\plh_L \Gamma\big)=\eta_1\big(a_2\Gamma\big)\eta_2\big(a_1 a_2^{-1}\Gamma_L\big), \qquad\forall (a_1,a_2)\in G\plh_L G.
$$
\end{Remark}

Observe that if $G$ is connected then horizontal characters descend to the maximal factor torus, where they generate an algebra that is uniformly dense due to the Stone-Weierstrass Theorem. Therefore, 
\cref{rem_horizontal_characters_of_Gx_LG} helps us understand the maximal factor torus of the relatively independent self-product $X\plh_L X$ in the case when $G$ is connected. However, we also need to better understand the maximal factor torus if $G$ is not connected. First, let us characterize the identity component of $G\plh_L G$.

\begin{Lemma}
\label{lem_connected-component-of-Gx_LG}
We have $( G\plh_L G)^\circ =G^\circ \plh_{L^\circ} G^\circ$.
\end{Lemma}

\begin{proof}
If $(G_i)_{i\in\mathcal{I}}$ are the connected components of $G$ and $(L_j)_{j\in\mathcal{J}} $ are the connected components of $L$ then the connected components of $G\plh_L G$ are $(G_i\plh_{L_j}G_k)_{i,k\in\mathcal{I},j\in\mathcal{J}}$, where
\[
G_i\plh_{L_j}G_k\coloneqq \{(a_1,a_2)\in G_i\times G_k: a_1a_2^{-1}\in L_j\}.
\]
This is because $G_i\plh_{L_j}G_k$ are open and connected subsets of $G\plh_L G$, and if $(i,k,j)\neq (i',k',j')$ then $G_i\plh_{L_j}G_k$ and $G_{i'}\plh_{L_{j'}}G_{k'}$ are disjoint. 
It follows that the connected component of $G\plh_L G$ that contains the identity is $G^\circ \plh_{L^\circ} G^\circ$.
\end{proof}

To study the maximal factor torus for non-connected $G$, we utilize a variant of the notion of a horizontal character.

\begin{Definition}[Pseudo-horizontal characters]
\label{def_pseudo_horizontal_character}
Let $G$ be a nilpotent Lie group and $\Gamma$ a uniform and discrete subgroup of $G$. 
A \define{pseudo-horizontal character of $(G,\Gamma)$} is any continuous map $\eta\colon X\to\C\setminus\{0\}$ satisfying
\begin{equation}
\label{eqn_pseudo_horizontal_charcter_functional_equation}
\eta(ab\Gamma)\,=\, \eta(a\Gamma)\eta(b\Gamma), \qquad\forall a\in G^\circ,~\forall b\in G.
\end{equation}
\end{Definition}

Note that if $G$ is connected then pseudo-horizontal characters are the same as horizontal characters.
Also, pseudo-horizontal characters descend to the maximal factor torus, where they generated a dense algebra, provided that $G/\Gamma$ is connected. Indeed, if $G/\Gamma$ is connected then $G^\circ$ acts transitively on $G/\Gamma$ and hence $G^\circ/[G^\circ,G^\circ]$ acts transitively on the maximal factor torus. This implies that the maximal factor torus is isomorphic to the compact abelian group $G^\circ/[G^\circ,G^\circ] (\Gamma\cap G^\circ)$ and the pseudo-horizontal characters descend to the group characters of this group.

\begin{Remark}
\label{rem_pseudo_horizontal_characters_of_Gx_LG}
By combining \cref{cor_horizontal_characters_of_Gx_LG} with  \cref{lem_connected-component-of-Gx_LG} we are now able to describe the pseudo-horizontal characters of $(G\plh_L G,\Gamma\plh_L \Gamma)$ in the case when $(G\plh_L G)/(\Gamma\plh_L \Gamma)$ is connected, similar to the way we described the horizontal characters of $(G\plh_L G,\Gamma\plh_L \Gamma)$ in \cref{rem_horizontal_characters_of_Gx_LG} above.
If $(G\plh_L G)/(\Gamma\plh_L \Gamma)$ is connected then it is isomorphic to \[(G\plh_L G)^\circ/((G\plh_L G)^\circ\cap (\Gamma\plh_L \Gamma))\] and hence the pseudo-horizontal characters of $G\plh_L G$ can be identified with the horizontal characters of $(G\plh_L G)^\circ=G^\circ\plh_{L^\circ} G^\circ$.
It follows that for any pseudo-horizontal character $\eta$ of $(G\plh_L G,\Gamma\plh_L\Gamma)$ there exists a pseudo-horizontal character $\eta_1$ of $(G,\Gamma)$ and a pseudo-horizontal character $\eta_2$ of $(L,\Gamma_L)$ with $[G^\circ,L^\circ]\subset \ker \eta_2$ such that
$$
\eta\big((a_1,a_2)\Gamma\plh_L \Gamma\big)=\eta_1\big(a_2\Gamma\big)\eta_2\big(a_1 a_2^{-1}\Gamma_L\big), \qquad\forall (a_1,a_2)\in G\plh_L G.
$$
\end{Remark}

Finally, here is another lemma that we will invoke numerous times in the later sections.

\begin{Lemma}
\label{lem_connected-Gx_LG}
Let $G$ be a nilpotent Lie group, $\Gamma$ a uniform and discrete subgroup of $G$, and $L$ a normal and rational subgroup of $G$. Let $\pi\colon G\to X$ denote the natural projection of $G$ onto the nilmanifold $X\coloneqq G/\Gamma$. If both $X$ and the sub-nilmanifold $\pi(L)$ are connected then the relatively independent self-product $X\plh_L X$ is also connected.
\end{Lemma}

\begin{proof}
Note that $X$ is connected if and only if $G=G^\circ\Gamma$. Likewise, $\pi(L)$ is connected if and only if $L^\circ \Gamma = L\Gamma = \Gamma L$.
In particular, any $a\in G$ can be written as $b\gamma$ with $b\in G^\circ$ and $\gamma\in G$, and any $s\in L$ can be written as $\kappa t$ with $t\in L^\circ$ and $\kappa\in L\cap\Gamma$.
It follows that  
\begin{align*}
G\plh_L G &=\{(a,a s): a\in G,~ s\in L\}
\\
&=\{(b\gamma,b\gamma \kappa t): b\in G^\circ,~t\in L^\circ,~\gamma\in\Gamma,~\kappa\in \Gamma\cap L\}.
\end{align*}
Since $L^\circ$ is normal, we have $b\gamma \kappa t=   \tilde{t} b \gamma\kappa$ where $\tilde{t}= (b \gamma\kappa) t (b \gamma\kappa)^{-1}\in L^\circ$. Hence
\begin{align*}
G\plh_L G &=
\{(b\gamma, \tilde{t} b \gamma \kappa ): b\in G^\circ,~\tilde{t}\in L^\circ,~\gamma\in\Gamma,~\kappa\in \Gamma\cap L\}
\\
&=
\{(b,\tilde{t} b): b\in G^\circ,~\tilde{t}\in L^\circ\}~\{(\gamma,\gamma \kappa): \gamma\in\Gamma,~\kappa\in \Gamma\cap L\}
\\
&=(G^\circ\plh_{L^\circ}G^\circ) (\Gamma\plh_L\Gamma),
\end{align*}
which implies $X\plh_L X$ is connected.
\end{proof}

\section{Reducing Theorems \ref{thm_B} and \ref{thm_C} to \cref{thm_G}}
\label{sec_reduction}

Our first step in proving Theorems \ref{thm_B} and \ref{thm_C} is to reduce them to the following result. 
%

\begin{Maintheorem}
\label{thm_G}
Let $G$ be a simply connected nilpotent Lie group and $\Gamma$ a uniform and discrete subgroup of $G$. Assume $v\colon\N\to G$ is a mapping of the form
$$
v(n)\,=\, a_1^{f_1(n)}\cdot\ldots\cdot a_k^{f_k(n)}b_1^{p_1(n)}\cdot\ldots\cdot b_m^{p_m(n)},\qquad\forall n\in\N,
$$
where $a_1,\ldots, a_k\in G^\circ$, $b_1,\ldots,b_m\in G$, the elements $a_1,\ldots,a_k,b_1,\ldots,b_m$ are pairwise commuting, $\overline{b_1^\Z\Gamma},\ldots, \overline{b_m^\Z\Gamma}$ are connected sub-nilmanifolds of $X=G/\Gamma$, $p_1,\ldots,p_m\in\R[t]$ are polynomials satisfying
\begin{enumerate}
[label=~(\ref{thm_G}$_{\arabic{enumi}}$), ref=(\ref{thm_G}$_{\arabic{enumi}}$),leftmargin=*]
\item
\label{itm_thm_G_A}
$p_j(\Z)\subset \Z$, for all $j=1,\ldots,m$,
\item
\label{itm_thm_G_B}
$\deg(p_j)=j$ for all $j=1,\ldots,m$,
\end{enumerate}
and $f_1, \ldots, f_k$ are functions belonging to some Hardy field $\Hardy$ satisfying
\begin{enumerate}
[label=~(\ref{thm_G}$_{\arabic{enumi}}$),ref=(\ref{thm_G}$_{\arabic{enumi}}$),leftmargin=*]
\setcounter{enumi}{2}
\item
\label{itm_thm_G_C}
$f_1(t)\prec \ldots\prec f_k(t)$,
\item
\label{itm_thm_G_D}
for all $f\in\{f_1,\ldots,f_k\}$ there exists $\ell\in\N$ such that $t^{\ell-1}\log(t)\prec f(t) \prec t^\ell$,
\item
\label{itm_thm_G_E}
for all $f\in\{f_1,\ldots,f_k\}$ with $\deg(f)\geq 2$ we have $f'\in\{f_1,\ldots,f_k\}$.
\end{enumerate}
Then $(v(n)\Gamma)_{n\in\N}$ is uniformly distributed in the sub-nilmanifold $\overline{ a_1^\R\cdots a_k^\R b_1^\Z\cdots b_m^\Z \Gamma}$.
\end{Maintheorem}

We remark that the set $\overline{ a_1^\R\cdots a_k^\R b_1^\Z\cdots  b_m^\Z \Gamma}$ appearing in the formulation of \cref{thm_G} above is indeed a sub-nilmanifold of $X$, which follows from \cite[Lemma A.6]{BMR17arXiv}.

\begin{Remark}
\label{rem_change_of_base_point}
Using the ``change of base-point'' trick (cf.\ \cite[p.\ 368]{Frantzikinakis09}), it is straightforward to see that if the sequence $ a_1^{f_1(n)}\cdots a_k^{f_k(n)}b_1^{p_1(n)}\cdots b_m^{p_m(n)}\Gamma$ is uniformly distributed in $\overline{ a_1^\R\cdots a_k^\R b_1^\Z\cdots  b_m^\Z \Gamma}$ then for every $c\in G$ the sequence $ a_1^{f_1(n)}\cdots a_k^{f_k(n)}b_1^{p_1(n)}\cdots b_m^{p_m(n)}c\Gamma$ is uniformly distributed in $\overline{ a_1^\R\cdots a_k^\R b_1^\Z\cdots  b_m^\Z c\Gamma}$.
\end{Remark}

The proof of \cref{thm_G} is given in \cref{sec_proof_orbit_closure_cesaro}. The remainder of this section is dedicated to showing that \cref{thm_G} implies Theorems \ref{thm_B} and \ref{thm_C}. For this, we will make use of \cref{lem_dom_of_group_lements} and \cref{cor_normal_form}, which are formulated and proved in the appendix. 

\begin{proof}[Proof that \cref{thm_G} implies Theorems \ref{thm_B} and \ref{thm_C}]
Let $G$ be a simply connected nilpotent Lie group, $\Gamma$ a uniform and discrete subgroup of $G$, and $X$ the nilmanifold $G/\Gamma$.
Let $a_1,\ldots,a_k\in G$ be pairwise commuting, $f_1,\ldots,f_k\in \Hardy$ satisfy \ref{property_P}, suppose $f_i(\N)\subset \dom(a_i)$ for all $i=1,\ldots,k$, and consider the sequence
$$
v(n)\coloneqq a_1^{f_1(n)}\cdot\ldots\cdot a_k^{f_k(n)},\qquad\forall n\in\N.
$$

To begin with, we divide the set $\{1,\ldots,k\}$ into two pieces. The first piece, which we denote by $\mathcal{I}$, consists of all $i\in\{1,\ldots,k\}$ for which $a_i$ belongs to the identity component $G^\circ$. The second piece $\mathcal{I}^c$ is defined as $\{1,\ldots,k\}\setminus\mathcal{I}$. According to \cref{lem_dom_of_group_lements}, for any $i\in\mathcal{I}^c$ there exist $m_i\in \N$ and $p_i\in\R[t]$ such that $p_i(\Z)\subset \Z$, ${1}/{m_i}\in\dom(a_i)$, and
$$
f_i(n)=p_i(n)/m_i,\qquad\forall n\in\N.
$$
Then, by \cref{cor_normal_form} applied with $V(t)=t$ to the family of functions $\{f_i:i\in\mathcal{I}\}$, we can find $m\in\N$, functions $g_1,\ldots,g_m\in\Hardy$, a set of polynomials $\{p_i:i\in\mathcal{I}\}\subset\R[t]$, and coefficients $\{\lambda_{i,1},\ldots,\lambda_{i,m}: i\in\mathcal{I}\}\subset\R$ such that the following properties hold:
\begin{enumerate}
[label=(\Roman{enumi}),ref=(\Roman{enumi}),leftmargin=*]
\item
\label{itm_1_0}
$g_1(t)\prec \ldots\prec g_m(t)$;
\item
\label{itm_2_0}
for all $g\in\{g_1,\ldots,g_m\}$ either $g=0$ or there exists $\ell\in\N$ such that $t^{\ell-1}\log(t)\prec g(t) \prec t^\ell$;
\item
\label{itm_3_0}
for all $g\in\{g_1,\ldots,g_m\}$ with $\deg(g)\geq 2$ we have $g'\in\{g_1,\ldots,g_m\}$;
\item
\label{itm_4_0}
for all $i\in\mathcal{I}$,
$$
\lim_{t\to\infty}\Bigg|f_i(t)-\sum_{l=1}^m \lambda_{i,l} g_l(t)-p_i(t)\Bigg|=0.
$$
\end{enumerate}
For $i\in\mathcal{I}$ the polynomial $p_i$ can be written as
$$
p_i(n)\,=\,c_{i,0}\binom{n}{0}+c_{i,1}\binom{n}{1}+c_{i,2}\binom{n}{2}+\ldots+c_{i,M}\binom{n}{M}
$$
for some real coefficients $c_{i,0},\ldots,c_{i,M}$.
For $i\notin\mathcal{I}$ the polynomial $p_i$ can also be written as
$$
p_i(n)\,=\,c_{i,0}\binom{n}{0}+c_{i,1}\binom{n}{1}+c_{i,2}\binom{n}{2}+\ldots+c_{i,M}\binom{n}{M}
$$
but with an additional feature. It is a standard fact from algebra that any polynomial which takes integer values on the integers can we expressed as an integer linear combination of binomial coefficients. Therefore, if $i\in\mathcal{I}^c$ then the coefficients $c_{i,0},\ldots,c_{i,M}$ are actually integers.

Next, for $l=\{1,\ldots,m\}$, define
$$
u_l\,\coloneqq\, \prod_{i\in\mathcal{I}}a_i^{\lambda_{i,l}}
$$
and for $j\in\{0,1,\ldots,M\}$ define
$$
e_j\,\coloneqq\, \left(\prod_{i\in\mathcal{I}}a_i^{c_{i,j}}\right)\cdot \left(\prod_{i\in\mathcal{I}^c}a_i^{c_{i,j}/m_i}\right).
$$
Note that for all $i\in\mathcal{I}$ the elements $a_i^{\lambda_{i,j}}$ and $a_i^{c_{i,j}}$ are well defined because $a_i\in G^\circ$, and for all $i\in\mathcal{I}^c$ the elements $a_i^{c_{i,j}/m_i}$ are well defined because ${1}/{m_i}\in\dom(a_i)$ and $c_{i,j}\in\Z$.
It follows that the elements $u_1,\ldots,u_m$ belong to $G^\circ$ and $u_1,\ldots,u_m,e_0,e_1,\ldots,e_M$ are pairwise commuting.

Let $d_G$ be any right-invariant metric on the nilpotent Lie group $G$.
Using property \ref{itm_4_0}, it is straightforward to check that
$$
\lim_{n\to\infty}~d_G\left(
a_1^{f_1(n)}\cdot\ldots\cdot a_k^{f_k(n)},~u_1^{g_1(n)}\cdot\ldots\cdot u_m^{g_m(n)}e_0^{\binom{n}{0}}e_1^{\binom{n}{1}}\cdot\ldots\cdot e_M^{\binom{n}{M}}
\right)\,=\,0.
$$
Therefore, instead of showing the conclusions of Theorems \ref{thm_B} and \ref{thm_C} for $(v(n)\Gamma)_{n\in\N}$, it suffices to show the conclusions of Theorems \ref{thm_B} and \ref{thm_C} for the sequence $(w(n)\Gamma)_{n\in\N}$, where
$$
w(n)\,\coloneqq\,u_1^{g_1(n)}\cdot\ldots\cdot u_m^{g_m(n)}e_1^{\binom{n}{1}}\cdot\ldots\cdot e_M^{\binom{n}{M}},\qquad n\in\N.
$$
Note that $e_0^{\binom{n}{0}}$ was omitted because it is a constant (cf.\ \cref{rem_change_of_base_point}).

It is a consequence of \cref{thm_leib_orb_cls} that for every $j\in\{1,\ldots,M\}$ there exists $q_j\in\N$ such that
for all $r\in\{0,1,\ldots,q_j-1\}$ the set
$$
\overline{e_j^{q_j\Z+r}\Gamma}=\overline{\{e_j^{q_jn+r}\Gamma: n\in\Z\}}
$$
is a connected sub-nilmanifold of $X$. If we take $q=\mathrm{lcm}(q_1,\ldots,q_M)$ then 
$$
\overline{e_j^{q\Z+r}\Gamma}
$$
is also connected for all $r\in\{0,1,\ldots,q-1\}$ and all $j\in\{1,\ldots,M\}$. Let $p_{j,r}(n)$ be the polynomial defined as
$$
p_{j,r}(n)\coloneqq\frac{1}{q}\left(\binom{qn+r}{j}-\binom{r}{j}\right),
$$
let $s_j\coloneqq e_j^q$ and $c_r\,\coloneqq\, e_1^{\binom{r}{1}}\cdot\ldots\cdot e_M^{\binom{r}{M}}$,
let $h_{i,r}(t)\coloneqq q^{-\deg(g_i)}g_i(qt+r)$ and $z_i\coloneqq u_i^{q^{\deg(g_i)}}$, and define 
$$
w_r(n)\coloneqq z_1^{h_{1,r}(n)}\cdot\ldots\cdot z_m^{h_{m,r}(n)}s_1^{p_{1,r}(n)}\cdot\ldots\cdot s_M^{p_{M,r}(n)}.
$$
A straightforward calculation shows that
$$
w(qn+r)\,=\,w_r(n) c_r,\qquad\forall r\in\{0,1,\ldots,q-1\}.
$$
Also note that the elements $z_1,\ldots,z_m$ belong to $G^\circ$ and $z_1,\ldots,z_m,s_1,\ldots,s_M$ are pairwise commuting. Additionally, the polynomials $p_{1,r},\ldots,p_{M,r}$ satisfy conditions \ref{itm_thm_G_A} and \ref{itm_thm_G_B} from \cref{thm_G}, and the functions $h_{1,r},\ldots,h_{m,r}$ belong to $\Hardy$ and satisfy conditions \ref{itm_thm_G_C} and \ref{itm_thm_G_D}. Condition \ref{itm_thm_G_E} is also satisfied, because $$
h_{i,r}'(t)\,=\, q^{-\deg(g_i)+1}g_i'(qt+r)\,=\, q^{-\deg(g_i')}g_i'(qt+r)
$$
and $g_1,\ldots,g_m$ satisfy \ref{itm_3_0}.
In conclusion, all the conditions of \cref{thm_G} are met, which means that for every $r\in\{0,1,\ldots,q-1\}$ the sequence $(w_r(n)c_r\Gamma)_{n\in\N}=(w(qn+r)\Gamma)_{n\in\N}$ is uniformly distributed in the sub-nilmanifold $\overline{ z_1^\R\cdots z_m^\R s_1^\Z\cdots s_M^\Z c_r\Gamma}$.
Define
$$
Y_r=\overline{ z_1^\R\cdots z_m^\R s_1^\Z\cdots s_M^\Z c_r\Gamma}
$$
and note that $Y_r$ is connected because $\overline{e_j^{q\Z+r}\Gamma}$ is connected for all $j\in\{1,\ldots,M\}$. Moreover, using that the map $g\mapsto c_r g$ from $G\to G$ is homeomorphic, we have
\begin{align*}
Y_r
&=\overline{ z_1^\R\cdots z_m^\R s_1^\Z\cdots s_M^\Z c_r\Gamma} 
\\
&=\overline{ c_r z_1^\R\cdots z_m^\R s_1^\Z\cdots s_M^\Z \Gamma}
\\
&=c_r\,\overline{z_1^\R\cdots z_m^\R s_1^\Z\cdots s_M^\Z \Gamma} 
\\
&=c_rY_0.
\end{align*} 
Let $Z=\overline{ z_1^\R\cdots z_m^\R e_1^\Z\cdots e_M^\Z \Gamma}$ and observe that $Y_0, Y_1,\ldots, Y_{q-1}$ are sub-nilmanifolds of $Z$, because $s_j= e_j^q$ and $c_r= e_1^{\binom{r}{1}}\cdot\ldots\cdot e_M^{\binom{r}{M}}$. Also, finitely many translates of $Y_0$ cover $Z$, because
\[
\bigcup_{0\leq \alpha_1,\ldots,\alpha_M\leq q-1} e_1^{\alpha_1}\cdot\ldots\cdot e_M^{\alpha_M} Y_0 ~=~ Z.
\] 
Since $Y_0$ is connected and finitely many translates of it cover $Z$, we conclude that $Y_0$ is the identity component of $Z$, i.e., $Y_0$ is the connect component of $Z$ containing $\Gamma$. Let $\tilde{H}$ be the stabilizer of $Z$,
\[
\tilde{H}=\{g\in G: gZ=Z\},
\]
and define $H\coloneqq \tilde{H}^\circ$. Then $Z=\tilde{H}\Gamma$ and $Y_0=H\Gamma$. Moreover, since $c_r\in \tilde{H}$ and $H$ is a normal subgroup of $\tilde{H}$, it follows that
\begin{align*}
Y_r & = c_rY_0
\\
& = c_r H\Gamma
\\
& = Hc_r\Gamma.
\end{align*}
Thus, choosing $x_r=c_r\Gamma$, we have shown that $Y_r=Hx_r$ for a connected and closed subgroup $H$ of $G$, which finishes the proof of \cref{thm_C}.
Also, since $\overline{ z_1^\R\cdots z_m^\R s_1^\Z\cdots s_M^\Z c_r\Gamma}$ is dense in $X$ if and only if its projection onto the maximal factor torus is dense there\footnote{Note that any nilmanifold possesses a niltranslation that acts ergodically on it. Therefore, given any sub-nilmanifold $Y$ of $X=G/\Gamma$ there exists some $a\in G$ that acts ergodically on $Y$. If the projection of $Y$ onto the maximal factor torus is surjective, then $a$ also acts ergodically on the maximal factor torus, which in view of \cref{thm_leib_crit_non-connected} implies that $a$ acts ergodically on $X$. This shows that a sub-nilmanifold of $X$ coincides with $X$ if and only if its projection onto the maximal factor torus equals the the maximal factor torus.}, it follows that $(w(qn+r)\Gamma)_{n\in\N}$ is uniformly distributed in $X$ if and only if its projection $(\vartheta(w(qn+r)\Gamma))_{n\in\N}$ onto $[G^\circ,G^\circ]\backslash X$ is uniformly distributed there.
In particular, $(w(n)\Gamma)_{n\in\N}$ is uniformly distributed in $X$ if and only if $(\vartheta(w(n)\Gamma))_{n\in\N}$ is uniformly distributed in $[G^\circ,G^\circ]\backslash X$, which proves the conclusion of \cref{thm_B}.
\end{proof}

\section{Proof of \cref{thm_G}}
\label{sec_proof_orbit_closure_cesaro}

For the proof of \cref{thm_G} we distinguish three cases:

The first case that we consider is when $G$ is abelian.
Although this case essentially follows from the work of Boshernitzan \cite{Boshernitzan94}, for completeness we state and prove it in \cref{sec_abelian-case} below. It serves as the base case for the induction used in the proofs of the two subsequent cases.

The second case of \cref{thm_G} that we consider is when $f_k(t)\prec t$ and $b_1=\ldots=b_m=1_G$.
We will refer to this as the ``sub-linear'' case of \cref{thm_G}, and it is proved in \cref{sec_sub-linear-case} below using induction on the nilpotency step of the Lie group $G$.
The reason why we consider this case separately is because its proof requires the use of a special type of van der Corput Lemma that is specifically designed to handle functions of sub-linear growth from a Hardy field (see \cref{prop_slow_vdC}).

Finally, in \cref{sec_sup-linear-case}, we prove the general case of \cref{thm_G}. The proof of the general case bears many similarities to the proof of the ``sub-linear'' case, but relies on the standard van der Corput Lemma and instead of induction on the nilpotency step of $G$ uses induction on the so-called \define{degree} of the sequence $v(n)$, see \cref{def_degree_of_g}.

\subsection{The abelian case}
\label{sec_abelian-case}

The following corresponds to the special case of \cref{thm_G} where $G$ is abelian.

\begin{Theorem}[The abelian case of \cref{thm_G}]
\label{thm_G_abelian}
Let $d=d_1+d_2$ and consider a map $v\colon\N\to \R^{d_1}\times\Z^{d_2}$ of the form
$$
v(n)= f_1(n)\alpha_1+\ldots+ f_k(n)\alpha_k+p_1(n)\beta_1+\ldots+p_m(n)\beta_m,
$$
where $\alpha_1,\ldots,\alpha_k\in\R^d\times\{(0,\ldots,0)\}$, $\beta_1,\ldots,\beta_m\in \R^{d_1}\times\Z^{d_2}$, $p_1,\ldots,p_m\in\R[t]$ satisfy properties \ref{itm_thm_G_A} and \ref{itm_thm_G_B}, and $f_1,\ldots,f_k\in\Hardy$ satisfy properties \ref{itm_thm_G_C}, \ref{itm_thm_G_D}, and \ref{itm_thm_G_E}. Moreover, let $\Lambda$ be a subgroup of  $\Z^{d_2}$ of finite index, define $\Delta\coloneqq \Z^{d_2}/\Lambda$, and assume that for every $j=1,\ldots,m$ the set $\overline{\{n\beta_j\bmod(\Z^{d_1}\times\Lambda): n\in\Z\}}$ is a connected subgroup of $(\R^{d_1}\times\Z^{d_2})/(\Z^{d_1}\times\Lambda) = \T^{d_1}\times \Delta$. Then the sequence $(v(n)\bmod(\Z^{d_1}\times\Lambda))_{n\in\N}$ is uniformly distributed in the subgroup
$$
T\coloneqq \overline{ \R\alpha_1+\ldots+ \R\alpha_k+\Z\beta_1+\ldots+\Z\beta_m\bmod(\Z^{d_1}\times\Lambda)}
$$
of $\T^d\times\Delta$.
\end{Theorem}

\begin{proof}
Since $\T^{d_1}\times \Delta$ is a compact abelian group, the algebra generated by continuous group characters $\eta\colon \T^{d_1}\times \Delta\to S^1=\{z\in\C: |z|=1\}$ is uniformly dense $C(\T^{d_1}\times \Delta)$. Therefore, to prove that $(v(n)\bmod(\Z^{d_1}\times\Lambda))_{n\in\N}$ is uniformly distributed in $T$ it suffices to show that for every continuous group characters $\eta$ that is non-trivial when restricted to $T$ we have
\begin{equation}
\label{eqn_thm_G_abelian_1}
\lim_{N\to\infty}\frac{1}{N}\sum_{n=1}^N \eta(v(n)\bmod(\Z^{d_1}\times\Lambda))\,=\,0.
\end{equation}
Let $e(x)$ be shorthand for $e^{2\pi i x}$ and choose $\xi_1,\ldots,\xi_k\in\R$ and $\zeta_1,\ldots,\zeta_m\in[0,1)$ such that
$\eta(t\alpha_i\bmod (\Z^{d_1}\times\Lambda))=e(t\xi_i)$ for all $t\in\R$ as well as
$\eta(n\beta_j\bmod (\Z^{d_1}\times\Lambda))=e(n\zeta_j)$ for all $n\in\N$.
This allows us to rewrite \eqref{eqn_thm_G_abelian_1} as
\begin{equation}
\label{eqn_thm_G_abelian_2}
\lim_{N\to\infty}\frac{1}{N}\sum_{n=1}^N e\big(f_1(n)\xi_1+\ldots+ f_k(n)\xi_k+p_1(n)\zeta_1+\ldots+p_m(n)\zeta_m\big)\,=\,0.
\end{equation}
Since $f_1,\ldots,f_k$ have different growth and satisfy property \ref{itm_thm_G_D}, if at least one of the numbers $\xi_1,\ldots,\xi_k$ is non-zero then it follows from Boshernitzan's Equidistribution Theorem (\cite[Theorem 1.8]{Boshernitzan94}) that \eqref{eqn_thm_G_abelian_2} holds.
If all of the $\xi_1,\ldots,\xi_k$ are zero and at least one of the numbers $\zeta_1,\ldots,\zeta_k$ is non-zero then \eqref{eqn_thm_G_abelian_2} holds too, because $p_1,\ldots,p_k$ have different degree and, since $\overline{\{n\beta_j\bmod(\Z^{d_1}\times\Lambda): n\in\Z\}}$ is connected for all $j\in\{1,\ldots,M\}$, any non-zero $\zeta_j$ must be an irrational number and polynomials with irrational coefficients are uniformly distributed mod $1$ by Weyl's Equidistribution Theorem \cite[Satz 14]{Weyl16}. To finish the proof, note that not all the numbers $\xi_1,\ldots,\xi_k,\zeta_1,\ldots,\zeta_m$ can be zero since $\eta$ was assumed to be non-trivial when restricted to $T$.
\end{proof}

\subsection{The sub-linear case}
\label{sec_sub-linear-case}

The purpose of this subsection is to prove the special case of \cref{thm_G} when $b_1,\ldots,b_m=1_G$ and  $f_k(t)\prec t$, which we dubbed the ``sub-linear case''.
For the convenience of the reader, let us state it as a separate theorem here.

\begin{Theorem}[The sub-linear case of \cref{thm_G}]
\label{thm_G_sub-linear}
Let $G$ be a simply connected nilpotent Lie group, $\Gamma$ a uniform and discrete subgroup of $G$, and $\Hardy$ a Hardy field.
For any mapping $v\colon\N\to G$ of the form
$$
v(n)= a_1^{f_1(n)}\cdot\ldots\cdot a_k^{f_k(n)},\qquad\forall n\in\N,
$$
where $a_1,\ldots, a_k\in G^\circ$ are commuting and $\log(t)\prec f_1 \prec \ldots \prec f_k\prec t\in\Hardy$, the sequence $(v(n)\Gamma)_{n\in\N}$ is uniformly distributed in the sub-nilmanifold $\overline{ a_1^\R\cdots a_k^\R \Gamma}$.
\end{Theorem}

Given a sub-nilmanifold $Y$ of a nilmanifold $X=G/\Gamma$ and a function $F\in\Cont(X)$ we will write $F(Y)$ to denote the quantity $\int F\d\mu_Y$. A sequence of sub-nilmanifolds $(Y_n)_{n\in\N}$ is said to be \define{uniformly distributed} in $X$ if
$$
\lim_{N\to\infty}\frac{1}{N}\sum_{n=1}^N F(Y_n)~=~\int F\d\mu_X,\qquad\forall F\in\Cont(X).
$$
The following lemma will be instrumental in our proof of \cref{thm_G_sub-linear}.
\begin{Lemma}
\label{lem_equidistribution_of_diagonal_in_Gx_LG_sub-linear}
Let $G$ be a connected simply connected nilpotent Lie group, $\Gamma$ a uniform and discrete subgroup of $G$ and consider the nilmanifold $X\coloneqq G/\Gamma$.
Let $b\in G$ be arbitrary and let $L$ denote the smallest connected, normal, rational, and closed subgroup of $G$ containing $b^\R=\{b^t:t\in\R\}$.
Let $X^\triangle$ denote the diagonal $\{(x,x): x\in X\}$.
Then for all but countably many $\xi\in\R$, the sequence  $\big((b^{\xi n},1_G)X^\triangle\big)_{n\in\N}$ is uniformly distributed in the relatively independent product $X\plh_L X$.\footnote{Since $\{(g,g):g\in G\}$ is a rational and closed subgroup of $G\plh_L G$, we can identify $(g\Gamma,g\Gamma)$ with $(g,g)\Gamma\plh_L\Gamma$. This allows us to view the diagonal $X^\triangle\coloneqq\{(x,x): x\in X\}$ as a sub-nilmanifold of the relatively independent product $X\plh_L X$.}
\end{Lemma}

\begin{proof}
It follows from \cite[Corollary 1.9]{Leibman05b} that $\big((b^{\xi n},1_G)X^\triangle\big)_{n\in\N}$ is uniformly distributed in $X\plh_L X$ if and only if
\begin{equation}
\label{eqn_equidistribution_of_diagonal_in_Gx_LG_1}
\lim_{N\to\infty}\frac{1}{N}\sum_{n=1}^N \eta\big((b^{\xi n},1_G)X^\triangle\big) =0
\end{equation}
for every non-trivial horizontal character $\eta$ of $(G\plh_L G,\Gamma\plh_L \Gamma)$.
As was mentioned in \cref{rem_horizontal_characters_of_Gx_LG}, for any horizontal character $\eta$ of $(G\plh_L G,\Gamma\plh_L \Gamma)$ there exists a horizontal character $\eta_1$ of $(G,\Gamma)$ and a horizontal character $\eta_2$ of $(L,\Gamma_L)$ with $[G,L]\subset \ker \eta_2$ and such that
$$
\eta\big((a_1,a_2)\Gamma\plh_L \Gamma\big)=\eta_1\big(a_2\Gamma\big)\eta_2\big(a_1 a_2^{-1}\Gamma_L\big), \qquad\forall (a_1,a_2)\in G\plh_L G.
$$
Therefore we get
$$
\lim_{N\to\infty}\frac{1}{N}\sum_{n=1}^N \eta\big((b^{\xi n},1_G)X^\triangle\big)
=\left(\int\eta_1\d\mu_X\right) \left( \lim_{N\to\infty}\frac{1}{N}\sum_{n=1}^N \eta_2(b^{\xi n}\Gamma_L)\right).
$$
Since $\eta$ is non-trivial, either $\eta_1$ or $\eta_2$ is non-trivial. If $\eta_1$ is non trivial then $\int\eta_1\d\mu_X=0$ and hence \eqref{eqn_equidistribution_of_diagonal_in_Gx_LG_1} is satisfied. It remains to deal with the case when $\eta_2$ is non-trivial. 

Note that $\eta_2 (b^t\Gamma_L)=\lambda^t$ for some $\lambda\in\C$ with $|\lambda|=1$. We claim that $\lambda\neq 1$. Before we verify this claim, let us show how it allows us finish the proof of \eqref{eqn_equidistribution_of_diagonal_in_Gx_LG_1}.
Indeed, if $\lambda\neq 1$ then $\lambda^{\xi}\neq 1$ for all but countably many $\xi$. Since there are only countably many horizontal characters, by excluding countably many ``bad'' $\xi$s for each horizontal character, there exists a co-countable set of ``good'' $\xi$s independent of the choice of $\eta_2$.
Since the \Cesaro{} average of $\lambda^{\xi n}$ is $0$ whenever $\lambda^\xi\neq 1$, it follows that for any such ``good'' $\xi$ we have
$$
 \lim_{N\to\infty}\frac{1}{N}\sum_{n=1}^N \eta_2(b^{\xi n}\Gamma_L) =0,
$$
which proves that \eqref{eqn_equidistribution_of_diagonal_in_Gx_LG_1} holds.

Let us now prove the claim that $\lambda\neq 1$ whenever $\eta_2$ is non-trivial.
By way of contradiction, assume $\lambda= 1$. Therefore $\eta_2(b^t \Gamma_L)=1$ for all $t\in\R$.
Consider the set
$$
K\coloneqq \{g\in L: \eta_2(g\Gamma_L)=1\}.
$$
Clearly $b^\R\subset K$. Since $\eta_2(g_1g_2\Gamma_L)=\eta_2(g_1\Gamma_L)\eta_2(g_2\Gamma_L)$ for all $g_1,g_2\in L$, we see that $K$ is a subgroup of $L$, and hence also a subgroup of $G$. Moreover, $[L,G]\subset \ker\eta_2$ implies $[L,G]\subset K$, from which we conclude that $K$ is a normal subgroup of $G$. Since $g\mapsto g\Gamma_L$ and $\eta_2\colon L/\Gamma_L\to \C$ are continuous maps, $K$ is a closed set. Also, since $K\Gamma_L=\eta^{-1}(\{1\})$ is a closed subset of $L/\Gamma_L$, $K$ is rational. In summary, $K$ is a normal, rational, and closed subgroup of $G$.
Let $K^\circ$ be the identity component of $K$. Then $K^\circ$ is also normal, rational, and closed. On top of that, $K^\circ$ is connected and contains $b^\R$. By the minimality assumption on $L$, we must have $K=L$. However, we also have $K\neq L$ because $\eta_2$ was assumed to be non-trivial. This is a contradiction.
\end{proof}

Besides \cref{lem_equidistribution_of_diagonal_in_Gx_LG_sub-linear}, another important ingredient in our proof of \cref{thm_G_sub-linear} is the following variant of van der Corput's Lemma.

\begin{Proposition}[van der Corput's Lemma for sub-linear functions]
\label{prop_slow_vdC}
Assume $f_1,\ldots,f_k$ are functions from a Hardy field $\Hardy$ satisfying $\log(t)\prec f_1(t)\prec\ldots\prec f_k\prec t$.
Let $\Psi\colon\R^k\to \C$ be a bounded and uniformly continuous function and suppose for all $s\in\R$ the limit
$$
A(s)~\coloneqq~\lim_{N\to\infty}
\frac{1}{N}\sum_{n=1}^N \Psi(f_1(n), \ldots,f_{k-1}(n), f_k(n)+s)\overline{\Psi(f_1(n), \ldots,f_{k}(n))}
$$
exists.
If for every $\epsilon>0$ there exists $\xi\in(0,\epsilon)$ such that
$$
\lim_{H\to\infty}\frac{1}{H}\sum_{h=1}^H A(\xi h)~=~0
$$
then necessarily
\begin{equation}
\label{eqn_conclusion_of_prop_slow_vdC}
\lim_{N\to\infty}\frac{1}{N}\sum_{n=1}^N \Psi(f_1(n),\ldots,f_k(n))=0.
\end{equation}
\end{Proposition}

\cref{prop_slow_vdC} is a special case of \cref{prop_slow_vdC_W-averages}, which is stated and proved in \cref{sec_slow_vdC}. Let us now turn to the proof of \cref{thm_G_sub-linear}.

\begin{proof}[Proof of \cref{thm_G_sub-linear}]
We proceed by induction on the nilpotency step of $G$. The case of \cref{thm_G_sub-linear} where $G$ is abelian (i.e., where the nilpotency step of $G$ equals $1$) has already been taken care of by \cref{thm_G_abelian}.
Let us therefore assume that $G$ is a $d$-step nilpotent Lie group with $d\geq 2$ and that \cref{thm_G_sub-linear} has already been proven for all cases where the nilpotency step of the Lie group is smaller than $d$.

By replacing $X$ with the sub-nilmanifold $\overline{a_1^\R\cdots a_k^\R \Gamma}$ 
(and $G$ with a rational and closed subgroup of itself)
if necessary, we can assume without loss of generality that
\begin{equation}
\label{eqn_proof_a_0}
X\,=\,\overline{a_1^\R\cdots a_k^\R \Gamma}. 
\end{equation}
In particular, $X$ is connected and therefore $G^\circ\Gamma = G$. On top of that, $a_1,\ldots,a_k\in G^\circ$. This means we can replace $G$ by $G^\circ$ if needed, which allows us to also assume that $G$ is connected. 

To show that $(v(n)\Gamma)_{n\in\N}$ is uniformly distributed in $X$ we must verify
\begin{equation}
\label{eqn_proof_a_1}
 \lim_{N\to\infty} \frac{1}{N}\sum_{n=1}^N  F\big(v(n)\Gamma\big) = \int F\d\mu_X
\end{equation}
for all continuous functions $F\in\Cont(X)$.
However, in light \cref{rem_central_characters_are_uniformly_dense} it is actually not necessary to check \eqref{eqn_proof_a_1} for all continuous functions.  
Indeed, since the linear span of central characters is uniformly dense in $\Cont(X)$, instead of \eqref{eqn_proof_a_1} it suffices to show
\begin{equation}
\label{eqn_proof_a_2}
 \lim_{N\to\infty} \frac{1}{N}\sum_{n=1}^N  \phi\big(v(n)\Gamma\big) = \int\phi\d\mu_X
\end{equation}
for central characters $\phi$ only.
Let us therefore fix a central character $\phi$ and let $\chi$ be the character of $Z(G)$ corresponding to $\phi$, i.e., the continuous group homomorphism from $Z(G)$ to the unit circle $\{z\in\C: |z|=1\}$ such that \eqref{eqn_central_charcter_functional_equation} holds.

Write $L$ for the smallest connected, normal, rational and closed subgroup of $G$ containing the element $a_k$.
Let $V$ be the intersection of $L$ with $Z(G)$ and note that by \cref{lem_normal-center-intersection}, $V$ is a non-trivial subgroup of $Z(G)$. Moreover, since $L$ is connected and $Z(G)$ is connected (note that $Z(G)$ is connected because $G=G^\circ$), $V$ is connected too.\footnote{It follows from \cite[Lemma 3]{Malcev49} that if $a$ belongs to a connected subgroup of a simply connected nilpotent Lie group $G$ then so does $a^t$ for all $t\in\R$. This property implies that the intersection of connected subgroups of $G$ is connected.\label{ftn_malcev}}

We now claim that we can assume $\chi$ is non-trivial when restricted to $V$ (by which we mean that there exists $s\in V$ such that $\chi(s)\neq 1$).
Indeed, if $\chi$ is trivial on $V$ then $\phi$ is invariant under the action of $V$. In this case, $\phi$ descends to a continuous function on the quotient space $X/V$. We can identify $X/V$ with the nilmanifold $(G/V)/(\Gamma /V)$ and, since $V$ is connected and non-trivial, the dimension of $X/V$ is smaller than the dimension of $X$. This allows us to reduce \eqref{eqn_proof_a_2} to an analogous question on a nilmanifold of strictly smaller dimension. By induction on the dimension, we can thus assume that $\chi$ is non-trivial when restricted to $V$. 

Since $\chi$ is non-trivial when restricted to $V$, it is in particular a non-trivial central character. Non-trivial central character have zero mean (see \cref{rem_zero_mean_central_char}), and hence \eqref{eqn_proof_a_2} becomes
\begin{equation}
\label{eqn_proof_a_3}
\lim_{N\to\infty} \frac{1}{N}\sum_{n=1}^N \phi\big(v(n)\Gamma\big)=0.
\end{equation}

In order to prove \eqref{eqn_proof_a_3} we use the ``van der Corput Lemma for sub-linear functions'', i.e., \cref{prop_slow_vdC}.
Define
\begin{equation}
\label{eqn_proof_a_4}
A(s)\coloneqq \lim_{N\to\infty} \frac{1}{N}\sum_{n=1}^N \phi\big(v(n)a_k^s \Gamma\big)\overline{\phi}\big(v(n)\Gamma\big).
\end{equation}
According to \cref{prop_slow_vdC}, if we can show that $A(s)$ is well defined for all $s\in\R$ (meaning that the limit on the right hand side of \eqref{eqn_proof_a_4} exists for all $s\in\R$) and for every $\epsilon>0$ there exists $\xi \in (0,\epsilon)$ such that
\begin{equation}
\label{eqn_proof_a_5}
\lim_{H\to\infty} \frac{1}{H}\sum_{h=1}^H A(\xi h) ~=~0,
\end{equation}
then \eqref{eqn_proof_a_3} holds.
Define
$$
v^{\triangle}(n)\coloneqq \big(v(n),v(n)\big),\qquad\forall n\in\N,
$$
and let $b\coloneqq a_k$.
Clearly, $(b^s,1_G)v^\triangle(n)$ takes values in $G\plh_L G$ for all $s\in\R$ and $n\in \N$.
This will allow us to express $A(s)$ in terms of the relatively independent product $X\plh_L X$ instead of the cartesian product $X\PLH X$, which, as we will see, turns out to be a big advantage.
Define the map $\Phi\colon X\plh_L X\to\C$ as $\Phi\big((g_1,g_2)\Gamma\plh_L \Gamma\big)= \phi(g_1\Gamma)\overline{\phi}(g_2\Gamma)$ for all $(g_1,g_2)\in G\plh_L G$.
Note that $\Phi$ is well defined and continuous.  We can now rewrite \eqref{eqn_proof_a_4} as
\begin{equation*}
A(s)= \lim_{N\to\infty} \frac{1}{N}\sum_{n=1}^N \Phi\big((b^s,1_G) v^\triangle(n)(\Gamma\plh_L \Gamma)\big).
\end{equation*}
We make two claims:
\begin{named}{Claim 1}{}
\label{claim_1_sub-linear}
The integral $\int\Phi\d\mu_{X\plh_L X}$ equals $0$.
\end{named}
\begin{named}{Claim 2}{}
\label{claim_2_sub-linear}
For all $s\in\R$, 
\begin{equation}
\label{eqn_proof_a_6-0}
\lim_{N\to\infty} \frac{1}{N}\sum_{n=1}^N \Phi\big((b^s,1_G) v^\triangle(n)(\Gamma\plh_L \Gamma)\big)\, =\, \Phi\big((b^s,1_G)\, X^\triangle\big).
\end{equation}
\end{named}
Once \ref{claim_1_sub-linear} and \ref{claim_2_sub-linear} have been verified, we can finish the proof of \eqref{eqn_proof_a_5} rather quickly. Indeed, \ref{claim_2_sub-linear} implies that the limit in $A(s)$ exists for all $s\in\R$ and
$$
\lim_{H\to\infty} \frac{1}{H}\sum_{h=1}^H A(\xi h) =
\lim_{H\to\infty} \frac{1}{H } \sum_{h=1}^H 
\Phi\big((b^{\xi h},1_G)X^\triangle\big).
$$
In view of \cref{lem_equidistribution_of_diagonal_in_Gx_LG_sub-linear}, we thus have for a co-countable set of $\xi$ that
$$
\lim_{H\to\infty} \frac{1}{H } \sum_{h=1}^H 
\Phi\big((b^{\xi h},1_G)X^\triangle\big)=\int\Phi\d\mu_{X\plh_L X},
$$
which together with \ref{claim_1_sub-linear} implies \eqref{eqn_proof_a_5}.

It remains to verify Claims 1 and 2.
\begin{proof}[Proof of \ref{claim_1_sub-linear}]
\renewcommand{\qedsymbol}{$\triangle$}
Recall that there exists $s\in V$ such that $\chi(s)\neq 1$.
Using the definition of $\Phi$ it is straightforward to check that
$$
\Phi\big((s,1_G)(g_1,g_2)(\Gamma\plh_L \Gamma)\big)=\chi(s)\Phi\big((g_1,g_2)(\Gamma\plh_L \Gamma)\big),\qquad \forall (g_1,g_2)\in G\plh_L G.
$$
Note that $(s,1_G)$ is an element of $G\plh_L G$ because $s\in V$ and $V\subset L$. Therefore $\mu_{X\plh_L X}$ is invariant under left-multiplication by $(s,1_G)$, which gives 
$$
\int \Phi\d\mu_{X\plh_L X}
~=~\int (s,1_G)\cdot \Phi\d\mu_{X\plh_L X}
~=~\chi(s)
\int\Phi\d\mu_{X\plh_L X}.
$$
Since $\chi(s)\neq 1$, we obtain $\int\Phi\d\mu_{X\plh_L X}=0$ as claimed.
\end{proof}

\begin{proof}[Proof of \ref{claim_2_sub-linear}]
\renewcommand{\qedsymbol}{$\triangle$}
Define a new function $\Phi_s\colon X\to\C$ via
$$
\Phi_s\big(g\Gamma\big)\coloneqq \Phi\big((b^s,1_G) (g,g)(\Gamma\plh_L \Gamma )\big),\qquad \forall g\in G.
$$
It is straightforward to check that \eqref{eqn_proof_a_6-0} is equivalent to
\begin{equation}
\label{eqn_proof_a_7}
\lim_{N\to\infty}\frac{1}{N} \sum_{n=1}^N 
\Phi_s\big(v(n)\Gamma\big)=\int \Phi_s \d\mu_X.
\end{equation}
Write $\sigma\colon G\to G/Z(G)$ for the natural projection of $G$ onto $G/Z(G)$ and set
\begin{eqnarray*}
\mark{G}&\coloneqq & G/Z(G);
\\
\mark{\Gamma}&\coloneqq & \sigma(\Gamma);
\\
\mark{X}&\coloneqq &\mark{G}/\mark{\Gamma};
\\
\mark{v}(n)&\coloneqq &  \sigma(v(n)).
\end{eqnarray*}
It follows from \eqref{eqn_central_charcter_functional_equation} that $\Phi$ is invariant under the action of $Z(G)^\triangle=\{(g,g):g\in Z(G)\}$ and so $\Phi_s$ is invariant under the action of $Z(G)$. Therefore, $\Phi_s$ descends to a continuous function $\mark{\Phi}_s$ on $\mark{X}$, which makes \eqref{eqn_proof_a_7} equivalent to
\begin{equation}
\label{eqn_proof_a_8}
\lim_{N\to\infty}\frac{1}{N} \sum_{n=1}^N 
\mark{\Phi}_s\big(\mark{v}(n)\mark{\Gamma}\big)=\int \mark{\Phi}_s \d\mu_{\mark{X}}.
\end{equation}
Since $\mark{G}$ has nilpotency step $d-1$, we can invoke the induction hypothesis and conclude that the sequence
$(\mark{v}(n)\mark{\Gamma})_{n\in\N}$ is uniformly distributed on the sub-nilmanifold
$
\overline{ \mark{a}_1^{\R}\cdot\ldots\cdot \mark{a}_k^{\R}\mark{\Gamma} },
$
where $\mark{a}_i\coloneqq\sigma(a_i)$, $i=0,1,\ldots,k$.
However, \eqref{eqn_proof_a_0} implies
$$
\mark{X}\,=\,\overline{ \mark{a}_1^{\R}\cdot\ldots\cdot \mark{a}_k^{\R}\mark{\Gamma} },
$$
which proves \eqref{eqn_proof_a_8}.
\end{proof}
This finishes the proofs of Claims 1 and 2, which in turn completes the proof of \cref{thm_G_sub-linear}.
\end{proof}

\subsection{The general case}
\label{sec_sup-linear-case}

In this subsection we deal with the general case of \cref{thm_G}.
In the proof of the ``sub-linear'' case in the previous subsection we got away with using induction on the nilpotency step of the Lie group $G$. Unfortunately, the proof of the general case requires a more complicated inductive procedure. 
This inductive scheme bears similarities to the ones used in \cite{GT12a} and \cite[Section 5]{BMR17arXiv} and relies on the notion of the  ``degree'' associated to a mapping $v\colon\N\to G$.
For the definition of this degree, the notion of a \define{filtration} is needed.

\begin{Definition}
\label{def_filtration}
Let $G$ be a nilpotent Lie group and $\Gamma$ a uniform and discrete subgroup of $G$.
Let $d\in\N$ and let $G_1\supset \ldots \supset G_d\supset G_{d+1}$ be subgroups of $G$ that are normal, rational and closed. We call $G_\bullet\coloneqq \{G_1,\ldots,G_d,G_{d+1}\}$ a \define{$d$-step filtration} of $G$ if $G_1=G$, $G_{d+1}=\{1_G\}$, and
$$
[G_i,G_j]\subset G_{i+j},\qquad \forall i,j\in \{1,\ldots,d\}~\text{with}~i+j\leq d+1.
$$
\end{Definition}

The next lemma shows how one can turn a $d$-step filtration of $G$ into a $d$-step filtration of the relatively independent self-product $G\plh_L G$.

\begin{Lemma}[cf.\ {\cite[Proposition 7.2]{GT12a}}]
\label{lem_filtration-of-Gx_LG}
Let $G_\bullet=\{G_1,\ldots,G_d,G_{d+1}\}$ be a $d$-step filtration of $G$ and suppose $L$ is a normal subgroup of $G$ with $L\subset G_2$.
Define $L_i\coloneqq L\cap G_{i+1}$ for $i=0,1,\ldots,d$, and set $L_{d+1}\coloneqq \{1_G\}$. Then
$$
(G\plh_L G)_\bullet=\big\{G_1\plh_{L_1}G_1,\ \ldots,\ G_d\plh_{L_d}G_d,\ G_{d+1}\plh_{L_{d+1}}G_{d+1}\big\}
$$
is a $d$-step filtration of $G\plh_L G$. 
\end{Lemma}

\begin{proof}
Suppose $(a,b)\in G_i\plh_{L_i}G_i$ and $(c,d)\in G_j\plh_{L_j}G_j$, where $i$ and $j$ belong to $\{1,\ldots,d\}$ and satisfy $i+j\leq d+1$. To complete the proof we must show that $([a,c],[b,d])\in G_{i+j}\plh_{L_{i+j}}G_{i+j}$. Clearly, $[a,c]\in G_{i+j}$ and $[b,d]\in G_{i+j}$. It remains to prove that
\begin{equation}
\label{eqn_contained_in_L_iplj}
[a,c][b,d]^{-1}\in L_{i+j}.
\end{equation}
Since $L_{i+j}=L\cap G_{i+j+1}$, we will establish \eqref{eqn_contained_in_L_iplj} in two steps: First we will show that $[a,c][b,d]^{-1}\in L$, and thereafter we will show that $[a,c][b,d]^{-1}\in G_{i+j+1}$.

Write $s\coloneqq a^{-1}b$ and $t\coloneqq c^{-1}d$. Then $s\in L_i$ and $(a,b)=(a,as)$. Likewise, $t\in L_j$ and $(c,d)=(c,ct)$.
So $[a,c][b,d]^{-1}\in L$ can be written as
\begin{equation}
\label{eqn_contained_in_L_1}
[a,c][b,d]^{-1}=a^{-1}c^{-1}act^{-1}c^{-1}s^{-1}a^{-1}ctas \in L.
\end{equation}
Since $s\in L$, \eqref{eqn_contained_in_L_1} is equivalent to
\begin{equation}
\label{eqn_contained_in_L_2}
a^{-1}c^{-1}act^{-1}c^{-1}s^{-1}a^{-1}cta \in L.
\end{equation}
Next, using $a L a^{-1} =L$ and $t\in L$, we see that \eqref{eqn_contained_in_L_2} is equivalent to
\begin{equation}
\label{eqn_contained_in_L_3}
c^{-1}act^{-1}c^{-1}s^{-1}a^{-1}c \in L.
\end{equation}
Using normality of $L$ again, \eqref{eqn_contained_in_L_3} reduces to
\begin{equation*}
\label{eqn_contained_in_L_4}
ct^{-1}c^{-1}s^{-1} \in L.
\end{equation*}
Finally, since $s\in L$ and $t\in L$, it follows that $ct^{-1}c^{-1}s^{-1}\in L$, which finishes the proof of \eqref{eqn_contained_in_L_1}.

It remains to show that
\begin{equation}
\label{eqn_contained_in_G_ipljpl1_1}
[a,c][b,d]^{-1}=a^{-1}c^{-1}act^{-1}c^{-1}s^{-1}a^{-1}ctas \in G_{i+j+1}.
\end{equation}
We can assume without loss of generality that $i\geq j$. Since $s\in G_{i+1}$, we have that $[s,c]\in G_{i+j+1}$, $[s,t]\in G_{i+j+2}\subset G_{i+j+1}$ and $[s,a]\in G_{2i+1}\subset G_{i+j+1}$. In particular, modulo $G_{i+j+1}$ the element $s$ commutes with $a$, $c$, and $t$. Hence
\begin{equation*}
\label{eqn_contained_in_G_ipljpl1_2}
[a,c][b,d]^{-1}\bmod G_{i+j+1}=  
a^{-1}c^{-1}act^{-1}c^{-1}a^{-1}cta.
\end{equation*}
Next, observe that $[t,a]\in G_{i+j+1}$ and hence $t$ commutes with $a$ modulo $G_{i+j+1}$. Therefore
\begin{eqnarray*}
[a,c][b,d]^{-1}\bmod G_{i+j+1}&=&  
a^{-1}c^{-1}act^{-1}c^{-1}a^{-1}cta
 \bmod G_{i+j+1}
\\
&=& 
a^{-1}c^{-1}act^{-1}c^{-1}a^{-1}cat
 \bmod G_{i+j+1}
\\
&=& [[c,a],t] \bmod G_{i+j+1}.
\end{eqnarray*}
Finally, since  
$[c,a]\in G_{i+j}$ we have $[[c,a],t]\in G_{i+j+1}$ and conclude that $[a,c][b,d]^{-1}\in G_{i+j+1}$.
\end{proof}

Given a nilpotent Lie group $G$, let $d_G\colon G\PLH G\to [0,\infty)$ be a right-invariant metric on $G$. For any uniform and discrete subgroup $\Gamma$ the metric $d_G$ descends to a metric $d_{X}$ on the nilmanifold $X= G/\Gamma$ in the following way:
\begin{equation}
\label{eqn:def-of-metric}
d_{X}(x\Gamma,y\Gamma)\coloneqq \inf\{d_G(x\gamma,y\gamma'):\gamma,\gamma'\in\Gamma\}.
\end{equation}
Given a subset $S\subset G$ and a point $g\in G$ we denote by $d(g,S)\coloneqq\inf_{s\in S}d_G(g,s)$ the distance between $S$ and $g$.

\begin{Definition}[{cf.\ \cite[Definition 1.8]{GT12a} and \cite[Definition 5.9]{BMR17arXiv}}]
\label{def_degree_of_g}
Let $G$ be a simply connected nilpotent Lie group, $\Hardy$ a Hardy field, and $v\colon\N\to G$ a mapping of the form
$$
v(n)\, \coloneqq\, a_1^{f_1(n)}\cdot\ldots\cdot a_k^{f_k(n)}b_1^{p_1(n)}\cdot\ldots\cdot b_m^{p_m(n)},\qquad\forall n\in\N,
$$
where $a_1,\ldots,a_k\in G^\circ$, $b_1,\ldots,b_m\in G$, the elements $a_1,\ldots,a_k, b_1,\ldots,b_m$ are commuting, $f_1,\ldots,f_k\in\Hardy$ have polynomial growth, and $p_1,\ldots,p_m\in\R[t]$ with $p_j(\Z
)\subset\Z$.
We define the \define{degree} of $v$ to be the smallest number $d\in\N$ such that there exists a $d$-step filtration $G_\bullet =\{G_1,G_2,\ldots, G_{d},G_{d+1}\}$ with the property that $b_j\in G_{\deg(p_j)+1}$ for all $j=1,\ldots,m$ and $a_i\in G_{\deg(f_i)}^\circ$ for all $i=1,\ldots,k$.
If $G_\bullet$ is such a minimal filtration then we say $G_\bullet$ \define{realizes the degree of $v$}.
If there exists no such filtration, then we say that $v$ has infinite degree.
\end{Definition}

\begin{Lemma}
\label{lem_finite_degree}
Let $G$ be a simply connected nilpotent Lie group $G$ and $\Hardy$ a Hardy field. Assume $v\colon\N\to G$ is a mapping of the form
$$
v(n)\, \coloneqq\, a_1^{f_1(n)}\cdot\ldots\cdot a_k^{f_k(n)}b_1^{p_1(n)}\cdot\ldots\cdot b_m^{p_m(n)},\qquad\forall n\in\N,
$$
where $a_1,\ldots,a_k\in G^\circ$, $b_1,\ldots,b_m\in G$, the elements $a_1,\ldots,a_k, b_1,\ldots,b_m$ are commuting, $f_1,\ldots,f_k\in\Hardy$ have polynomial growth, and $p_1,\ldots,p_m\in\R[t]$ with $p_j(\Z
)\subset\Z$.
Then $v$ has finite degree.
\end{Lemma}

\begin{proof}
Let $M\in\N$ be any number such that $\deg(f_i)\leq M$ for all $i=1,\ldots,k$ and $\deg(p_j)+1\leq M$ for all $j=1,\ldots,m$.
Let $C_\bullet \coloneqq  \{C_1,C_2,\ldots, C_{s},C_{s+1}\}$ denote the lower central series of $G$. Set $r\coloneqq (s+1)(M+1)$ and define a filtration
$$
G_\bullet=\{G_1, G_2,\ldots,G_{r},G_{r+1}\}
$$
by setting $G_{(j-1)(M+1)+i}\coloneqq C_j$ for all $j\in\{1,\ldots,s+1\}$ and $i\in\{1,\ldots,M+1\}$ and $G_{r+1}=\{1_G\}$.
It is straightforward to check that $G_\bullet$ is a filtration.
Also, since $G_i=G$ for all $i=1,\ldots, M+1$, we certainly have that $b_j\in G_{\deg(p_j)+1}$ for all $j=1,\ldots,m$ and $a_i\in G_{\deg(f_i)}^\circ$ for all $i=1,\ldots,k$.
\end{proof}

\begin{Remark}
Note that the the filtration $G_\bullet=\{G_1, G_2,\ldots,G_{r},G_{r+1}\}$ constructed in the above proof is not necessarily a filtration that realizes the degree of $v$. Nonetheless, its existence proves that the degree of $v$ does not exceed $r=(s+1)(M+1)$.
\end{Remark}

\begin{proof}[Proof of \cref{thm_G}, the general case]
Let $v\colon\N\to G$ be as in the statement of \cref{thm_G}.
We use induction on the degree $d$ of $v$, which is finite due to \cref{lem_finite_degree}.
The base case of this induction, which is when $d=1$, is covered by \cref{thm_G_abelian}, because if $d=1$ then $G$ must be abelian. 
Therefore, we only have to deal with the inductive step. Assume $d>1$ and \cref{thm_G} has already been proven for all 
mappings $\mark{v}\colon\N\to\mark{G}$ satisfying the hypothesis of \cref{thm_G} 
and whose degree
is strictly smaller than $d$.

By replacing $X$ with $\overline{a_1^\R\cdots a_k^\R b_1^\Z\cdots b_m^\Z\Gamma}$ if necessary\footnote{Let $G'$ be the the smallest rational and closed subgroup of $G$ containing $a_1^\R\cdots a_k^\R b_1^\Z\cdots b_m^\Z$. Then 
$\Gamma'=G'\cap \Gamma$ is a uniform and discrete subgroup of $G'$ and the nilmanifold $X'\coloneqq G'/\Gamma'$ can be identified with the sub-nilmanifold $\overline{a_1^\R\cdots a_k^\R b_1^\Z\cdots b_m^\Z\Gamma}$ of $X$. Moreover, $(v(n)\Gamma)_{n\in\N}$ is uniformly distributed in $\overline{a_1^\R\cdots a_k^\R b_1^\Z\cdots b_m^\Z\Gamma}$ if and only if $(v(n)\Gamma')_{n\in\N}$ is uniformly distributed in $X'$. Thus, by replacing $G$ with $G'$, $\Gamma$ with $\Gamma'$, and $X$ with $X'$, we can assume without loss of generality that $X=\overline{a_1^\R\cdots a_k^\R b_1^\Z\cdots b_m^\Z\Gamma}$.}, we will assume that
\begin{equation}
\label{eqn_proof_c_0}
X\,=\,\overline{a_1^\R\cdots a_k^\R b_1^\Z\cdots b_m^\Z\Gamma}. 
\end{equation}
Since $\overline{b_1^\Z\Gamma}, \ldots, \overline{b_m^\Z\Gamma}$ are assumed to be connected, the sub-nilmanifold $\overline{b_1^\Z\cdots b_m^\Z \Gamma}$ is also connected. It follows that $a_1^\R\cdots a_k^\R (\overline{b_1^\Z\cdots b_m^\Z \Gamma})$ is connected, which in turn implies that
$$
\overline{a_1^\R\cdots a_k^\R (\overline{b_1^\Z\cdots b_m^\Z  \Gamma})}\,=\,\overline{a_1^\R\cdots a_k^\R b_1^\Z\cdots b_m^\Z  \Gamma}\,=\,X
$$
is connected.
For technical reasons, it will be convenient to assume that  for every $j=1,\ldots,m$, if the sub-nilmanifold $\overline{b_j^\Z\Gamma}$ is a point then $b_j=1_G$. This assumption can be made without loss of generality, because if $\overline{b_j^\Z\Gamma}$ is a point then $b_j$ must belong to $\Gamma$, in which case we can simply replace $b_j$ with $1_G$ and the sequence $(v(n)\Gamma)_{n\in\N}$ remains unchanged.
 
Our goal is to show that $(v(n)\Gamma)_{n\in\N}$ is uniformly distributed in $X$, or equivalently,
\begin{equation}
\label{eqn_proof_c_1}
 \lim_{N\to\infty} \frac{1}{N}\sum_{n=1}^N  F\big(v(n)\Gamma\big) = \int F\d\mu_X
\end{equation}
for all $F\in\Cont(X)$. Repeating the same argument as was already used in the proof of \cref{thm_G_sub-linear}, we see that instead of \eqref{eqn_proof_c_1} it suffices to show
\begin{equation}
\label{eqn_proof_c_2}
 \lim_{N\to\infty} \frac{1}{N}\sum_{n=1}^N  \phi\big(v(n)\Gamma\big) = \int\phi\d\mu_X
\end{equation}
for all central characters $\phi$.
Let $\mathcal{I}$ denote all the numbers $i\in\{1,\ldots,k\}$ for which $\deg(f_i)\geq 2$, and let $\mathcal{J}$ be all the numbers $j\in\{1,\ldots,m\}$ for which $b_j\neq 1_G$. 
Note that if $\mathcal{I}$ and $\mathcal{J}$ are both the empty set then $f_k(t)\prec t$ and $b_1=\ldots=b_m=1_G$, and so we find ourselves in the ``sub-linear case'' of \cref{thm_G}. Since this case has already been taken care of by \cref{thm_G_sub-linear}, we can assume that either $\mathcal{I}$ or $\mathcal{J}$ is non-empty.

Let $G_\bullet =\{G_1,G_2,\ldots, G_{d},G_{d+1}\}$ be a $d$-step filtration that realizes the degree of $v$ (cf.\ \cref{def_degree_of_g}).
Among other things, this means $b_j^\Z\subset G_2$ for all $j=1,\ldots,m$ and $a_i\subset G_2^\circ$ for all $ i\in\mathcal{I}$. Note that $a_i\subset G_2^\circ$ implies $a_i^\R\subset G_2^\circ$ (cf.\ Footnote~\ref{ftn_malcev}).
Let $L$ denote the smallest closed rational and normal subgroup of $G$ containing  $a_{i}^\R$ for all $i\in\mathcal{I}$ and $b_j^\Z$ for all $j=1,\ldots,m$.
Then, if $\pi\colon G\to X$ denotes the natural projection of $G$ onto $X$, the set $\pi(L)$ is a sub-nilmanifold of $X$ containing the sub-nilmanifold $\overline{\prod_{i\in\mathcal{I}}a_{i}^\R \prod_{1\leq j\leq m}b_{j}^\Z\Gamma}$. This sub-nilmanifold is what Leibman calls the \define{normal closure}, and it is shown in \cite[p.\ 844]{Leibman10a} that the normal closure of a connected sub-nilmanifold is connected. In particular, since $\overline{\prod_{i\in\mathcal{I}}a_{i}^\R \prod_{1\leq j\leq m}b_{j}^\Z \Gamma}$ is connected, $\pi(L)$ is connected too.

Next, we claim that the identity component $L^\circ$ of $L$ is non-trivial.
To verify this claim, we are going to distinguish two cases. The first case is when $\mathcal{I}$ is non-empty. In this case, $L^\circ$ contains a one-parameter subgroup $a_i^\R$ for $i\in\mathcal{I}$ and is therefore non-trivial (we assume without loss of generality that $a_i\neq 1_G$ for all $i=1,\ldots,k$).
The second case is when $\mathcal{J}$ is non-empty. Note that if we are not in the first case, then we must be in the second, since either $\mathcal{I}$ or $\mathcal{J}$ is non-empty.
If $\mathcal{J}$ is non-empty then $\pi(L^\circ)$ contains $\overline{b_j^\Z\Gamma}$ for some $j\in\mathcal{J}$, and since $\overline{b_j^\Z\Gamma}$ is connected and not a point for every $j\in\mathcal{J}$, it follows that $L^\circ$ is non-trivial. 

Let $V$ be the intersection of $L^\circ$ with $Z(G)^\circ$. By \cref{cor_normal-center-intersection}, $V$ is a non-trivial subgroup of $Z(G)$. Also, as an intersection of connected subgroups, $V$ is connected (cf.\ Footnote~\ref{ftn_malcev}).
We can now use the same argument as in the proof of \cref{thm_G_sub-linear}, which involved induction on the dimension of the nilmanifold $X$, to show that it suffices to prove \eqref{eqn_proof_c_2} for the case when the central character $\phi$ has a ``central frequency'' $\chi$ that is non-trivial when restricted to $V$, i.e., there exists $s\in V$ such that $\chi(s)\neq 0$.
This also implies that $\int \phi\d\mu_X=0$, and hence \eqref{eqn_proof_c_2} can be written as
\begin{equation}
\label{eqn_proof_c_3}
\lim_{N\to\infty} \frac{1}{N}\sum_{n=1}^N \phi\big(v(n)\Gamma\big)=0.
\end{equation}
To prove \eqref{eqn_proof_c_3} we use van der Corput's trick. 
Define
\begin{equation}
\label{eqn_proof_c_4}
A(h)\coloneqq \lim_{N\to\infty} \frac{1}{N}\sum_{n=1}^N \phi\big(v(n+h)\Gamma\big)\overline{\phi}\big(v(n)\Gamma\big)
\end{equation}
whenever this limit exists.
In light of \cref{prop_vdC} (applied with $p_n=1$ for all $n\in\N$), \eqref{eqn_proof_c_3} holds if we can show that the limit on the right hand side of \eqref{eqn_proof_c_4} exists for all $h\in\N$ and
\begin{equation}
\label{eqn_proof_c_5}
\lim_{H\to\infty} \frac{1}{H}\sum_{h=1}^H A(h) ~=~0.
\end{equation}
We can interpret $\phi(v(n+h)\Gamma)\overline{\phi(v(n)\Gamma)}$ as $\phi\oPLH\overline{\phi}((v(n+h),v(n))(\Gamma\PLH\Gamma))$, where $\phi\oPLH \overline{\phi}$ is a continuous function on the product nilmanifold $X\PLH X= (G\PLH G)/(\Gamma\PLH\Gamma)$ and $(v(n+h),v(n))$ is an element in $G\PLH G$. Note that the sequence $(v(n+h),v(n))$ can be rewritten as
\begin{equation}
\label{eqn_proof_c_6}
(v(n+h),v(n))~=~ (\Delta_h w(n),1_G)\, v^\triangle(n),
\end{equation}
where $v^\triangle(n)= (v(n),v(n))$ and
\begin{equation}
\label{eqn_proof_c_7}
\Delta_h w(n)~=~ a_1^{\Delta_h f_1(n)}\cdot\ldots\cdot a_k^{\Delta_h f_k(n)}b_1^{\Delta_h p_1(n)}\cdot\ldots\cdot b_m^{\Delta_h p_m(n)},\qquad\forall n,h\in\N.
\end{equation}
For every $i\notin\mathcal{I}$ the function $f_i$ has degree $1$, which means its discrete derivative $\Delta_h f_i(n)$ is negligibly small for large $n$.
This implies that for every $i\notin\mathcal{I}$ the element $a_i^{\Delta_h f_i(n)}$ converges to the identity $1_G$ and can therefore be ignored. More precisely, using the right-invariance of the metric $d_G$, we have
$$
\lim_{n\to\infty}d_G\left(\Delta_h w(n),\, w_h(n) \right)\,=\, 0,
$$
where
$$
w_h(n)\coloneqq \prod_{i\in\mathcal{I}}a_{i}^{\Delta_h f_i(n)} \cdot\prod_{1\leq j\leq m}b_{j}^{\Delta_h p_j(n)}.
$$
It follows that if we set
\begin{equation}
\label{eqn_proof_c_7.5}
v^\sq_h(n)\,\coloneqq\, (w_h(n),1_G)\,v^\triangle(n)
\end{equation}
then, in view of \eqref{eqn_proof_c_6}, the difference between $(v(n+h),v(n))$ and $v^\sq_h(n)$ goes to zero as $n\to\infty$. Hence $A(h)$ equals
\begin{equation}
\label{eqn_proof_c_8}
A(h)\,=\,\lim_{N\to\infty} \frac{1}{N}\sum_{n=1}^N  (\phi\oPLH \overline{\phi}) \big(v_h^\sq(n)(\Gamma\PLH\Gamma)\big).
\end{equation}

The advantage of using \eqref{eqn_proof_c_8} instead of \eqref{eqn_proof_c_4} is that $v^\sq_h(n)\in G\plh_L G$ for all $n\in\N$. Define the map $\Phi\colon X\plh_L X\to\C$ as $\Phi\big((g_1,g_2)\Gamma\plh_L \Gamma\big)= \phi(g_1\Gamma)\overline{\phi(g_2\Gamma)}$ for all $(g_1,g_2)\in G\plh_L G$.
Note that $\Phi$ is well defined and continuous.
This allows us to rewrite \eqref{eqn_proof_c_8} as
\begin{equation}
\label{eqn_proof_c_9}
A(h)\,=\,\lim_{N\to\infty} \frac{1}{N}\sum_{n=1}^N \Phi\big(v^\sq_h (n)(\Gamma\plh_L \Gamma )\big).
\end{equation}

Let $Z(G)^\triangle\coloneqq\{(g,g):g\in Z(G)\}$ and denote by $\sigma\colon G\plh_L G\to (G\plh_L G)/Z(G)^\triangle$ the natural projection of $G\plh_L G$ onto $(G\plh_L G)/Z(G)^\triangle$.
Define
\begin{eqnarray*}
\mark{G}&\coloneqq & \sigma(G\plh_L G),
\\
\mark{\Gamma}&\coloneqq & \sigma\big(\Gamma\plh_L \Gamma\big),
\\
\mark{X} &\coloneqq & \mark{G}/\mark{\Gamma},
\\
\mark{v}_h(n)&\coloneqq  &  \sigma\big(v^\sq_h (n)\big).
\end{eqnarray*}
It follows from \eqref{eqn_central_charcter_functional_equation} that $\Phi$ is invariant under the action of $Z(G)^\triangle$. Therefore, $\Phi$ descends to a continuous function on $\mark{X}$, meaning there exists $\mark{\Phi}\in\Cont(\mark{X})$ such that
$$
\Phi\big((g_1,g_2)(\Gamma\plh_L \Gamma)\big)= \mark{\Phi}\big(\sigma(g_1,g_2)\mark{\Gamma}\big),\qquad\forall (g_1,g_2)\in G\plh_L G.
$$
It thus follows form \eqref{eqn_proof_c_9} that
\begin{equation}
\label{eqn_proof_c_10}
A(h)\,=\,\lim_{N\to\infty} \frac{1}{N}\sum_{n=1}^N \mark{\Phi}\big(\mark{v}_h(n) \mark{\Gamma}\big).
\end{equation}
We now make four claims.

\begin{named}{Claim 1}{}
\label{claim_1_sup-linear}
The integral $\int\mark{\Phi}\d\mu_{\mark{X}}$ equals zero.\end{named}

\begin{named}{Claim 2}{}
\label{claim_2_sup-linear}
For all non-trivial pseudo-horizontal characters $\mark{\eta}$ of $(\mark{G},\mark{\Gamma})$ (see \cref{def_pseudo_horizontal_character}) we have
\begin{equation}
\label{eqn_proof_c_11}
\lim_{H\to\infty}\lim_{N\to\infty}\frac{1}{H}\sum_{h=1}^H \frac{1}{N}\sum_{n=1}^N \mark{\eta}\big(\mark{v}_h(n)\, \mark{\Gamma}\big)=0.
\end{equation}
\end{named}

In the following, we call any mapping $u\colon\N\to H$ of the from $u(n)=g_1^{p_1(n)}\cdot\ldots\cdot g_\ell^{p_\ell(n)}$, where $g_1,\ldots,g_n$ are elements in a nilpotent Lie group $H$ and $p_1,\ldots,p_\ell$ are polynomials with $p_i(\Z)\subset\Z$, a \define{polynomial mapping} (cf.\ \cite[Subsection 1.3]{Leibman05a}).

\begin{named}{Claim 3}{}
\label{claim_3_sup-linear}
There exist polynomials $q_1,\ldots, q_m\in\R[t]$ with $q_j(\Z)\subset\Z$ and $1\leq \deg(q_1)< \ldots < \deg(q_m)\leq m$, polynomial mappings $c, e_1,\ldots,e_m\colon \N\to \mark{G}$, and polynomial mappings $u_1,\ldots, u_k\colon\N\to \mark{G}^\circ$ such that the elements $c(h), e_1(h), \ldots, e_m(h), u_1(h), \ldots, u_k(h)$ are pairwise commuting for every $h\in\N$, and
$$
d_{\mark{G}}\Big(\mark{v}_h(n),\  u_1(h)^{f_1(n)}\cdot\ldots\cdot u_k(h)^{f_k(n)}e_1(h)^{q_1(n)}\cdot\ldots\cdot e_m(h)^{q_m(n)}c(h)\Big)=\oh_{n\to\infty}(1)
$$
for every $h\in\N$.
\end{named}

\begin{named}{Claim 4}{}
\label{claim_4_sup-linear}
For every $h$ the degree of $\mark{v}_h$ is smaller than $d$.
\end{named}

Before we provide the proofs of Claims 1, 2, 3, and 4, let us see how they can be used to prove that the limit in $A(h)$ exists for all $h\in\N$ and \eqref{eqn_proof_c_5} holds.
Claims 3 and 4 allow us to invoke the induction hypothesis and deduce that for every $h\in\N$ the sequence $(\mark{v}_h(n)\mark{\Gamma})_{n\in\N}$ is uniformly distributed in the sub-nilmanifold
$$
\mark{Y}_h\,\coloneqq\, \overline{u_1(h)^\R\cdot\ldots\cdot u_k(h)^\R e_1(h)^\Z\cdot\ldots\cdot e_m(h)^\Z c(h) \mark{\Gamma}}.
$$
(As was explained in \cref{rem_change_of_base_point}, it is not a problem that the ``base point'' of the sequence $(\mark{v}_h(n)\mark{\Gamma})_{n\in\N}$ is $c(h)\mark{\Gamma}$ instead of $\mark{\Gamma}$.)
As a consequence we have $A(h)= \mark{\Phi}(\mark{Y}_h)$, which in particular proves that the limit in $A(h)$ exists for all $h\in\N$. Moreover, \eqref{eqn_proof_c_5} will follow if we can show that
\begin{equation}
\label{eqn_proof_c_12}
\lim_{H\to\infty}\frac{1}{H}\sum_{h=1}^H \mark{\Phi}(\mark{Y}_h)\,=\, 0.
\end{equation}
Given a vector $\xi=(\xi_1,\ldots,\xi_k)\in\R^k$, consider the multi-parameter polynomial sequence $w_{\xi}\colon \Z^{k+m+1}\to \mark{G}$ defined as
\[w_\xi(h,n_1,\ldots,n_{k}, \ell_1,\ldots,\ell_m)= u_1(h)^{\xi_1 n_1}\cdot\ldots\cdot u_k(h)^{\xi_kn_k} e_1(h)^{\ell_1}\cdot\ldots\cdot e_m(h)^{\ell_m} c(h).\]
Arguing as in the proof of \cite[Lemma A.7]{BMR17arXiv} we can find for every $h\in\N$ a co-null set $\Xi_h\subset \R^k$ such that for all $\xi=(\xi_1,\ldots,\xi_k)\in \Xi_h$ we have
\begin{equation}
\label{eqn_cht_1}
\begin{split}
&\overline{u_1(h)^\R\cdot\ldots\cdot u_k(h)^\R e_1(h)^\Z\cdot\ldots\cdot e_m(h)^\Z c(h) \mark{\Gamma}}
\\
&\qquad\qquad~=~
\overline{u_1(h)^{\xi_1\Z}\cdot\ldots\cdot u_k(h)^{\xi_k\Z} e_1(h)^\Z\cdot\ldots\cdot e_m(h)^\Z c(h) \mark{\Gamma}}.
\end{split}
\end{equation}
It follows that for every $h\in\N$ and every $\xi\in\Xi_h$ the sequence 
\[
(w_\xi(h,n_1,\ldots,n_k,\ell_1,\ldots,\ell_m,h)\mark{\Gamma})_{(n_1,\ldots,n_k,\ell_1,\ldots,\ell_m)\in\N^{k+m}}
\]
is dense in $\mark{Y}_h$.
By invoking \cite[Theorem A, p.\ 216]{Leibman05b}, we have that since this sequence is dense in $\mark{Y}_h$, it is also uniformly distributed in $\mark{Y}_h$. This means that
\[
\lim_{N\to\infty}\frac{1}{N^{m+k}}\sum_{n_1,\ldots,\ell_m=1}^N \mark{F}\big(w_\xi(h,n_1,\ldots,n_k,\ell_1,\ldots,\ell_m,h)\mark{\Gamma}\big) = \int \mark{F}\d\mu_{\mark{Y}_h}
\]
for all continuous functions $\mark{F}\colon\mark{X}\to\C$. Henceforth, let $\xi$ be any number in $\bigcap_{h\in\N}\Xi_h$.
Note that Claim 2 implies
\begin{equation*}
\lim_{H\to\infty}\frac{1}{H}\sum_{h=1}^H \mark{\eta}(\mark{Y}_h)\,=\, 0
\end{equation*}
for all non-trivial pseudo-horizontal characters $\mark{\eta}$ of $(\mark{G},\mark{\Gamma})$.
It follows that
\begin{equation*}
\lim_{H\to\infty}\lim_{N\to\infty}\frac{1}{H}\sum_{h=1}^H 
\frac{1}{N^{m+k}}\sum_{n_1,\ldots,\ell_m=1}^N \mark{\eta}\big(w_\xi(n_1,\ldots,n_k,\ell_1,\ldots,\ell_m,h)\mark{\Gamma}\big) 
\,=\, 0.
\end{equation*}
Note also that $X\plh_L X$ is connected, due to \cref{lem_connected-Gx_LG} and the fact that both $X$ and $\pi(L)$ are connected. Thus, it follows from the work of Leibman (see \cite[Theorems A and B]{Leibman05b}) that
the sequence 
\[
(w_\xi(h,n_1,\ldots,n_k,\ell_1,\ldots,\ell_m,h)\mark{\Gamma})_{(h,n_1,\ldots,n_k,\ell_1,\ldots,\ell_m)\in\N^{k+m+1}}
\]
is \define{well distributed}\footnote{A sequence $(x_{n_1,\ldots,n_k})_{(n_1,\ldots,n_k)\in \N^k}$ of points in a nilmanifold $X$ is said to be \define{well distributed} in $X$ if for all $F\in\Cont(X)$ and all $\epsilon>0$ there exists $K\in \N$ such that for all $M_1,N_1,M_2,N_2,\ldots,M_k,N_k\in\N$ with $N_i-M_i\geq K$ for all $i=1,\ldots,k$ we have $$\Big|\frac{1}{(N_1-M_1)\cdot\ldots\cdot (N_k-M_k)}\sum_{(n_1,\ldots,n_k)\in [M_1,N_1)\PLH\ldots\PLH[M_k,N_k)} F(x_{n_1,\ldots,n_k}) - \int F\d\mu_X\Big|~\leq~ \epsilon.$$} in $\mark{X}$.
We conclude that
\begin{align*}
\lim_{H\to\infty}\frac{1}{H} & \sum_{h=1}^H \mark{F}\big(\mark{Y}_h\big) 
\\
&=
\lim_{H\to\infty}\lim_{N\to\infty}\frac{1}{H} \sum_{h=1}^H 
\frac{1}{N^{m+k}}\sum_{n_1,\ldots,\ell_m=1}^N \mark{F}\big(w_\xi(n_1,\ldots,n_k,\ell_1,\ldots,\ell_m,h)\mark{\Gamma}\big) 
\\
&=
\int\mark{F}\d\mu_{\mark{X}}
\end{align*}
for all continuous functions $\mark{F}\colon\mark{X}\to\C$.
In particular,
\begin{equation*}
\lim_{H\to\infty}\frac{1}{H}\sum_{h=1}^H  \mark{\Phi}(\mark{Y}_h) \,=\, \int\mark{\Phi}\d\mu_{\mark{X}}.
\end{equation*}
Now we can simply invoke Claim 1 to conclude that \eqref{eqn_proof_c_12} holds.

Let us now turn to the proofs of Claims 1, 2, 3, and 4.

\begin{proof}[Proof of \ref{claim_1_sup-linear}]
\renewcommand{\qedsymbol}{$\triangle$}
The following argument is very similar to the proof of Claim 1 which appeared in the proof of \cref{thm_G_sub-linear} in \cref{sec_sub-linear-case} above.
Recall that $\chi$ is non-trivial when restricted to $V$, meaning that there exists $s\in V$ such that $\chi(s)\neq 1$.
Let $\mark{s}\coloneqq \sigma(s,1_G)$, where $1_G$ denotes the identity element of $G$. 
Using the definition of $\mark{\Phi}$ it is straightforward to check that
$$
\mark{\Phi}(\mark{s}\mark{x})\,=\,\chi(s)\mark{\Phi}(\mark{x}),\qquad \forall \mark{x}\in \mark{X}.
$$
Since $\mu_{\mark{X}}$ is invariant under $\mark{s}$, we have that
$$
\int\mark{\Phi}(\mark{x})\d\mu_{\mark{X}}(\mark{x})\,=\,
\int\mark{\Phi}(\mark{s}\mark{x})\d\mu_{\mark{X}}(\mark{x})\,=\,
\chi(s)
\int\mark{\Phi}(\mark{x})\d\mu_{\mark{X}}(\mark{x}),
$$
and hence $\int\mark{\Phi}\d\mu_{\mark{X}}=0$ as claimed.
\end{proof}

\begin{proof}[Proof of \ref{claim_2_sup-linear}]
\renewcommand{\qedsymbol}{$\triangle$}
For any pseudo-horizontal character $\mark{\eta}$ of $(\mark{G},\mark{\Gamma})$ there exists a pseudo-horizontal character $\eta$ of $(G\plh_L G, \Gamma \plh_L\Gamma)$ such that $\mark{\eta}\circ \sigma=\eta $. Thus, instead of \eqref{eqn_proof_c_11}, it suffices to show that for all non-trivial pseudo-horizontal characters $\eta$ of $(G\plh_L G,\Gamma \plh_L\Gamma)$ we have
\begin{equation}
\label{eqn_proof_c_13}
\lim_{H\to\infty}\lim_{N\to\infty}\frac{1}{H}\sum_{h=1}^H \frac{1}{N}\sum_{n=1}^N \eta\big(v_h^\sq(n) \Gamma\plh_L\Gamma \big)=0.
\end{equation}

According to \cref{rem_pseudo_horizontal_characters_of_Gx_LG}, there exist a pseudo-horizontal character $\eta_1$ of $(G,\Gamma)$ and a pseudo-horizontal character $\eta_2$ of $(L,(L\cap\Gamma))$ with $[G^\circ,L^\circ]\subset \ker \eta_2$ such that
$$
\eta\big((a,b)\Gamma\plh_L\Gamma\big)=\eta_1(b\Gamma)\eta_2(a b^{-1}\Gamma_L),\qquad \forall(a,b)\in G\plh_L G,
$$
where $\Gamma_L\coloneqq \Gamma\cap L$.
Thus, by \eqref{eqn_proof_c_7.5},
$$
\eta\big(v_h^\sq(n)\Gamma\plh_L\Gamma\big)
=
\eta_1\big(v(n)\Gamma\big)\eta_2\big(w_h(n)\Gamma_L\big).
$$

Although we have $b_1,\ldots,b_m\in L$, we don't necessarily have $b_1,\ldots,b_m\in L^\circ$. This makes it more difficult to study the expressions
$\eta_1(v(n)\Gamma)$ and $\eta_2(w_h(n)\Gamma_L)$, because $\eta_1$ and $\eta_2$ are only pseudo-horizontal characters and not horizontal characters.
However, we can circumvent these difficulties in the following way. Since $\pi(L)$ is connected, we have
\begin{equation}
\label{eqn_claim_2_-1}
L^\circ \Gamma_L=L.
\end{equation}
As is explained in \cite[Subsections 2.6 and 2.7 on p.~204]{Leibman05a}, under these conditions there exist a polynomial sequence $g_1\colon\N\to L^\circ$ such that
\begin{equation}
\label{eqn_claim_2_0}
g_1(n)\Gamma_L \,=\, b_1^{p_1(n)}\cdot\ldots\cdot b_m^{p_m(n)}\Gamma_L,\qquad\forall n\in\N.
\end{equation}
The advantage of using $g_1(n)$ instead of $b_1^{p_1(n)}\cdot\ldots\cdot b_m^{p_m(n)}$ is that $g_1(n)$ takes values in $L^\circ\subset G^\circ$ and hence the image of $g_1(n)$ under $\eta_1$ and $\eta_2$ is easier to understand.
A downside of making this trade-off is that, unlike $b_1,\ldots,b_m$, the values of $g_1(n)$ do not necessarily commute with $a_1,\ldots,a_k$. But for the current proof (meaning the proof of \ref{claim_2_sup-linear}) this commutativity is not needed.
Similarly, we can find a polynomial sequence in two variables $g_2\colon\N^2\to L^\circ$ such that
$$
g_2(n,h)\Gamma_L \,=\, b_1^{\Delta_h p_1(n)}\cdot\ldots\cdot b_m^{\Delta_h  p_m(n)}\Gamma_L,\qquad\forall n,h\in\N.
$$
Note that even though
$$
b_1^{\Delta_h p_1(n)}\cdot\ldots\cdot b_m^{\Delta_h  p_m(n)}~=~\left(b_1^{p_1(n+h)}\cdot\ldots\cdot b_m^{p_m(n+h)}\right)\left(b_1^{p_1(n)}\cdot\ldots\cdot b_m^{p_m(n)}\right)^{-1},
$$
we do not necessarily have $g_2(n,h)=g_1(n+h)g_1(n)^{-1}$. But we do have
\begin{equation}
\label{eqn_claim_2_1}
g_2(n,h)[L^\circ,L^\circ]\Gamma_L~=~g_1(n+h)g_1(n)^{-1}[L^\circ,L^\circ]\Gamma_L,
\end{equation}
which we will make use of later.

It will be convenient to pick $\alpha_1,\ldots,\alpha_k$ and $(\zeta_i)_{i\in\mathcal{I}}$ such that
\begin{eqnarray*}
\eta_1(a_i^t\Gamma) &=& e(t\alpha_i),\quad i=1,\ldots,k,~\forall t\in\R,
\\
\eta_2(a_i^t \Gamma_L) & =& e(t \zeta_i),\quad i\in\mathcal{I},~\forall t\in\R,
\end{eqnarray*}
where $e(x)$ is shorthand for $e^{2\pi i x}$.
From \eqref{eqn_pseudo_horizontal_charcter_functional_equation} and the fact that $a_1,\ldots,a_k\in G^\circ$ as well as $a_i\in L^\circ$ for all $i\in\mathcal{I}$, it follows that
\begin{equation}
\label{eqn_claim_2_1.5}
\begin{split}
\eta_1\big(v(n) & \Gamma\big) \eta_2\big(w_h(n)\Gamma_L\big)=
\\
=&~
\eta_1\left(\prod_{1\leq i\leq k}a_{i}^{ f_i(n)} \cdot\prod_{1\leq j\leq m}b_{j}^{ p_j(n)}\Gamma\right)\eta_2\left(\prod_{i\in\mathcal{I}}a_{i}^{\Delta_h f_i(n)} \cdot\prod_{1\leq j\leq m}b_{j}^{\Delta_h p_j(n)} \Gamma_L\right)
\\
=&~
e\left(\sum_{i=1}^k  f_i(n)\alpha_i \, + \sum_{i\in\mathcal{I}}\Delta_h f_i(n) \zeta_i \right)
\eta_1\left(\prod_{1\leq j\leq m}b_{j}^{ p_j(n)}\Gamma\right)\eta_2\left(\prod_{1\leq j\leq m}b_{j}^{\Delta_h p_j(n)} \Gamma_L\right)
\\
=&~
e\left(\sum_{i=1}^k  f_i(n)\alpha_i \, + \sum_{i\in\mathcal{I}}\Delta_h f_i(n) \zeta_i \right)
\eta_1\left(g_1(n)\Gamma\right)\eta_2\left(g_2(n,h)\Gamma_L\right).
\end{split}
\end{equation}
Note that $[L^\circ,L^\circ]\Gamma_L$ belongs to the kernel of $\eta_2$ and so it follows from \eqref{eqn_claim_2_1} that
$$
\eta_2\left(g_2(n,h)\Gamma_L\right)~=~\eta_2\left(g_1(n+h)g_1(n)^{-1}\Gamma_L\right).
$$
Let $r_1,r_2\in\R[t]$ be polynomials such that 
$$
\eta_1\left(g_1(n)\right)~=~e\big(r_1(n)\big)
$$
as well as
$$
\eta_2\left(g_1(n)\Gamma_L\right)~=~e\big(r_2(n)\big).
$$
Then \eqref{eqn_claim_2_1.5} implies
$$
\eta_1\big(v(n) \Gamma\big) \eta_2\big(w_h(n)\Gamma_L\big)=
e\left(\sum_{i=1}^k  f_i(n)\alpha_i \, + \sum_{i\in\mathcal{I}}\Delta_h f_i(n) \zeta_i +r_1(n)+r_2(n+h)-r_2(n)\right)
$$
and so \eqref{eqn_proof_c_13} becomes
\begin{equation}
\label{eqn_claim_2_2}
\lim_{H\to\infty}\lim_{N\to\infty}\frac{1}{H}\sum_{h=1}^H \frac{1}{N}\sum_{n=1}^N e\left(\sum_{i=1}^k  f_i(n)\alpha_i  + \sum_{i\in\mathcal{I}}\Delta_h f_i(n) \zeta_i +r_1(n)+r_2(n+h)-r_2(n)\right)=0.
\end{equation}
Since $f_1,\ldots,f_k$ have different growth (see property \ref{itm_thm_G_C}) and behave independently from polynomials (due to property \ref{itm_thm_G_D}), it follows that if at least one of the $\alpha_i$ is non-zero or at least one of the $\zeta_i$ is non-zero, then \eqref{eqn_claim_2_2} is satisfied and we are done.
Let us therefore assume $\alpha_i=0$ for all $i=1,\ldots,k$ and $\zeta_i=0$ for all $i\in\mathcal{I}$. In this case, \eqref{eqn_claim_2_2} is equivalent to
\begin{equation}
\label{eqn_claim_2_3}
\lim_{H\to\infty}\lim_{N\to\infty}\frac{1}{H}\sum_{h=1}^H \frac{1}{N}\sum_{n=1}^N e\left(r_1(n)+r_2(n+h)-r_2(n)\right)=0.
\end{equation}
Averages of polynomial sequences are known to behave very regularly. In particular, the order of limits in \eqref{eqn_claim_2_3} can be interchanged freely, which means that \eqref{eqn_claim_2_3} is equivalent to 
\begin{equation*}
\lim_{N\to\infty}\lim_{H\to\infty}\frac{1}{N}\sum_{n=1}^N \frac{1}{H}\sum_{h=1}^H  e\left(r_1(n)+r_2(n+h)-r_2(n)\right)=0,
\end{equation*}
which is the same as
\begin{equation}
\label{eqn_claim_2_4}
\left(\lim_{N\to\infty}\frac{1}{N}\sum_{n=1}^N e\left(r_1(n)-r_2(n)\right)\right) \left(\lim_{H\to\infty}\frac{1}{H}\sum_{h=1}^H  e\left(r_2(h)\right)\right)=0.
\end{equation}
Recall that $\overline{b_i^\Z\Gamma}$ is connected for all $i=1,\ldots,m$. This implies that $\overline{b_i^\Z\Gamma_L}$ is also connected for all $i=1,\ldots,m$ and therefore
$$
\overline{\big\{b_1^{p_1(n)}\cdot\ldots\cdot b_m^{p_m(n)}: n\in\Z\big\}} ~=~ \overline{b_1^\Z\cdot\ldots\cdot b_m^\Z\Gamma}
$$
is connected. It now follows from \eqref{eqn_claim_2_0} that $\overline{g_1(\Z)\Gamma}$ and $\overline{g_1(\Z)\Gamma_L}$ are connected. But if $\overline{g_1(\Z)\Gamma_L}$ is connected then, because $e(r_2(n))=\eta_2(g_1(n)\Gamma_L)$, as soon as the function $n\mapsto e(r_2(n))$ is non-constant, the average
$$
\lim_{H\to\infty}\frac{1}{H}\sum_{h=1}^H  e\left(r_2(h)\right)
$$
must equal $0$. If this average equals $0$ then \eqref{eqn_claim_2_4} holds, which implies that \eqref{eqn_claim_2_3} holds, and once again we are done. Let us therefore assume that $n\mapsto e(r_2(n))$ is constant. Since $e(r_2(0))=1$, if $n\mapsto e(r_2(n))$ is constant then we must have $e(r_2(n))=1$ for all $n\in\N$. Therefore $g_1(n)$ belongs to the kernel of $\eta_2$ and \eqref{eqn_claim_2_3} becomes
\begin{equation}
\label{eqn_claim_2_5}
\lim_{N\to\infty}\frac{1}{N}\sum_{n=1}^N e\left(r_1(n)\right)=0.
\end{equation}
Since $\overline{g_1(\Z)\Gamma}$ is connected, we once again only have two possibilities: either $ e(r_1(n))$ is constant equal to $1$ or \eqref{eqn_claim_2_5} is satisfied. Since we are done if \eqref{eqn_claim_2_5} is satisfied, the proof of \ref{claim_2_sup-linear} is completed if we can show that $ e(r_1(n))$ cannot be constant equal to $1$ under the current assumptions. 

By way of contradiction, assume $e(r_1(n))=1$ for all $n\in\N$. This implies that $g_1(n)$ belongs to the kernel of $\eta_1$. But we also have that $a_1,\ldots,a_k$ belong to the kernel of $\eta_1$ and $a_i$ for $i\in\mathcal{I}$ as well as $g_1(n)$ for all $n\in\N$ belong to the kernel of $\eta_2$. We claim that having all those elements belong to the kernels of $\eta_1$ and $\eta_2$ contradicts the hypothesis that either $\eta_1$ or $\eta_2$ are non-trivial.

To verify this claim, we first need to make a simplifying assumption. Note that the group generated by $G^\circ$ and $b_1,\ldots,b_m$ is closed and rational\footnote{
Since $X$ is connected, be have $G^\circ \Gamma=G$. Therefore, for every $i=1,\ldots,m$ there exists $\gamma_i\in\Gamma$ such that $b_i\gamma_i\in G^\circ$.
This means that the group generated by $G^\circ$ and $b_1,\ldots,b_m$ equals $G^\circ \Gamma'$, where $\Gamma'$ is the subgroup of $\Gamma$ generated by $\gamma_1,\ldots,\gamma_m$. This proves that the group generated by $G^\circ$ and $b_1,\ldots,b_m$ is both closed and rational.}, and it contains $a_1,\ldots,a_k$ as well as $b_1,\ldots,b_m$. Therefore we can replace $G$ by $\langle G^\circ, b_1,\ldots,b_m\rangle$ if necessary, and will henceforth assume that $G=\langle G^\circ, b_1,\ldots,b_m\rangle$.

Next, define
$$
N\,\coloneqq \,\prod_{i\in\mathcal{I}}a_i^\R \cdot b_1^\Z\cdot\ldots\cdot b_m^\Z [G^\circ,L^\circ]\Gamma_{L^\circ},
$$
where $\Gamma_{L^\circ}=\Gamma\cap L^\circ$.
We claim that $N$ is a normal subgroup of $G$ with a dense subset of rational elements. Once verified, this claim will imply that the closure of $N$, which we denote by $\overline{N}$, is a closed, rational, and normal subgroup of $G$.

To show that $N$ is a group, define
$$
N_0\,\coloneqq \,\prod_{i\in\mathcal{I}}a_i^\R [G^\circ,L^\circ]\Gamma_{L^\circ}\qquad\text{and}\qquad H\,\coloneqq\, b_1^\Z\cdot\ldots\cdot b_m^\Z
$$
and note that $N=HN_0$. Certainly, $H$ is a subgroup of $G$. If we can show that $N_0$ is a normal subgroup of $G$ then it will follow that $N$ is a group.

Since $G=\langle G^\circ, b_1,\ldots,b_m\rangle$, to prove that $N_0$ is normal it suffices to show that $g^{-1}N_0 g=N_0$ for all $g\in G^\circ$, and $b_j^{-1}N_0 b_j=N_0$ for all $j=1,\ldots,m$. It is easy to see that that $g^{-1}N_0 g=N_0$ holds for all $g\in G^\circ$, because $N_0$ contains $[G^\circ, L^\circ]$. To show that $b_j^{-1}N_0 b_j=N_0$ for all $j=1,\ldots,m$ fix some $j$ between $1$ and $m$. Since $b_j$ commutes with $a_i$, we have
$$
b_j^{-1}N_0 b_j\,=\, b_j^{-1} \left(\prod_{i\in\mathcal{I}}a_i^\R [G^\circ,L^\circ]\Gamma_{L^\circ}\right) b_j
\,=\,
\prod_{i\in\mathcal{I}}a_i^\R b_j^{-1}[G^\circ,L^\circ]\Gamma_{L^\circ}b_j.
$$
Since both $G^\circ$ and $L^\circ$ are normal subgroups of $G$, the commutator $[G^\circ,L^\circ]$ is normal and hence
$$
\prod_{i\in\mathcal{I}}a_i^\R b_j^{-1}[G^\circ,L^\circ]\Gamma_{L^\circ}b_j\,=\,\prod_{i\in\mathcal{I}}a_i^\R [G^\circ,L^\circ]b_j^{-1}\Gamma_{L^\circ}b_j.
$$
Since $L^\circ \Gamma_L=\Gamma_L L^\circ=L$ (cf.~\eqref{eqn_claim_2_-1}), there exists $\gamma\in \Gamma_L$ and $c\in L^\circ$ such that $\gamma c=b_j$. Hence
$$
\prod_{i\in\mathcal{I}}a_i^\R [G^\circ,L^\circ]b_j^{-1}\Gamma_{L^\circ}b_j\,=\,\prod_{i\in\mathcal{I}}a_i^\R [G^\circ,L^\circ]c^{-1}\gamma^{-1}\Gamma_{L^\circ}\gamma c.
$$
Since $\Gamma_{L^\circ}$ is a normal subgroup of $\Gamma$, we have $\gamma^{-1}\Gamma_{L^\circ}\gamma=\Gamma_{L^\circ}$. So 
$$
\prod_{i\in\mathcal{I}}a_i^\R [G^\circ,L^\circ]c^{-1}\gamma^{-1}\Gamma_{L^\circ}\gamma c \,=\,\prod_{i\in\mathcal{I}}a_i^\R [G^\circ,L^\circ]c^{-1}\Gamma_{L^\circ}c.
$$
Finally, observe that $c^{-1}\Gamma_{L^\circ}c= \Gamma_{L^\circ}\bmod [G^\circ,L^\circ]$, which gives
$$
\prod_{i\in\mathcal{I}}a_i^\R [G^\circ,L^\circ]c^{-1}\Gamma_{L^\circ}c \,=\,\prod_{i\in\mathcal{I}}a_i^\R [G^\circ,L^\circ]\Gamma_{L^\circ}\,=\,N_0.
$$
This proves that $N_0$ is a normal subgroup of $G$ and hence $N$ is a subgroup of $G$.

Next, let us show that $N$ is normal too. When proving that $N_0$ is normal, we used that $G$ is generated by $G^\circ$ and $b_1,\ldots,b_m$. For the proof that $N$ is normal, this does not seem to be particularly helpful. Instead, we shall use that the set $a_1^\R\cdots a_k^\R b_1^\Z\cdots b_m^\Z\Gamma$ is dense in $G$ (which follows from \eqref{eqn_proof_c_0}). Therefore, to show that $N$ is normal, it suffices to prove that $a_i^{-1}Na_i=N$ for all $i=1,\ldots,k$, $b_j^{-1}N b_j=N$ for all $j=1,\ldots,m$, and $\gamma^{-1}N\gamma=N$ for all $\gamma\in\Gamma$. To verify the first assertion, namely that $a_i^{-1}Na_i=N$, simply note that $N=HN_0$, where $N_0$ is normal and $H$ is a group every element of which commutes with $a_i$. A similar argument shows that $b_j^{-1}N b_j=N$. To see why $\gamma^{-1}N\gamma=N$ holds for all $\gamma\in\Gamma$, simply note that $N=HN_0=\Gamma_L N_0$ and $\Gamma_L$ is a normal subgroup of $\Gamma$.

Finally let us show that rational elements are dense in $N$, or equivalently, that $\overline{N}$ is rational. It is well known (see \cite[Subsection 2.2, p.~203--204]{Leibman05a}) that a closed subgroup of $G$ is rational if and only if its intersection with $\Gamma$ is a uniform subgroup of that group. Hence, to prove that $\overline{N}$ is rational, it suffices to show that $\Gamma\cap\overline{N}$ is a uniform subgroup of $\overline{N}$. However, since $\Gamma_L\subset N$ and $\overline{N}\subset L$, it follows that $\Gamma\cap\overline{N}=\Gamma_L$. Since $\Gamma_L$ is a uniform subgroup of $L$ and $\overline{N}$ is a subgroup of $L$, it follows that $\Gamma_L$ is a uniform subgroup of $\overline{N}$ and we are done.

In conclusion, $\overline{N}$ is a closed and rational subgroup of $L$ that contains $a_i$ for all $i\in\mathcal{I}$ and $b_j$ for all $j=1,\ldots,m$. Moreover, $\overline{N}$ is a normal subgroup of $G$. Since, by definition, $L$ is the smallest subgroup of $G$ with all these properties, we must have
$$
L=\overline{N}.
$$ 

Recall that $a_1,\ldots,a_k$ and $g_1(n)$ for all $n\in\N$ belong to the kernel of $\eta_1$ and $a_i$ for $i\in\mathcal{I}$ and $g_1(n)$ for all $n\in\N$ belong to the kernel of $\eta_2$. From this it follows that $b_1,\ldots,b_m$ also belong to the kernel of both $\eta_1$ and $\eta_2$. Recall also that $[G^\circ,L^\circ]\subset\ker\eta_2$.
In other words $a_1^\R\cdots a_k^\R b_1^\Z\cdots b_m^\Z\Gamma$ is a subset of $\ker(\eta_1)$ and $N$ is a subset of $\ker(\eta_2)$. Since $a_1^\R\cdots a_k^\R b_1^\Z\cdots b_m^\Z\Gamma$ is dense in $G$ it follows that $\eta_1$ is trivial, and since $N$ is dense in $L$, $\eta_2$ is also trivial. This contradicts the fact that either $\eta_1$ or $\eta_2$ is non-trivial and finishes the proof of \ref{claim_2_sup-linear}.
\end{proof}

\begin{proof}[Proof of \ref{claim_3_sup-linear}]
\renewcommand{\qedsymbol}{$\triangle$}
Recall that $\mark{v}_h(n)=\sigma(v^\sq_h(n))$, where  $v^\sq_h(n)= (w_h(n),1_G)\,v^\triangle(n)$,
$$
v(n)= a_1^{f_1(n)}\cdot\ldots\cdot a_k^{f_k(n)}b_1^{p_1(n)}\cdot\ldots\cdot b_m^{p_m(n)},\qquad\forall n\in\N,
$$
and
$$
w_h(n)\,=\,\prod_{i\in\mathcal{I}}a_{i}^{\Delta_h f_i(n)} \cdot\prod_{1\leq j\leq m}b_{j}^{\Delta_h p_j(n)}.
$$
Let $\nu_i$ denote the degree of $f_i$.
Using Taylor's Theorem, we can approximate $\Delta_h f_i(n)$ by
$$
\Delta_h f_i(n)\,=\,hf_i'(n)+\ldots+\frac{h^{(\nu_i-1)}}{(\nu_i-1)!}f_i^{(\nu_i-1)}(n)+\Oh\Big(f^{(\nu_i)}(n)\Big).
$$
In view of \cref{lem_useful_hardy}, we have $\Oh(f^{(\nu_i)}(n))=\oh_{n\to\infty}(1)$. Thus, 
$$
d_G\Big(w_h(n),\,\prod_{i\in\mathcal{I}}a_{i}^{hf_i'(n)+\ldots+\frac{h^{(\nu_i-1)}}{(\nu_i-1)!}f_i^{(\nu_i-1)}(n)} \cdot\prod_{1\leq j\leq m}b_{j}^{\Delta_h p_j(n)}\Big)\,=\,\oh_{n\to\infty}(1).
$$
If $i\in \mathcal{I}$ then $\nu_i\geq 2$. Also, according to the hypothesis of \cref{thm_G}, for every $j\in \{1,\ldots,\nu_i-1\}$ there exists $z(i,j)\in\{1,\ldots,k\}$ such that $f_i^{(j)}=f_{z(i,j)}$. For every $l\in\{1,\ldots,k\}$ define $Q_l\coloneqq\{(i,j): z(i,j)=l\}$ and set
$$
\tilde{u}_l(h)\,\coloneqq\, \prod_{(i,j)\in Q_l} a_i^{\frac{h^j}{j!}}.
$$
Now define, for every $i\in\{1,\ldots,k\}$, the polynomial mapping $u_i\colon\N\to \mark{G}^\circ$ as
\begin{equation*}
u_i(h)\,\coloneqq\,
\begin{cases}
\sigma(a_i,a_i),&\text{if}~i\notin\mathcal{I},\\
\sigma(a_i \tilde{u}_i(h),a_i),&\text{if}~i\in\mathcal{I}.
\end{cases}
\end{equation*}
In a similar way, one can find $c, e_1,\ldots,e_m\colon \N\to \mark{G}$. 
\end{proof}

\begin{proof}[Proof of \ref{claim_4_sup-linear}]
\renewcommand{\qedsymbol}{$\triangle$}
Fix $h\in\N$. Since $G_\bullet =\{G_1,G_2,\ldots, G_{d},G_{d+1}\}$ is filtration that realizes the degree of $v$, we have $b_j\in G_{\deg(p_j)+1}$ for all $j=1,\ldots,m$ and $a_i\in G_{\deg(f_i)}^\circ$ for all $i=1,\ldots,k$.
Define $L_i\coloneqq L\cap G_{i+1}$, $i=0,1,\ldots,d$, $L_{d+1}\coloneqq \{1_G\}$, and
$$
(G\plh_L G)_\bullet=\big\{G_1\plh_{L_1}G_1,\ \ldots,\ G_d\plh_{L_d}G_d,\ G_{d+1}\plh_{L_{d+1}}G_{d+1}\big\}.
$$
According to \cref{lem_filtration-of-Gx_LG}, $(G\plh_L G)_\bullet$ is a $d$-step filtration of $G\plh_L G$.
We claim that $(G\plh_L G)_\bullet$ is a filtration that realizes the degree of $v_h^\sq(n)$. 
Recall that
$$
v_h^\sq(n)\, =\, (w_h(n),1_G)\, v^\triangle(n),
$$
where
$$
v^\triangle(n)\,=\, \prod_{1\leq i\leq k} (a_i,a_i)^{f_i(n)} \cdot \prod_{1\leq j\leq m}(b_j,b_j)^{p_j(n)}
$$
and
$$
(w_h(n),1_G)\,=\, \prod_{i\in\mathcal{I}}(a_{i},1_G)^{\Delta_h f_i(n)} \cdot\prod_{1\leq j\leq m}(b_{j},1_G)^{\Delta_h p_j(n)}.
$$
To show that $(G\plh_L G)_\bullet$ is a filtration realizing the degree of $v_h^\sq(n)$, we must prove that
\begin{enumerate}
[label=(\roman{enumi}),ref=(\roman{enumi}),leftmargin=*]
\item
\label{itm_claim4_i}
$(b_j,b_j)\in G_{\deg(p_j)+1}\plh_{L_{\deg(p_j)+1}} G_{\deg(p_j)+1}$ for all $j=1,\ldots,m$;
\item
\label{itm_claim4_ii}
$(a_i,a_i)\in (G_{\deg(f_i)}\plh_{L_{\deg(f_i)}} G_{\deg(f_i)})^\circ$ for all $i=1,\ldots,k$;
\item
\label{itm_claim4_iii}
$(b_j,1_G)\in G_{\deg(\Delta_hp_j)+1}\plh_{L_{\deg(\Delta_h p_j)+1}} G_{\deg(\Delta_hp_j)+1}$ for all $j=1,\ldots,m$;
\item
\label{itm_claim4_iv}
$(a_i,1_G)\in (G_{\deg(\Delta_hf_i)}\plh_{L_{\deg(\Delta_hf_i)}} G_{\deg(\Delta_hf_i)})^{\circ}$ for all $i\in\mathcal{I}$;
\end{enumerate}
Parts \ref{itm_claim4_i} and \ref{itm_claim4_ii} follow from the fact that $G_i^\triangle \subset G_{i}\plh_{L_{i}} G_{i}$ for all $i=1,\ldots,d$ and that $b_j\in G_{\deg(p_j)+1}$ for all $j=1,\ldots,m$ and $a_i\in G_{\deg(f_i)}^\circ$ for all $i=1,\ldots,k$.
Parts \ref{itm_claim4_iii} and \ref{itm_claim4_iv} follow from the fact that $\deg(\Delta_h f_i)=\deg(f_i)-1$ and $\deg(\Delta_h p_j)=\deg(p_j)-1$ and that $(L\cap G_{i+1})\times\{1_G\}\subset G_{i}\plh_{L_{i}} G_{i}$.

To complete the proof of \ref{claim_4_sup-linear}, note that if $(G\plh_L G)_\bullet$ is a filtration realizing the degree of $v_h^\sq(n)$, then the filtration $\mark{G}_\bullet =\{\mark{G}_1,\mark{G}_2,\ldots, \mark{G}_{d}\}$, defined as
$$
\mark{G}_i\coloneqq \sigma\big(G_i\plh_{L_i}G_i\big),\qquad i=1,\ldots,d,
$$
is a filtration realizing the degree of $\mark{v}_h(n)$. Moreover, $G_d\plh_{L_d}G_d$ is equal to $G_d^\triangle$ because $L_d=\{1_G\}$, and hence $G_d\plh_{L_d}G_d$ belongs to the kernel of $\sigma$. This shows that $\mark{G}_\bullet =\{\mark{G}_1,\mark{G}_2,\ldots, \mark{G}_{d}\}$ is a $(d-1)$-step filtration and hence $\mark{v}_h(n)$ has degree $d-1$.
\end{proof}
This finishes the proofs of Claims 1, 2, 3, and 4, which in turn completes the proof of \cref{thm_G}.
\end{proof}

\section{Theorems \ref{thm_D} and \ref{thm_E}}
\label{sec_proof_thms_D_and_E}

For proving Theorems \ref{thm_D} and \ref{thm_E} we use essentially the same ideas as were used in the proofs of Theorems \ref{thm_B} and \ref{thm_C}, only that all \Cesaro{} averages get replaced with $W$-averages.
Similar to what we did in \cref{sec_reduction}, the first step is to reduce Theorems \ref{thm_D} and \ref{thm_E} to the following analogue of \cref{thm_G}.

\begin{Maintheorem}
\label{thm_H}
Let $G$ be a simply connected nilpotent Lie group, $\Gamma$ a uniform and discrete subgroup of $G$, and $\Hardy$ a Hardy field. Assume $v\colon\N\to G$ is a mapping of the form
$$
v(n)= a_1^{f_1(n)}\cdot\ldots\cdot a_k^{f_k(n)}b_1^{p_1(n)}\cdot\ldots\cdot b_m^{p_m(n)},\qquad\forall n\in\N,
$$
where $a_1,\ldots, a_k\in G^\circ$, $b_1,\ldots,b_m\in G$, the elements $a_1,\ldots,a_k,b_1,\ldots,b_m$ are pairwise commuting, $\overline{b_1^\Z\Gamma},\ldots, \overline{b_m^\Z\Gamma}$ are connected sub-nilmanifolds of $X=G/\Gamma$, $p_1,\ldots,p_m\in\R[t]$ are polynomials satisfying
\begin{enumerate}
[label=~(\ref{thm_H}$_\arabic{enumi}$),ref=(\ref{thm_H}$_\arabic{enumi}$),leftmargin=*]
\item
\label{itm_thm_H_A}
$p_j(\Z)\subset \Z$, for all $j=1,\ldots,m$,
\item
\label{itm_thm_H_B}
$\deg(p_j)=j$, for all $j=1,\ldots,m$,
\end{enumerate}
and $f_1, \ldots, f_k\in\Hardy$ satisfy
\begin{enumerate}
[label=~(\ref{thm_H}$_\arabic{enumi}$),ref=(\ref{thm_H}$_\arabic{enumi}$),leftmargin=*]
\setcounter{enumi}{2}
\item
\label{itm_thm_H_C}
$f_1(t)\prec \ldots\prec f_k(t)$,
\item
\label{itm_thm_H_D}
for all $f\in\{f_1,\ldots,f_k\}$ there exists $\ell\in\N$ such that $t^{\ell-1}\log(W(t))\prec f(t) \prec t^\ell$,
\item
\label{itm_thm_H_E}
for all $f\in\{f_1,\ldots,f_k\}$ with $\deg(f)\geq 2$ we have $f'\in\{f_1,\ldots,f_k\}$,
\end{enumerate}
where $W\in\Hardy$ has degree $1$. 
Then $(v(n)\Gamma)_{n\in\N}$ is uniformly distributed with respect to $W$-averages in the sub-nilmanifold $\overline{ a_1^\R\cdots a_k^\R b_1^\Z\cdots b_m^\Z \Gamma}$.
\end{Maintheorem}

The proof that \cref{thm_H} implies Theorems \ref{thm_E} and \ref{thm_D} is almost identical to the proof that \cref{thm_G} implies Theorems \ref{thm_B} and \ref{thm_C} given in \cref{sec_reduction}. The only difference is that instead of applying \cref{cor_normal_form} with $V(t)=t$, we apply \cref{cor_normal_form} with $V(t)=W(t)$ where $W\in\Hardy$ is chosen (using \cref{cor_finding_W}) such that $f_1,\ldots,f_k$ satisfy \ref{property_P_W}. Since these proofs are so similar, we omit the details.

The proof of \cref{thm_H} is, just like the proof of \cref{thm_G}, split into three cases: the abelian case, the sub-linear case, and the general case.
The proof of the sub-linear case of \cref{thm_H} is the same as the proof of the sub-linear case of \cref{thm_G}, except that all \Cesaro{} averages are replaced with $W$-averages and instead of utilizing \cref{prop_slow_vdC} one uses \cref{prop_slow_vdC_W-averages}, which is precisely the analogue of \cref{prop_slow_vdC} for $W$-averages. Therefore, we omit the details of this part of the proof of \cref{thm_H} too.

Similarly, the arguments used in the proof of the general case of \cref{thm_H} are almost identical to the ones used in the proof of the general case of \cref{thm_G} in \cref{sec_sup-linear-case} if one replaces all \Cesaro{} averages with $W$-averages and instead of applying \cref{prop_vdC} with $p_n=1$ one applies \cref{prop_vdC} with $p_n=w(n)$. We omit the details of this part as well.

This leaves only the abelian case of \cref{thm_H} to be verified. For the proof of this case we can also copy the proof of the abelian case of \cref{thm_G} given in \cref{sec_abelian-case}. The only missing ingredient is a variant of Boshernitzan's Equidistribution Theorem (\cite[Theorem 1.8]{Boshernitzan94}) for $W$-averages. Let us formulate and prove such a variant now.

\begin{Theorem}
\label{thm_bosh_W-averages}
Let $\Hardy$ be a Hardy field, let $W,f\in\Hardy$, and assume $1\prec W(t)\ll t$ and $t^{\ell-1} \log(W(t))\prec f(t)\prec t^\ell$ for some $\ell\in\N$. Then
\begin{equation}
\label{eqn_bosh_W-averages_1}
\lim_{N\to\infty}\frac{1}{W(n)}\sum_{n=1}^N w(n)\, e(f(n))\,=\,0,
\end{equation}
where $w=\Delta W$.
\end{Theorem}

Although \cref{thm_bosh_W-averages} for $W(t)=t$ does not imply Boshernitzan's Equidistribution Theorem in full generality, it is good enough for the proof of the abelian case of \cref{thm_H}.

\begin{proof}[Proof of \cref{thm_bosh_W-averages}]
For the proof we use induction on $\ell$. For the base case of the induction, when $\ell=1$, we use \cref{prop_slow_vdC_W-averages}. In light of \cref{prop_slow_vdC_W-averages}, instead of  \eqref{eqn_bosh_W-averages_1} it suffices to show that for every $\epsilon>0$ there exists $\xi\in(0,\epsilon)$ such that
\begin{equation}
\label{eqn_bosh_W-averages_2}
\lim_{H\to\infty}\lim_{N\to\infty}\frac{1}{H}\sum_{h=1}^H \frac{1}{W(N)}\sum_{n=1}^N w(n)e(f(n)+\xi h)\overline{e(f(n))}~=~0.
\end{equation}
We can simplify the left hand side of \eqref{eqn_bosh_W-averages_2} to
\begin{equation}
\begin{split}
\lim_{H\to\infty}\lim_{N\to\infty}\frac{1}{H} \sum_{h=1}^H \frac{1}{W(N)} & \sum_{n=1}^N w(n)e(f(n)+\xi h)\overline{e(f(n))}
\\
=&
\lim_{H\to\infty}\lim_{N\to\infty}\frac{1}{H}\sum_{h=1}^H \frac{1}{W(N)}\sum_{n=1}^N w(n)e(\xi h)
\\
=&\lim_{H\to\infty}\frac{1}{H}\sum_{h=1}^H e(\xi h),
\end{split}
\end{equation}
which equals $0$ for all $\xi\notin\Z$.

For the proof of the inductive step, we use \cref{prop_vdC}.
In view of \cref{prop_vdC}, applied with $P_N=W(N)$ and $p_n=w(n)$, we see that \eqref{eqn_bosh_W-averages_1} holds if we can show
\begin{equation}
\label{eqn_bosh_W-averages_3}
\lim_{H\to\infty}\lim_{N\to\infty} \frac{1}{H}\sum_{h=1}^H\frac{1}{W(n)}\sum_{n=1}^N w(n)\, e(f(n+h)-f(n))~=~0.
\end{equation}
However, since $f$ satisfied $t^{\ell-1} \log(W(t))\prec f(t)\prec t^\ell$ and $\ell\geq 2$, the function $\Delta_h f(t)\coloneqq f(t+h)-f(t)$ satisfies $t^{\ell-2} \log(W(t))\prec \Delta_hf(t)\prec t^{\ell-1}$, which is a consequence of \cref{lem_useful_hardy}. Hence \eqref{eqn_bosh_W-averages_3} follows form the induction hypothesis. 
\end{proof}

\section{A variant of van der Corput's Lemma}
\label{sec_slow_vdC}

The purpose of this section is to prove the following proposition which was used in the proofs of \cref{thm_G_sub-linear} in \cref{sec_sub-linear-case} and \cref{thm_H} in \cref{sec_proof_thms_D_and_E}.

\begin{Proposition}
\label{prop_slow_vdC_W-averages}
Assume $W$ and $f_1,\ldots,f_k$ are functions from a Hardy field $\Hardy$ satisfying $1\prec W(t) \ll t$ and $\log(W(t))\prec f_1(t)\prec\ldots\prec f_k\prec t$.
Let $\Psi\colon\R^k\to \C$ be a bounded and uniformly continuous function and suppose for all $s\in\R$ the limit
\begin{equation}
\label{eqn_conclusion_of_prop_slow_vdC_W-averages-2}
A(s)~\coloneqq~\lim_{N\to\infty}
\frac{1}{W(N)}\sum_{n=1}^N w(n)\Psi(f_1(n), \ldots,f_{k-1}(n), f_k(n)+s)\overline{\Psi(f_1(n), \ldots,f_{k}(n))}
\end{equation}
exists, where $w=\Delta W$.
If for every $\epsilon>0$ there exists $\xi\in(0,\epsilon)$ such that
\begin{equation*}
\lim_{H\to\infty}\frac{1}{H}\sum_{h=1}^H A(\xi h)~=~0
\end{equation*}
then necessarily
\begin{equation}
\label{eqn_conclusion_of_prop_slow_vdC_W-averages}
\lim_{N\to\infty}\frac{1}{W(N)}\sum_{n=1}^N w(n)\Psi(f_1(n),\ldots,f_k(n))=0.
\end{equation}
\end{Proposition}

Note that \cref{prop_slow_vdC} follows from \cref{prop_slow_vdC_W-averages} by choosing $W(t)=t$.

The next lemma will be useful for the proof of \cref{prop_slow_vdC_W-averages}.
We say a function $f\colon [1,\infty) \to\R$ has \define{sub-exponential growth} if $|f(t)|\prec c^t$ for all $c>1$.

\begin{Lemma}
\label{lem_subexponential_growth}
Let $f\in\Hardy$ with $\log(t)\prec f(t)$.
Then $f^{-1}$ has sub-exponential growth.
\end{Lemma}

\begin{proof}
Note that $f^{-1}(t)$ has sub-exponential growth if and only if
$$
\lim_{t\to\infty}\frac{f^{-1}(t+1)}{f^{-1}(t)}=1.
$$ 
Let us therefore consider the number $c_0\coloneqq \lim_{n\to\infty}{f^{-1}(n+1)}/{f^{-1}(n)}$. Note that this limit exists because if $f$ belongs to some Hardy Field, then so does $f^{-1}$.

Since $f^{-1}$ is eventually increasing, we have $c_0\geq 1$. It remains to show that $c_0\leq 1$, which we will do by showing that $c\leq 1$ for all $c<c_0$.
Thus, fix any $c$ with $0<c<c_0$. There exists $M>0$ such that for all but finitely many $n$ we have
$$
f^{-1}(n)\geq M c^n
$$
and hence, using $f(f^{-1}(n))=n$, we obtain
$$
f(M c^n) \leq n.
$$
Since $\log(t)\prec f(t)$, we conclude that $\log(M c^n) \prec  n$ and hence $c\leq 1$.
\end{proof}

\begin{Remark}
\label{rem_subexponential_growth}
It follows from \cref{lem_subexponential_growth} that if $W,f\in\Hardy$ with $1\prec W(t)$ and $\log(W(t))\prec f(t)$ then $W\circ f^{-1}$ has sub-exponential growth. \end{Remark}

\begin{Lemma}
\label{lem_properties_of_weights_coming_from_approximate inverse}
Let $W,f\in\Hardy$ with $1\prec W(t)\ll 1$ and $\log(W(t))\prec f(t)\prec t$, and define $w\coloneqq \Delta W$.
For every $n\in \N$ and $\xi\in(0,1]$ define 
$$
K_n\coloneqq \big\{j\in\N: f(j)\in \left(\xi n,\xi(n+1)\right]\big\}.
$$
Define $g(t)\coloneqq \xi^{-1}f(t)$, $p_n\coloneqq \sum_{i\in K_n} w(i)$, and $P_N\coloneqq \sum_{n=1}^N p_n$. Then the following hold:
\begin{enumerate}
[label=(\roman{enumi}),ref=(\roman{enumi}),leftmargin=*]
\item\label{itm_Kset_1}
$\lim_{n\to\infty}\frac{W(g^{-1}(n+1))-W(g^{-1}(n))}{p_{n}}=1$.
\item\label{itm_Kset_2}
$\lim_{n\to\infty} \frac{W(g^{-1}(n))}{P_n}=1$;
\item\label{itm_Kset_3}
$\lim_{n\to\infty} P_n=\infty$;
\item\label{itm_Kset_4}
$\lim_{n\to\infty}\frac{p_{n}}{P_n}=0$.
\end{enumerate}
\end{Lemma}

\begin{proof}
Since $1\prec f(t)\prec t$ and $f$ is eventually monotone increasing, for sufficiently large $n$ the set $K_{n}$ is an interval of the form $[a_n,b_n]\subset \N$, where $\lim_{n\to\infty} b_n-a_n=\infty$. For all such $n$ we thus also have $p_n=W(b_n+1)-W(a_n)=W(b_n)-W(a_n)+\oh_{n\to\infty}(1)$.
Let $s_n\coloneqq \min\{t\in\R: f(t)\geq \xi n\}$. Then for all but finitely many $n$ we have that $s_n=g^{-1}(n)$.
Note that the difference between $a_n$ and $s_n$ can be bounded from above by $\Delta f(n)$, and since $\Delta f(n)\to 0$ as $n\to\infty$, we have that $\lim_{n\to\infty} a_n-s_n =0$. Similarly, we can show that $\lim_{n\to\infty} b_n-s_{n+1} =0$. Therefore $\lim_{n\to\infty} W(a_n)-W(g^{-1}(n)) =\lim_{n\to\infty} W(b_n)-W(g^{-1}(n+1))=0$ and hence
$$
\lim_{n\to\infty}\frac{W(g^{-1}(n+1))-W(g^{-1}(n))}{p_{n}}=
\lim_{n\to\infty}\frac{W(g^{-1}(n+1))-W(g^{-1}(n))}{W(b_{n})-W(a_n)}=1.
$$

Part \ref{itm_Kset_2} follows straight away from \ref{itm_Kset_1} and part \ref{itm_Kset_3} follows immediately from part \ref{itm_Kset_2} and the fact that $W(g^{-1}(t))\to\infty$.

For the proof of part \ref{itm_Kset_4} note that 
$$
\lim_{n\to\infty}p_n/P_n = \lim_{n\to\infty} \frac{W(g^{-1}(n+1))-W(g^{-1}(n))}{W(g^{-1}(n))}
$$
because of parts \ref{itm_Kset_1} and \ref{itm_Kset_2}.
However, $\lim_{n\to\infty} (g^{-1}(n+1)-g^{-1}(n))/g^{-1}(n)=0$ because $W(g^{-1}(t))$ has sub-exponential growth due to \cref{rem_subexponential_growth} and the fact that $\log(W(t))\prec g(t)$.
\end{proof}

\begin{proof}[Proof of \cref{prop_slow_vdC_W-averages}]
Fix $\xi\in(0,1]$ and define $K_{j}\coloneqq \left\{n\in\N: f_k(n)\in \left(\xi(j-1),\xi j\right]\right\}$ and $g(t)\coloneqq \xi^{-1}f_k(t)$.
Since
\begin{equation*}
\begin{split}
\frac{1}{W(N)}\sum_{n=1}^N & w(n)\Psi(f_1(n),\ldots,f_k(n))
\\
&=\frac{1}{W(N)}\sum_{j=1}^{\lfloor g(N)\rfloor} \sum_{n\in K_{j}} w(n)\Psi(f_1(n),\ldots,f_k(n))\,+\,\oh_{N\to\infty}(1),
\end{split}
\end{equation*}
instead of \eqref{eqn_conclusion_of_prop_slow_vdC_W-averages} is suffices to show that
\begin{equation}
\label{eqn_conclusion_of_prop_slow_vdC_2}
\lim_{N\to\infty}\frac{1}{W(N)}\sum_{j=1}^{\lfloor g(N)\rfloor} \sum_{n\in K_{j}} w(n)\Psi(f_1(n),\ldots,f_k(n))~=~0.
\end{equation}
Set $p_j\coloneqq \sum_{n\in K_j} w(n)$ and $P_J\coloneqq \sum_{j=1}^J p_j$.
According to \cref{lem_properties_of_weights_coming_from_approximate inverse}, part \ref{itm_Kset_2}, we have
$$
\lim_{N\to\infty}
\frac{W(N)}{P_{\lfloor g(N)\rfloor}}~=~ 1.
$$
Therefore, \eqref{eqn_conclusion_of_prop_slow_vdC_2} is equivalent to
\begin{equation*}
\label{eqn_conclusion_of_prop_slow_vdC_3}
\lim_{N\to\infty}\frac{1}{P_{\lfloor g(N)\rfloor}}\sum_{j=1}^{\lfloor g(N)\rfloor} \sum_{n\in K_{j}}w(n) \Psi(f_1(n),\ldots,f_k(n))~=~0.
\end{equation*}
which we can write as
\begin{equation}
\label{eqn_conclusion_of_prop_slow_vdC_4}
\lim_{J\to\infty}\frac{1}{P_{J}}\sum_{j=1}^{J} \sum_{n\in K_{j}} w(n)\Psi(f_1(n),\ldots,f_k(n))~=~0.
\end{equation}
Define $g_i(n)\coloneqq f_i(g^{-1}(n))$ for $i=1,\ldots,k$ and note that $g_k(n)=\xi n$.
Then $\sup_{n\in K_j}|g_k(j)-f_k(n)|=\Oh(\xi)$ and, for all $i< k$, we have $\sup_{n\in K_j}|g_i(j)-f_i(n)|=\oh_{n\to\infty}(1)$. Therefore, using the uniform continuity of $\Psi$, we have that
$$
\sum_{n\in K_{j}} w(n) \Psi(f_1(n),\ldots,f_k(n))= p_j \Psi\big(g_1(j),\ldots,g_k(j)\big)+ \oh_{j\to\infty,\xi\to 0}(p_j).
$$
It follows that
$$
\frac{1}{P_J}\sum_{j=1}^{J}\, \sum_{n\in K_{j}} w(n)\Psi(f_1(n),\ldots,f_k(n))
=
\frac{1}{P_J}\sum_{j=1}^{J} p_j \Psi(g_1(j),\ldots,g_k(j)) + \oh_{J\to\infty,\xi\to 0}(1).
$$
In conclusion, \eqref{eqn_conclusion_of_prop_slow_vdC_4}, and therefore also \eqref{eqn_conclusion_of_prop_slow_vdC_W-averages}, are equivalent to 
\begin{equation}
\label{eqn_conclusion_of_prop_slow_vdC_51}
\lim_{J\to\infty}\frac{1}{P_J}\sum_{j=1}^{J} p_j\Psi(g_1(j),\ldots,g_k(j))~=~\oh_{\xi\to 0}(1).
\end{equation}
Using essentially the same argument one can also show that $A(s)$, which was defined in \eqref{eqn_conclusion_of_prop_slow_vdC_W-averages-2}, is given by
\begin{equation}
\label{eqn_conclusion_of_prop_slow_vdC_52}
A(s)\,=\,\lim_{J\to\infty}\frac{1}{P_J}\sum_{j=1}^{J} p_j\Psi(g_1(j), \ldots,g_{k-1}(j), g_k(j)+s)\overline{\Psi(g_1(j), \ldots,g_{k}(j))}.
\end{equation}
According to \cref{prop_vdC}, instead of \eqref{eqn_conclusion_of_prop_slow_vdC_51}, it is enough to prove that
\begin{equation}
\label{eqn_conclusion_of_prop_slow_vdC_61-0}
\lim_{H\to\infty} \frac{1}{H}\sum_{h=1}^H\left(\lim_{J\to\infty}\frac{1}{P_J}\sum_{j=1}^J p_j \Psi(g_1(j+h),\ldots,g_k(j+h))\overline{\Psi(g_1(j),\ldots,g_k(j))}\right)= \oh_{\xi\to 0}(1).
\end{equation}
For $i< k$ we have $g_i(t)\prec t$ and hence $g_i(j+h)=g_i(j)+\oh_{j\to\infty}(1)$.
For $i=k$ we have $g(j+h)=g(j)+\xi$.
Therefore the left hand side of \eqref{eqn_conclusion_of_prop_slow_vdC_61-0} can be replaced with
\begin{equation*}
\lim_{H\to\infty} \frac{1}{H}\sum_{h=1}^H\left(\lim_{J\to\infty}\frac{1}{P_J}\sum_{j=1}^J p_j \Psi(g_1(j),\ldots,g_{k-1}(j), g_k(j)+\xi h)\overline{\Psi(g_1(j),\ldots,g_k(j))}\right),
\end{equation*}
which in combination with \eqref{eqn_conclusion_of_prop_slow_vdC_52} shows that \eqref{eqn_conclusion_of_prop_slow_vdC_61-0} is equivalent to
$$
\lim_{H\to\infty} \frac{1}{H}\sum_{h=1}^H A(\xi h)= \oh_{\xi\to 0}(1).
$$
This finishes the proof.
\end{proof}

\appendix
\section{Appendix}

\subsection{Some basic results regarding functions form a Hardy field}

\begin{Lemma}
\label{lem_range_of_Hardy_functions}
Let $\Hardy$ be a Hardy field and let $f\in\Hardy$ be of polynomial growth. If $f(\N)\subset\Z$ then $f\in\R[t]$.
\end{Lemma}

\begin{proof}
For $k\in\N$, let $\Delta^k f$ denote the $k$-fold finite difference of $f$, that is, $\Delta^0 f(n)=f(n)$, $\Delta^1 f(n)=\Delta f(n)=f(n+1)-f(n)$, $\Delta^2 f(n)= \Delta f(n+1) - \Delta f(n)= f(n+2)-2f(n+1)+f(n)$, and so on. If $d$ is the degree of $f$ then the function $\Delta^d f$ has degree $0$. Moreover, since $f(\N)\subset \Z$, we have $\Delta^d f(n)\in\Z$ for all $n\in\N$. However, the only function from a Hardy field that has degree $0$ and for which all its values belong to $\Z$ is a constant function. That means that $\Delta^d f(n)=c$ for some $c\in \Z$, which  implies that $f$ is a polynomial of degree $d$ satisfying $f(\N)\subset \Z$.
\end{proof}

\begin{Lemma}
\label{lem_dom_of_group_lements}
Let $G$ be a simply connected nilpotent Lie group, $a\in G$, $\Hardy$ a Hardy field, and $f\in\Hardy$ of polynomial growth. If $f(\N)\subset \dom(a)$ then one of the following two cases holds:
\begin{enumerate}
[label=(\roman{enumi}),ref=(\roman{enumi}),leftmargin=*]
\item
either $a\in G^\circ$;
\item\label{itm_part_ii_dom}
or there exist $m\in \N$ and $p\in\R[t]$ with $p(\Z)\subset \Z$ such that $\frac{1}{m}\in\dom(a)$ and $f(n)=p(n)/m$ for all $n\in\N$.
\end{enumerate}
\end{Lemma}

\begin{proof}
Suppose $a$ is not an element of $G^\circ$. This means there exists $m\in\N$ such that $\dom(a)=\frac{1}{m}\Z$. 
In particular, $f(\N)\subset \frac{1}{m}\Z$. By \cref{lem_range_of_Hardy_functions} the function $p(n)\coloneqq m f(n)$ is polynomial with $p(\N)\subset\Z$. This finishes the proof.
\end{proof}

Define $\S(f_1,\ldots,f_k)=\{\lambda_1 f_1+\ldots+\lambda_k f_k+p: \lambda_1,\ldots,\lambda_k\in\R,~p\in\R[t]\}$.

\begin{Lemma}
\label{lem_simple_normal_form}
Let $\Hardy$ be a Hardy field and assume $f_1,\ldots,f_k\in \Hardy$ have polynomial growth.
Then there exist $m\in\N$, $g_1,\ldots,g_m\in\S(f_1,\ldots,f_k)$, $p_1,\ldots,p_k\in\R[t]$, and $\lambda_{1,1},\ldots,\lambda_{k,m}\in\R$ with the following properties:
\begin{enumerate}
[label=(\arabic{enumi}),ref=(\arabic{enumi}),leftmargin=*]
\item
\label{itm_s1}
$g_1(t)\prec \ldots\prec g_m(t)$;
\item
\label{itm_s2}
for all $g\in\{g_1,\ldots,g_m\}$ either $g=0$ or there exists $\ell\in\N$ such that $t^{\ell-1}\prec g(t) \prec t^\ell$;
\item
\label{itm_s3}
for all $i\in\{1,\ldots,k\}$,
$$
\lim_{t\to\infty}\Bigg|f_i(t)-\sum_{j=1}^m \lambda_{i,j} g_j(t)-p_i(t)\Bigg|=0.
$$
\end{enumerate}
\end{Lemma}

\begin{proof}
Let us associate to every finite set of functions $h_1,\ldots,h_r\in\Hardy$ of polynomial growth a pair $(d,e)\in(\N\cup\{0\})\times \N$, which we will call the \define{characteristic pair associated to $\{h_1,\ldots,h_r\}$}, in the following way: The number $d$ is the maximal degree among degrees of functions in $\{h_1,\ldots,h_r\}$, i.e.,
$$
d\, =\, \max\big\{\deg(h_i): 1\leq i\leq r\big\},
$$
and the number $e$ equals the number of functions in $\{h_1,\ldots,h_r\}$ whose degree is $d$, i.e.,
$$
e\, =\, \big|\big\{i\in \{1,\ldots,r\}: \deg(h_i)=d \big\}\big|.
$$
Using this notion of a characteristic pair, we can define a partial ordering on the set of finite subsets of functions in $\Hardy$ of polynomial growth: Given $h_1,\ldots,h_r,\markfour{h}_1,\ldots,\markfour{h}_{\markfour{r}}\in\Hardy$ of polynomial growth we write $\{h_1,\ldots,h_r\}\prec \{\markfour{h}_1,\ldots,\markfour{h}_{\markfour{r}}\}$ if
\begin{itemize}
\item
either $d<\markfour{d}$,
\item
or $d=\markfour{d}$ and $e<\markfour{e}$,
\end{itemize}
where $(d,e)$, $(\markfour{d},\markfour{e})$ are the characteristic pairs associated to $\{h_1,\ldots,h_r\}$ and $\{\markfour{h}_1,\ldots,\markfour{h}_{\markfour{r}}\}$ respectively.

Recall, our goal is to show for any $f_1,\ldots,f_k\in\Hardy$ of polynomial growth there exist $m\in\N$, $g_1,\ldots,g_m\in\S(f_1,\ldots,f_k)$, $p_1,\ldots,p_k\in\R[t]$, and $\lambda_{1,1},\ldots,\lambda_{k,m}\in\R$ such that properties \ref{itm_s1}, \ref{itm_s2}, and \ref{itm_s3} are satisfied. 
To accomplish this goal, we will use induction on the just defined partial ordering.

The base case of this induction corresponds to $(d,e)=(0,e)$ for some $e\in\N$. In this case we have $k=e$. Let $m\coloneqq 1$, $c_i\coloneqq \lim_{t\to\infty} f_i(t)$, $g_1(t)=0$, $p_i(t)\coloneqq c_i$, and $\lambda_{i,1}\coloneqq 0$. With this choice, \ref{itm_s1}, \ref{itm_s2}, and \ref{itm_s3} are satisfied, and we are done.

Next, suppose we are in the case when the characteristic pair associated to $\{f_1,\ldots,f_k\}$ is of the form $(d,e)$ and $d\geq 1$. 
Define $\sigma_i\coloneqq \lim_{t\to\infty}f_i(t)/t^d$ and set
$$
h_i(t)\coloneqq f_i(t)-\sigma_i t^d,\qquad i=1,\ldots,k.
$$
This yields a new collection of functions $\{h_1,\ldots,h_k\}$ with the property that $h_i(t)\prec t^d$ for all $i=1,\ldots,k$.
We now distinguish two cases, the case when the characteristic pair of $\{h_1,\ldots,h_k\}$ is the same as the characteristic pair of $\{f_1,\ldots,f_k\}$, and the case when $\{h_1,\ldots,h_k\}\prec \{f_1,\ldots,f_k\}$.

If we are in the first case then there exists some function in $\{h_1,\ldots,h_k\}$ of degree $d$. By relabeling $h_{1},\ldots,h_k$ if necessary, we can assume without loss of generality that $h_{1} \ll \ldots \ll h_k$. Then $h_k$ has degree $d$. Define $\eta_i\coloneqq \lim_{t\to\infty} h_{i}(t)/h_k(t)$ and set $\markfour{h}_i\coloneqq h_i-\eta_i h_k$ for all $ i=1,\ldots,k-1$.
It is straightforward to check that $\markfour{h}_1,\ldots,\markfour{h}_{k-1}\in\Hardy$ satisfies $\{\markfour{h}_1,\ldots,\markfour{h}_{k-1}\}\prec \{h_1,\ldots,h_k\}$. 
Therefore, by the induction hypothesis, we can find $\markfour{m}\in\N$, $\markfour{g}_1,\ldots,\markfour{g}_{\markfour{m}}\in\S(\markfour{h}_1,\ldots,\markfour{h}_{k-1})$, $\markfour{p}_1,\ldots,\markfour{p}_{k-1}\in\R[t]$, and $\markfour{\lambda}_{1,1},\ldots,\markfour{\lambda}_{k-1,\markfour{m}}\in\R$ such that $\markfour{g}_1,\ldots,\markfour{g}_{\markfour{m}}$ satisfy properties \ref{itm_s1} and \ref{itm_s2}, and for all $i\in\{1,\ldots,k-1\}$ we have
$$
\lim_{t\to\infty}\Bigg|\markfour{h}_i(t)- \sum_{j=1}^{\markfour{m}} \markfour{\lambda}_{i,j} \markfour{g}_j(t)-\markfour{p}_i(t)\Bigg|=0.
$$
Define $m=\markfour{m}+1$, let $p_i\coloneqq \markfour{p}_i+\sigma_i t^d$, and set
\begin{equation*}
\lambda_{i,j}=
\begin{cases}
\markfour{\lambda}_{i,j},&\text{if}~i<k~\text{and}~j<m
\\
\eta_i,&\text{if}~i<k~\text{and}~j=m
\\
0,&\text{if}~i=k~\text{and}~j<m
\\
1,&\text{if}~i=k~\text{and}~j=m
\end{cases},
\quad
\text{and}\quad
g_{j}=
\begin{cases}
\markfour{g}_{j},&\text{if}~j<m
\\
h_k,&\text{if}~j=m
\end{cases}.
\end{equation*}
Then, for $i\in\{1,\ldots,k-1\}$ we have
\begin{eqnarray*}
f_i(t)
&=&\markfour{h}_i(t) + \eta_i h_k(t) + \sigma_i t^d+\oh_{t\to\infty}(1)
\\
&=& \sum_{j=1}^{\markfour{m}} \markfour{\lambda}_{i,j} \markfour{g}_j(t)+ \markfour{p}_i(t)+\eta_i h_k(t)+ \sigma_i t^d +\oh_{t\to\infty}(1)
\\
&=&\sum_{j=1}^{m-1} \lambda_{i,j} g_j(t)+ \lambda_{i,m} g_m(t)+ p_i(t) +\oh_{t\to\infty}(1)
\\
&=&\sum_{j=1}^{m} \lambda_{i,j} g_j(t)+ p_i(t)+\oh_{t\to\infty}(1).
\end{eqnarray*}
For $i=k$ we have
$$
f_k(t)
\,=\, h_k(t)+ \sigma_i t^d+\oh_{t\to\infty}(1)\,=\, \sum_{j=1}^{m} \lambda_{k,j} g_j(t)+p_k(t)+\oh_{t\to\infty}(1).
$$
This shows that property \ref{itm_s3} is satisfied. 
Property \ref{itm_s2} holds by construction (and the fact that $g_m=h_k$ has degree $d$ but satisfies also $g_m(t)\prec t^d$) and property \ref{itm_1} holds because $h_k$ grows faster than any $\markfour{h}_i$ for $i=1,\ldots,k-1$, which implies that $g_m$ also grows faster than any $g_j$, $1\leq j \leq m-1$.

It remains to deal with the second case, i.e., the case when $\{h_1,\ldots,h_k\}\prec \{f_1,\ldots,f_k\}$. By the induction hypothesis, we can find $m\in\N$, $g_1,\ldots,g_m\in\S(h_1,\ldots,h_k)$, $p_1^*,\ldots,p_k^*\in\R[t]$, and $\lambda_{1,1},\ldots,\lambda_{k,m}\in\R$ such that $g_1,\ldots,g_m$ satisfy properties \ref{itm_s1}, \ref{itm_s2}, and 
for all $i\in\{1,\ldots,k\}$,
$$
\lim_{t\to\infty}\Bigg|h_i(t)-\sum_{j=1}^m \lambda_{i,j} g_j(t)-p_i^*(t)\Bigg|=0.
$$
Note that $\S(h_1,\ldots,h_k)=\S(f_1,\ldots,f_k)$ and hence $g_1,\ldots,g_m\in\S(f_1,\ldots,f_k)$. Moreover, if we take $p_i(t)\coloneqq p_i^*(t)+\sigma_i t^d$ then we get 
$$
\lim_{t\to\infty}\Bigg|f_i(t)-\sum_{j=1}^m \lambda_{i,j} g_j(t)-p_i(t)\Bigg|=0
$$
as desired.
\end{proof}

Define \[\S^*(f_1,\ldots,f_k)=\{\lambda_1 f_1^{(n_1)}+\ldots+\lambda_k f_k^{(n_k)}+p: \lambda_1,\ldots,\lambda_k\in\R,~p\in\R[t],~n_1,\ldots,n_k\in\N\cup\{0\}\}.\]

\begin{Lemma}
\label{lem_normal_form}
Let $\Hardy$ be a Hardy field and assume $f_1,\ldots,f_k\in \Hardy$ have polynomial growth. Then there exists $m\in\N$, $g_1,\ldots,g_m\in\S^*(f_1,\ldots,f_k)$, $p_1,\ldots,p_k\in\R[t]$, and $\lambda_{1,1},\ldots,\lambda_{k,m}\in\R$ with the following properties:
\begin{enumerate}
[label=(\arabic{enumi}),ref=(\arabic{enumi}),leftmargin=*]
\item
\label{itm_1}
$g_1(t)\prec \ldots\prec g_m(t)$;
\item
\label{itm_2}
for all $g\in\{g_1,\ldots,g_m\}$ either $g=0$ or there exists $\ell\in\N$ such that $t^{\ell-1}\prec g(t) \prec t^\ell$;
\item
\label{itm_3}
For all $g\in\{g_1,\ldots,g_m\}$ with $2\leq \deg(g)\leq \deg(g_m)$ we have $g'\in\{g_1,\ldots,g_m\}$ and for all $g\in\{g_1,\ldots,g_m\}$ with $1\leq \deg(g)\leq \deg(g_m)-1$ there exists an antiderivative of $g$ in $\{g_1,\ldots,g_m\}$;
\item
\label{itm_4}
for all $i\in\{1,\ldots,k\}$,
$$
\lim_{t\to\infty}\Bigg|f_i(t)-\sum_{j=1}^m \lambda_{i,j} g_j(t)-p_i(t)\Bigg|=0.
$$
\end{enumerate}
\end{Lemma}

In the proof of \cref{lem_simple_normal_form} we associated to every finite set of functions $h_1,\ldots,h_r\in\Hardy$ of polynomial growth a pair $(d,e)\in(\N\cup\{0\})\times \N$, called the \define{characteristic pair}, which gave rise to a partial ordering $\prec$ on the set of finite subsets of functions from $\Hardy$ of polynomial growth. For the proof of \cref{lem_normal_form} we shall use inductions on the same partial ordering.

\begin{proof}[Proof of \cref{lem_normal_form}]
If the characteristic pair associated to $\{f_1,\ldots,f_k\}$ is either of the from $(0,e)$ or $(1,e)$ for some $e\in\N$ then the conclusion of \cref{lem_normal_form} follows from \cref{lem_simple_normal_form}. Let us therefore assume that the characteristic pair associated to $\{f_1,\ldots,f_k\}$ is $(d,e)$ with $d\geq 2$ and that \cref{lem_normal_form} has already been proven for all $\markfour{f}_1,\ldots,\markfour{f}_{\markfour{k}}\in\Hardy$ satisfying $\{\markfour{f}_1,\ldots,\markfour{f}_{\markfour{k}}\}\prec \{f_1,\ldots,f_k\}$. 

By replacing $f_i$ with $-f_i$ if necessary, we can assume without loss of generality that all functions in $f_1,\ldots,f_k$ are eventually non-negative.
Also, since functions from a Hardy field can always be reordered according to their growth, we can relabel $f_1,\ldots,f_k$ such that $f_1 \ll f_2 \ll \ldots \ll f_k$.
Define $\eta_i\coloneqq \lim_{t\to\infty} f_{i}(t)/f_k(t)$. Set $\markfour{k}\coloneqq k$, define $\markfour{f}_i\coloneqq f_i-\eta_i f_k$ for all $i=1,\ldots,\markfour{k}-1$ and $\markfour{f}_{\markfour{k}}\coloneqq f_k'$.
It is straightforward to check that $\markfour{f}_1,\ldots,\markfour{f}_{\markfour{k}}\in\Hardy$ satisfies $\{\markfour{f}_1,\ldots,\markfour{f}_{\markfour{k}}\}\prec \{f_1,\ldots,f_k\}$. 
By the induction hypothesis, we can find $\markfour{m}\in\N$, $\markfour{g}_1,\ldots,\markfour{g}_{\markfour{m}}\in\Hardy$, $\markfour{p}_1,\ldots,\markfour{p}_{\markfour{k}}\in\R[t]$, and $\markfour{\lambda}_{1,1},\ldots,\markfour{\lambda}_{\markfour{k},\markfour{m}}\in\R$ such that $\markfour{g}_1,\ldots,\markfour{g}_{\markfour{m}}$ satisfy properties \ref{itm_1}, \ref{itm_2}, and \ref{itm_3}, and 
for all $i\in\{1,\ldots,\markfour{k}\}$ we have
$$
\lim_{t\to\infty}\Bigg|\markfour{f}_i(t)- \sum_{j=1}^{\markfour{m}} \markfour{\lambda}_{i,j} \markfour{g}_j(t)-\markfour{p}_i(t)\Bigg|=0.
$$
Next, set $g_j(t)\coloneqq \markfour{g}_j(t)$ for all $j\in\{1,\ldots,\markfour{m}\}$ and $p_i=\markfour{p_i}$ for all $i\in\{1,\ldots,\markfour{k}-1\}=\{1,\ldots,k-1\}$.
Let $j_0$ be the number in $\{1,\ldots,\markfour{m}\}$ uniquely determined by the property that $g_j$ has an antiderivative in $\{g_1,\ldots,g_{\markfour{m}}\}$ if $j\leq j_0$, and has no antiderivative in $\{g_1,\ldots,g_{\markfour{m}}\}$ if $j> j_0$. In other words, $j_0$ is the largest number in $\{1,\ldots,\markfour{m}\}$ for which $\deg(g_{j_0})=\deg(g_{\markfour{m}})-1$.
Then, for every $j\leq j_0$, let $\beta(j)$ be the number in $\{j+1,j+2,\ldots,\markfour{m}\}$ such that the derivative of $g_{\beta(j)}$ equals $g_j$. Define $m\coloneqq 2\markfour{m}-j_0$ and, for every $j\in\{\markfour{m}+1,\ldots,m\}$, let $g_j$ be any antiderivative of $g_{j_0+j-\markfour{m}}$.
Let $\mathcal{J}\coloneqq \{\beta(j): 1\leq j\leq j_0\}$ and take
\begin{equation*}
\lambda_{i,j}=
\begin{cases}
\markfour{\lambda}_{k,\beta^{-1}(j)},&\text{if}~i=k~\text{and}~j\in \mathcal{J}
\\
0,&\text{if}~i=k~\text{and}~ j\in \{1,\ldots,\markfour{m}\}\setminus\mathcal{J},
\\
\markfour{\lambda}_{k,j+j_0-\markfour{m}},&\text{if}~i=k~\text{and}~\markfour{m}< j\leq m
\\
\markfour{\lambda}_{i,j}+\eta_i\lambda_{k,j},&\text{if}~i<k~\text{and}~j\leq \markfour{m}
\\
\eta_i\lambda_{k,j},&\text{if}~i<k~\text{and}~\markfour{m}<j\leq m
\end{cases}.
\end{equation*}
Also, let $p_k$ be any polynomial with the property that it is an antiderivative of $\markfour{p}_k$.
Since $f_k'(t)=\markfour{f}_{\markfour{k}}(t)$ and $\markfour{k}=k$, we can write
$$
f_k'(t)
\,=\,
\sum_{j=1}^{\markfour{m}} \markfour{\lambda}_{\markfour{k},j} \,\markfour{g}_j(t)+\markfour{p}_{\markfour{k}}(t)+\oh_{t\to\infty}(1)
\,=\,
\sum_{j=1}^{\markfour{m}} \markfour{\lambda}_{k,j} \, g_j(t)+\markfour{p}_k(t)+\oh_{t\to\infty}(1).
$$
By integrating we get
\begin{equation}
\label{eqn_normal_form_1}
f_k(t)
\,=\,
\sum_{j=1}^{j_0} \markfour{\lambda}_{k,j} \, g_{\beta(j)}(t)+\sum_{j=j_0+1}^{\markfour{m}} \markfour{\lambda}_{k,j} \, g_{\markfour{m}+j-j_0}(t)+c+p_k(t)+\oh_{t\to\infty}(1),
\end{equation}
where $c$ is some real constant.
We can absorb $c$ into $p_k$, since $p_k$ was chosen to be an arbitrary antiderivative of $\markfour{p}_k$. 
Thus, \eqref{eqn_normal_form_1} becomes
\begin{eqnarray*}
f_k(t)
&=&
\sum_{j=1}^{j_0} \markfour{\lambda}_{k,j} \, g_{\beta(j)}(t)+\sum_{j=j_0+1}^{\markfour{m}} \markfour{\lambda}_{k,j} \, g_{\markfour{m}+j-j_0}(t)+p_k(t)+\oh_{t\to\infty}(1)
\\
&=&
\sum_{j=1}^{j_0} \markfour{\lambda}_{k,j} \, g_{\beta(j)}(t)+\sum_{j=\markfour{m}+1}^{m} \markfour{\lambda}_{k,j+j_0-\markfour{m}} \, g_j(t)+p_k(t)+\oh_{t\to\infty}(1)
\\
&=&
\sum_{j\in \mathcal{J}} \markfour{\lambda}_{k,\beta^{-1}(j)} \, g_{j}(t)+\sum_{j=\markfour{m}+1}^{m} \lambda_{k,j} \, g_j(t)+p_k(t)+\oh_{t\to\infty}(1)
\\
&=&
\sum_{j=1}^m \lambda_{k,j} \, g_{j}(t)+p_k(t)+\oh_{t\to\infty}(1).
\end{eqnarray*}
For $i\in\{1,\ldots,k-1\}$ we have
\begin{eqnarray*}
f_i(t)
&=&\markfour{f}_i(t) + \eta_i f_k(t) 
\\
&=& \sum_{j=1}^{\markfour{m}} \markfour{\lambda}_{i,j} \markfour{g}_j(t)+\markfour{p}_i(t)+ \eta_i f_k(t) +\oh_{t\to\infty}(1)
\\
&=&\sum_{j=1}^{\markfour{m}} \markfour{\lambda}_{i,j} g_j(t)+ \eta_i\sum_{j=1}^m \lambda_{k,j} \, g_{j}(t)+p_i(t) +\oh_{t\to\infty}(1)
\\
&=&\sum_{j=1}^{\markfour{m}} (\markfour{\lambda}_{i,j}+\eta_i\lambda_{k,j}) g_j(t)+ \sum_{j=\markfour{m}+1}^m \eta_i\lambda_{k,j} \, g_{j}(t) +p_i(t) +\oh_{t\to\infty}(1)
\\
&=&
\sum_{j=1}^m \lambda_{i,j} \, g_{j}(t)+p_i(t) +\oh_{t\to\infty}(1).
\end{eqnarray*}
This shows that property \ref{itm_4} holds. 
Since $\markfour{g}_1,\ldots,\markfour{g}_{\markfour{m}}$ satisfy properties \ref{itm_1} and $g_j=\markfour{g}_j$ for all $j\in\{1,\ldots,\markfour{m}\}$, we have $g_1(t)\prec \ldots\prec g_{\markfour{m}}(t)$. For $j_1<j_2\in \{\markfour{m}+1,\ldots, m\}$, it follows from L'H{\^o}pital's rule and $g_{j_1}'(t)\prec g_{j_2}'(t)$ that $g_{j_1}(t)\prec g_{j_2}(t)$. Also, $g_{\markfour{m}}(t)\prec g_{\markfour{m}+1}(t)$ because $\deg(g_{\markfour{m}+1})= \deg(g_{\markfour{m}})+1 $. 
This shows that Property \ref{itm_1} holds.
Properties \ref{itm_2}, and \ref{itm_3} are straightforward to derive from the definition of $g_1,\ldots,g_m$.
\end{proof}

\begin{Corollary}
\label{cor_finding_W}
Let $\Hardy$ be a Hardy field and assume $f_1,\ldots,f_k\in \Hardy$ have polynomial growth. Then there exists $W\in\Hardy$ with $1\prec W(t)\ll t$ such that $f_1,\ldots,f_k$ satisfy \ref{property_P_W}.
\end{Corollary}

\begin{proof}
Let $m\in\N$, $g_1,\ldots,g_m\in\S^*(f_1,\ldots,f_k)$, $p_1,\ldots,p_k\in\R[t]$, and $\lambda_{1,1},\ldots,\lambda_{k,m}\in\R$ be as guaranteed by \cref{lem_normal_form}.
According to \ref{itm_2}, for all $g\in\{g_1,\ldots,g_m\}$ either $g=0$ or there exists $\ell\in\N$ such that $t^{\ell-1}\prec g(t) \prec t^\ell$.
If $g=0$ for some $g\in\{g_1,\ldots,g_m\}$ then we must have $g=g_1$, since $g_1\prec\ldots\prec g_m$.
By discarding $g_1$ if it is the zero function, we can assume that for all $j\in\{1,\ldots,m\}$ there exists $\ell_j\in\N$ such that $t^{\ell_j-1}\prec g_j(t) \prec t^{\ell_j}$. Consider the functions $g_j(t)/t^{\ell_j-1}$, $j=1,\ldots,m$, and pick any function $W\in\Hardy$ with the property that $1\prec W(t)\ll t$ and $\log(W(t))\prec g_j(t)/t^{\ell_j-1}$ for all $j=1,\ldots,m$. Any $W$ with these properties is as desired.
\end{proof}

\begin{Corollary}
\label{cor_normal_form}
Let $\Hardy$ be a Hardy field, $V\in\Hardy$ with $1\prec V(t)\ll t$, and assume $f_1,\ldots,f_k\in \Hardy$ have the property that
\begin{enumerate}
[label=(0),ref=(0),leftmargin=*]
\item\label{itm_0}
for all $c_1,\ldots ,c_k\in\R$, $n_1,\ldots,n_k\in\N\cup\{0\}$, and $p\in\R[t]$ with $p(0)=0$ the function $f=c_1 f_1^{(n_1)}+\ldots+c_k f_k^{(n_k)}+p$ satisfies either $|f(t)|\ll 1$ or $t^{\ell-1}\log(V(t))\prec |f(t)| \ll t^\ell$ for some $\ell\in\N$.
\end{enumerate}
Then there exists $m\in\N$, $g_1,\ldots,g_m\in\Hardy$, $p_1,\ldots,p_k\in\R[t]$, and $\lambda_{1,1},\ldots,\lambda_{k,m}\in\R$ with the following properties:
\begin{enumerate}
[label=(\arabic{enumi}),ref=(\arabic{enumi}),leftmargin=*]
\item
$g_1(t)\prec \ldots\prec g_m(t)$;
\item
for all $g\in\{g_1,\ldots,g_m\}$ there exists $\ell\in\N$ such that $t^{\ell-1}\log(V(t))\prec g(t) \prec t^\ell$;
\item
for all $g\in\{g_1,\ldots,g_m\}$ with $\deg(g)\geq 2$ we have $g'\in\{g_1,\ldots,g_m\}$;
\item
for all $i\in\{1,\ldots,k\}$,
$$
\lim_{t\to\infty}\Bigg|f_i(t)-\sum_{j=1}^m \lambda_{i,j} g_j(t)-p_i(t)\Bigg|=0.
$$
\end{enumerate}
\end{Corollary}

\begin{proof}
This follows straightaway from \cref{lem_normal_form}.
\end{proof}

\begin{Remark}
It follows from \cref{lem_useful_hardy} that $t^{\ell-1}\log(V(t))\prec |f(t)| \ll t^\ell$ if and only if $\log(V(t))\prec |f^{(\ell-1)}(t)|$. Therefore, if $V(t)=t$ then Property \ref{itm_0} from \cref{cor_normal_form} is the same as \ref{property_P} from \cref{sec_intro}.
\end{Remark}

\subsection{Another variant of van der Corput's Lemma}

\begin{Theorem}
\label{prop_vdC}
Let $p_1, p_2,p_3,\ldots$ be a sequence of positive real numbers that is either non-decreasing or non-increasing.
Let $P_N\coloneqq \sum_{n=1}^N p_n$ and assume
$$
\lim_{N\to\infty} P_N\, =\, \infty
\qquad\text{and}\qquad
\lim_{N\to\infty}\frac{p_N}{P_N} \, =\, 0.
$$
Then for every $\epsilon>0$ there exists $\delta>0$ such that for every arithmetic function $f\colon\N\to\C$ bounded in modulus by $1$ and with the property that for every $h\in\N$ the limit
$$
A(h)~\coloneqq~\lim_{N\to\infty}
\frac{1}{P_N}\sum_{n=1}^N p_n f(n+h)\overline{f(n)}
$$
exists, we have
\begin{equation}
\label{eqn_prop_vdC_1}
\limsup_{H\to\infty}\left|\frac{1}{H}\sum_{h=1}^H A(h)\right|\leq \delta ~~\implies~~ \lim_{N\to\infty}\left|\frac{1}{P_N}\sum_{n=1}^N p_n f(n)\right|\leq\epsilon.
\end{equation}
\end{Theorem}

\begin{proof}
Assume $p_1,p_2,\ldots$ is non-decreasing, i.e., $p_n\geq p_{n-1}$ for all $n\in\N$.
For the case when $p_1,p_2,\ldots$ is non-increasing similar arguments apply.
We claim that for any bounded function $f\colon\N\to\C$ we have
\begin{equation}
\label{eqn_prop_vdC_2}
\lim_{N\to\infty}\left|\frac{1}{P_N}\sum_{n=1}^N p_n(f(n)-f(n+1))\right|~=~0.
\end{equation}

Let $f\colon\N\to\C$ be bounded. For the proof of \eqref{eqn_prop_vdC_2} we can assume without loss of generality that $\sup_{n\in\N}|f(n)|\leq 1$.
After an index-shift we obtain
\begin{eqnarray*}
\left|\frac{1}{P_N}\sum_{n=1}^N p_n(f(n)-f(n+1))\right|
&\leq & \frac{p_1}{P_N} + \frac{p_N}{P_N}+
\left|\frac{1}{P_N}\sum_{n=2}^N (p_n- p_{n-1})f(n)\right|
\end{eqnarray*}
Using $p_n\geq p_{n-1}$ we can estimate
$$
\left|\frac{1}{P_N}\sum_{n=2}^N (p_n- p_{n-1})f(n)\right|~\leq~ \frac{1}{P_N}\sum_{n=1}^{N-1} (p_{n}-p_{n-1}).
$$
The sum $\sum_{n=2}^N (p_n- p_{n-1})$ is telescoping and equals $p_N-p_1$. We are left with
$$
\left|\frac{1}{P_N}\sum_{n=1}^N p_n(f(n)-f(n+1))\right|~\leq~\frac{p_1}{P_N} + \frac{p_N}{P_N}+\frac{p_N}{P_N} - \frac{p_1}{P_N} = \frac{2p_N}{P_N}.
$$
The Claim now follows from the assumption $p_N/P_N\to 0$ as $N\to\infty$.

Next, fix any $f\colon\N\to\C$ bounded by $1$. Using \eqref{eqn_prop_vdC_2}, we get for all $H\in\N$ that
$$
\left|\frac{1}{P_N}\sum_{n=1}^N p_n f(n)\right|^2 = \left|\frac{1}{H}\sum_{h=1}^H\frac{1}{P_N}\sum_{n=1}^N p_n f(n+h)\right|^2\,+\,\oh_{N\to\infty}(1).
$$
By Jensen's inequality we have
\begin{eqnarray*}
\left|\frac{1}{H}\sum_{h=1}^H\frac{1}{P_N}\sum_{n=1}^N p_n f(n+h)\right|^2
& \leq &
\frac{1}{P_N}\sum_{n=1}^N p_n \left|\frac{1}{H}\sum_{h=1}^H f(n+h)\right|^2
\\
& = &
\frac{1}{P_N}\sum_{n=1}^N p_n \frac{1}{H^2}\sum_{h_1,h_2=1}^H f(n+h_1)\overline{f(n+h_2)}
\\
& = &
\frac{1}{P_N}\sum_{n=1}^N p_n \frac{1}{H^2}\sum_{h_1,h_2=1}^H f(n+h_1-h_2)\overline{f(n)} \,+\,\oh_{N\to\infty}(1)
\\
& = &
\frac{1}{P_N}\sum_{n=1}^N p_n \frac{1}{H}\sum_{h=-H}^H \frac{H-|h|}{H} f(n+h)\overline{f(n)} \,+\,\oh_{N\to\infty}(1).
\end{eqnarray*}
Moreover, we can write
\begin{eqnarray*}
\frac{1}{P_N}\sum_{n=1}^N p_n \frac{1}{H}\sum_{h=-H}^H \frac{H-|h|}{H} f(n+h)\overline{f(n)}
&=&
 \frac{1}{H}\sum_{h=-H}^H \frac{H-|h|}{H} A(h) \,+\,\oh_{N\to\infty}(1).
\end{eqnarray*}
In summary, we have shown that
$$
\left|\frac{1}{P_N}\sum_{n=1}^N p_n f(n)\right|^2 ~\leq~\frac{1}{H}\sum_{h=-H}^H \frac{H-|h|}{H} A(h)\,+\,\oh_{N\to\infty}(1).
$$
It is now not hard to show that for every $\epsilon>0$ there exists $\delta>0$ such that
\begin{equation}
\label{eqn_last}
\limsup_{H\to\infty}\left|\frac{1}{H}\sum_{h=1}^H A(h)\right|\leq \delta ~~\implies~~ \limsup_{H\to\infty}\left|\frac{1}{H}\sum_{h=-H}^H \frac{H-|h|}{H} A(h)\right|\leq\epsilon,
\end{equation}
from which \eqref{eqn_prop_vdC_1} follows.
Indeed, $\limsup_{H\to\infty}\left|\frac{1}{H}\sum_{h=1}^H A(h)\right|\leq \delta$ implies for all $K\in\N$ and $r\in\{0,1,\ldots,K-1\}$ that
\begin{align*}
\limsup_{H\to\infty}\left|\frac{1}{H}\sum_{\frac{rH}{K}\leq |h|\leq \frac{(r+1)H}{K}} A(h)\right|
&=\limsup_{H\to\infty}
\left|\frac{1}{H}\sum_{ |h|\leq \frac{(r+1)H}{K}} A(h) ~-~ \frac{1}{H}\sum_{ |h|\leq \frac{rH}{K}} A(h)\right|
\\
&\leq\limsup_{H\to\infty}
\left|\frac{1}{H}\sum_{ |h|\leq \frac{(r+1)H}{K}} A(h)\right| + \limsup_{H\to\infty}\left|\frac{1}{H}\sum_{ |h|\leq \frac{rH}{K}} A(h)\right|
\\
&\leq\limsup_{H\to\infty}
\left|\frac{2}{H}\sum_{ 1\leq h\leq \frac{(r+1)H}{K}} A(h)\right| + \limsup_{H\to\infty}\left|\frac{2}{H}\sum_{ 1\leq h\leq \frac{rH}{K}} A(h)\right|
\\
&\leq 4\delta.
\end{align*}
Hence
\begin{align*}
\limsup_{H\to\infty}\left|\frac{1}{H}\sum_{h=-H}^H \frac{H-|h|}{H} A(h)\right|
&\leq K\left(\max_{0\leq r\leq K} \left|\frac{1}{H}\sum_{\frac{rH}{K}\leq h\leq \frac{(r+1)H}{K}}\frac{H-|h|}{H} A(h)\right|\right)
\\
&= K\left(\max_{0\leq r\leq K} \left|\frac{1}{H}\sum_{\frac{rH}{K}\leq h\leq \frac{(r+1)H}{K}}\big(1-r/K\big) A(h)\right| + 1/K^2\right)
\\
&\leq 4K\delta + 1/K.
\end{align*}
Choosing $K=\lfloor \delta^{-1/2}\rfloor$ we obtain equation \eqref{eqn_last} with $\epsilon= 10 \delta^{1/2}$.
\end{proof}

\bibliography{mynewlibrary}




\bigskip
{\footnotesize
\noindent
Florian K.\ Richter\\
\textsc{{\'E}cole Polytechnique F{\'e}d{\'e}rale de Lausanne}\par\nopagebreak
\noindent
\href{mailto:f.richter@epfl.ch}
{\texttt{f.richter@epfl.ch}}

\end{document}